\documentclass[10pt]{article}

\usepackage{url}
\usepackage{mathtools}
\usepackage{amssymb}
\usepackage{amsthm}
\usepackage{empheq}
\usepackage{latexsym}
\usepackage{enumitem}
\usepackage{eurosym}
\usepackage{dsfont}
\usepackage{appendix}
\usepackage{color} 
\usepackage[unicode]{hyperref}
\usepackage{frcursive}
\usepackage[utf8]{inputenc}
\usepackage[T1]{fontenc}
\usepackage{geometry}
\usepackage{multirow}
\usepackage{todonotes}
\usepackage{lmodern}
\usepackage{anyfontsize}
\usepackage{stmaryrd}
\usepackage{natbib}
\usepackage{cleveref}

\setcitestyle{numbers,open={[},close={]}}

\definecolor{green}{rgb}{0,0.5,0}
\hypersetup{colorlinks, linkcolor={red},citecolor={green}, urlcolor={blue}}
			
\makeatletter \@addtoreset{equation}{section}

\newtheorem{theorem}{Theorem}[section]
\newtheorem{assumption}[theorem]{Assumption}
\newtheorem{corollary}[theorem]{Corollary}
\newtheorem{example}[theorem]{Example}

\newtheorem{lemma}[theorem]{Lemma}
\newtheorem{proposition}[theorem]{Proposition}

\newtheorem{definition}[theorem]{Definition}
\newtheorem{remark}[theorem]{Remark}

\DeclareUnicodeCharacter{014D}{\=o}
\def \B{\mathbb{B}}

\def \D{\mathbb{D}}
\def \E{\mathbb{E}}
\def \F{\mathbb{F}}
\def \G{\mathbb{G}}

\def \L{\mathbb{L}}

\def \N{\mathbb{N}}

\def \P{\mathbb{P}}
\def \Q{\mathbb{Q}}
\def \R{\mathbb{R}}
\def \S{\mathbb{S}}

\def \X{\mathbb{X}}

\def\Ac{{\cal A}}
\def\Bc{{\cal B}}
\def\Cc{{\cal C}}

\def\Fc{{\cal F}}
\def\Gc{{\cal G}}

\def\Lc{{\cal L}}

\def\Pc{{\cal P}}

\def\Rc{{\cal R}}
\def\Sc{{\cal S}}

\def\Uc{{\cal U}}
\def\Vc{{\cal V}}
\def\Wc{{\cal W}}

\def\Ah{{\widehat A}}

\def\Fb{{\bar F}}
\def\Gb{{\overline \G}}
\def\Gcb{\overline \Gc}

\def\Pb{{\overline \P}}
\def\Qb{{\overline \Q}}

\def\Sb{{\overline S}}

\def\Ub{{\overline U}}

\def\x{\times}
\def\eps{\varepsilon}
\def\Om{\Omega}
\def\Omh{\widetilde{\Omega}}
\def\Omh{\widehat{\Omega}}
\def\om{\omega}

\def \w{\mathsf{w}}
\def \wb{ \bar \w}
\def\Omb{\overline{\Om}}
\def\omb{\bar \om}
\def\Fb{\overline{\F}}
\def\Gcb{\overline{\Gc}}
\def\Fcb{\overline{\Fc}}
\def\Bt{\widetilde{B}}
\def\Gct{\widetilde{\Gc}}
\def\Fct{\widetilde{\Fc}}
\def\Ft{\widetilde{\F}}
\def\Gt{\widetilde{\G}}

\def \bb{\mathbf{b}}
\def \xb{\mathbf{x}}
\def \yb{\mathbf{y}}
\def \wb{\mathbf{w}}
\def \mub{\overline{\mu}}

\def \muh{\widehat{\mu}}

\def \nub{\bar{\nu}}
\def \nuh{\widehat{\nu}}
\def\Pch{\widehat \Pc}

\def\alphah{\widehat \alpha}
\def\Xh{\widehat X}
\def\Wh{\widehat W}
\def\Bh{\widehat B}

\def\alphab{\overline \alpha}
\def\Pcb{\overline \Pc}
\def\Omt{\widetilde{\Om}}
\def\Pt{\widetilde{\P}}

\def\taub{\bar \tau}

\title{McKean--Vlasov optimal control: the dynamic programming principle\footnote{Mao Fabrice Djete gratefully acknowledges support from the r\'egion \^Ile--de--France. This work also benefited from support of the ANR project PACMAN ANR--16--CE05--0027.}}

\author{
	Mao Fabrice {\sc Djete}\footnote{Universit\'e Paris--Dauphine, PSL University, CNRS, CEREMADE, 75016 Paris, France, djete@ceremade.dauphine.fr} 
	\and 
	Dylan {\sc Possama\"{i}} \footnote{Columbia University, Industrial Engineering \& Operations Research, 500 W 120th Street, New York, NY, 10027, dp2917@columbia.edu}
	\and 
	Xiaolu {\sc Tan}\footnote{Department of Mathematics, The Chinese University of Hong Kong. xiaolu.tan@cuhk.edu.hk}
}
	\date{\today}

\begin{document}

\maketitle
 
\begin{abstract}
We study the McKean--Vlasov optimal control problem with common noise in various formulations, namely the strong and weak formulation, as well as the Markovian and non--Markovian formulations, and allowing for the law of the control process to appear in the state dynamics. By interpreting the controls as probability measures on an appropriate canonical space with two filtrations, we then develop the classical measurable selection, conditioning and concatenation arguments in this new context, and establish the dynamic programming principle under general conditions.
\end{abstract}

\section{Introduction}

	We propose in this paper to study the problem of optimal control of mean--field stochastic differential equations, also called Mckean--Vlasov stochastic differential equations in the literature. This problem is a stochastic control problem where the state process is governed by a stochastic differential equation (SDE for short), which has coefficients depending on the current time, the paths of the state process, but also its distribution (or conditional distribution in the case with common noise). Similarly, the reward functionals are allowed to be impacted by the distribution of the state process.

\medskip
The pioneering work on McKean--Vlasov equations is due to \citeauthor*{mckean1969propagation} \cite{mckean1969propagation} and \citeauthor*{kac1956foundations} \cite{kac1956foundations}, who were interested in studying uncontrolled SDEs, and in establishing general propagation of chaos results. Let us also mention the illuminating notes of \citeauthor*{snitzman1991topics} \cite{snitzman1991topics}, which give a precise and pedagogical insight into this specific equation. Though many authors have worked on this equation following these initial papers, there has been a drastic surge of interest in the topic in the past decade, due to the connection that it shares with the so--called mean--field game (MFG for short) theory, introduced independently and simultaneously on the one hand by \citeauthor*{lasry2006jeux} in \cite{lasry2006jeux,lasry2006jeux2,lasry2007mean} and on the other hand by \citeauthor*{huang2003individual} \cite{huang2003individual,huang2006large,huang2007invariance,huang2007large,huang2007nash}. Indeed, the McKean--Vlasov equation naturally appears when one tries to describe the behaviour of many agents, which interact through the empirical distribution of their states, and who seek a Nash equilibrium (this is the competitive equilibrium case, leading to the MFG theory, see \citeauthor*{cardaliaguet2015master} \cite{cardaliaguet2015master}), or a Pareto equilibrium (this is the cooperative equilibrium case, associated to the optimal control of Mckean--Vlasov stochastic equations, see \citeauthor*{lacker2017limit} \cite{lacker2017limit}). Though related, these two problems have some subtle differences which were investigated thoroughly by \citeauthor*{carmona2013control} \cite{carmona2013control}.

\medskip
Our interest in this paper, as already mentioned, is in the optimal control of McKean--Vlasov stochastic equations, and more precisely in the rigorous establishment of the dynamic programming principle (DPP for short), under conditions as general as possible. The optimal control of McKean--Vlasov dynamics is a rather recent problem in the literature. The first approach to tackle the problem relied on the use of the celebrated Pontryagin stochastic maximum principle. This strategy allows to derive necessary and/or sufficient conditions characterising the optimal solution of the control problem, through a pair of processes $(Y,Z)$ satisfying a backward stochastic differential equation (BSDE for short), also called adjoint equation in this case, coupled with a forward SDE, corresponding to the optimal path. \citeauthor*{andersson2011maximum} \cite{andersson2011maximum} and \citeauthor*{buckdahn2011general} \cite{buckdahn2011general} use this approach for a specific case of optimal control of a McKean--Vlasov equation, corresponding to the case where the coefficients of the equation and the reward functions only depend on some moments of the law. We refer to \citeauthor*{carmona2015forward} \cite{carmona2015forward} for an analysis in a more general context thanks to the notion of differentiability in the space of probability measure introduced by \citeauthor*{lions2007theorie} in his Coll\`ege de France course \cite{lions2007theorie} (see also the lecture notes of \citeauthor*{cardaliaguet2010notes} \cite{cardaliaguet2010notes}). Related results were also obtained by \citeauthor*{acciaio2018generalized} \cite{acciaio2018generalized} for so--called generalised McKean--Vlasov control problems involving the law of the controls, and were a link with causal optimal transport was also highlighted. The stochastic maximum principle approach was also used in a context involving conditional distributions in \citeauthor*{buckdahn2017mean} \cite{buckdahn2017mean} and \citeauthor*{carmona2014probabilistic} \cite{carmona2014probabilistic}. Related results have been obtained in the context of a relaxed formulation of the control problem, allowing in addition to obtain existence results by \citeauthor*{chala2014relaxed} \cite{chala2014relaxed}, which were then revisited by \citeauthor*{lacker2017limit} \cite{lacker2017limit}, and \citeauthor*{bahlali2017existence} \cite{bahlali2014existence,bahlali2017existence,bahlali2018relaxed}
	
\medskip
Readers familiar with the classical theory of stochastic control know that another popular approach to the problem is to use Bellman's optimality principle, to obtain the so--called dynamic programming principle. In a nutshell, the idea behind the DPP is that the global optimisation problem can be solved by a recursive resolution of successive local optimisation problems. This fact is an intuitive result, which is often used as some sort of meta--theorem, but is not so easy to prove rigorously in general. Note also that, in contrast to the Pontryagin maximum principle approach, this approach in general requires fewer assumptions, though it can be applied in less situations. Notwithstanding these advantages, the DPP approach has long been unexplored for the control of McKean--Vlasov equations. One of the main reasons is actually a very bleak one for us: due to the non--linear dependency with respect to the law of process, the problem is actually a time inconsistent control problem (like the classical mean--variance optimisation problem in finance, see the recent papers by \citeauthor*{bjork2014theory} \cite{bjork2014theory},  \citeauthor*{bjork2017time} \cite{bjork2017time}, and \cite{hernandez2019me} for a more thorough discussion of this topic), and Bellman's optimality principle does not hold in this case. However, though the problem itself is time--inconsistent, one can recover some form of the DPP by extending the state space of the problem. This was first achieved by \citeauthor*{lauriere2014dynamic} \cite{lauriere2014dynamic}, and later by \citeauthor*{bensoussan2015master} \cite{bensoussan2013mean,bensoussan2015master,bensoussan2017interpretation}, who assumed the existence at all times of a density for the marginal distribution of the state process, and reformulated the problem as a deterministic density control problem, with a family of deterministic control terms. Under this reformulation, they managed to prove a DPP and deduce a dynamic programming equation in the space of density functions. Following similar ideas, but without the assumptions of the existence of density, and allowing the coefficients and reward functions to not only depend on the distribution of the state, but to the joint distribution of the state and the control, \citeauthor*{pham2018bellman} \cite{pham2018bellman} also deduced a DPP by looking at a set of closed loop (or feedback) controls, in a non common noise context. They then extended this strategy to a common noise setting (where the control process is adapted to common noise filtration) in \cite{pham2016dynamic}. Specialisations to linear--quadratic settings were also explored by \citeauthor*{pham2016linear} \cite{pham2016linear}, \citeauthor*{li2016mean} \cite{li2016mean}, \citeauthor*{li2019linear} \cite{li2019linear}, \citeauthor*{huang2015linear} \cite{huang2015linear}, \citeauthor*{yong2013linear} \cite{yong2013linear}, and \citeauthor*{basei2019weak} \cite{basei2019weak}.

\medskip
Concerning the (general) DPP approach to solve problems involving McKean--Vlasov stochastic equations, let us mention \citeauthor*{bouchard2017quenched} \cite{bouchard2017quenched}, who study a stochastic target problem for McKean--Vlasov stochastic equations. Roughly speaking, this means that they are interested in optimally controlling a McKean--Vlasov equation over the interval $[0,T]$, under the target constraint that the marginal law of the controlled process at time $T$ belongs to some Borel subset of the space of probability measures. They establish a general geometric dynamic programming, thus extending the seminal results of \citeauthor*{soner2002dynamic} \cite{soner2002dynamic} corresponding to the non--McKean--Vlasov case. Another important contribution is due to \citeauthor*{bayraktar2016randomized} \cite{bayraktar2016randomized}. There, the authors build upon the results of \citeauthor*{fuhrman2015randomized} \cite{fuhrman2015randomized} and \citeauthor*{bandini2015randomization} \cite{bandini2015randomization} to obtain a randomised control problem by controlling the intensity a Poisson random measure. This enables them to allow for general open--loop controls, unlike in \cite{pham2016dynamic,pham2018bellman}, though the setting is still Markovian and does not permit to consider common noise. 
	
\medskip
Our approach to obtaining the DPP is very different. One common drawback of all the results we mentioned above, is that they generically require some Markovian\footnote{An exception is the work of \citeauthor*{djehiche2018optimal} \cite{djehiche2018optimal}, which considers optimal control (and also a zero--sum game) of a non--Markovian McKean--Vlasov equation, and obtains both a characterisation of the value function and the optimal control using BSDE techniques, reminiscent of the classical results of \citeauthor*{hamadene1995backward} \cite{hamadene1995backward} and \citeauthor*{el1995dynamic} \cite{el1991programmation,el1995dynamic} for the non--McKean--Vlasov case. However, their approach does not allow for common noise, and is limited to control on the drift of the state process only.} property of the system or its distribution, as well as strong regularity assumptions on coefficient and reward functions considered. This should appear somehow surprising to people familiar with the classical DDP theory. Indeed, for stochastic control problems, it is possible to use measurable selection arguments to obtain the DPP, in settings requiring nothing beyond mild measurability assumptions. As a rule of thumb, one needs two essential ingredients to prove the dynamic programming principle: first ensuring the stability of the controls with respect to conditioning and concatenation, and second the measurability of the associated value function. The use of measurable selection argument makes it possible to provide an adequate framework for verifying the conditioning, the concatenation and the measurability requirements of the associated value function without strong assumptions. This technique was followed by \citeauthor*{dellacherie1985quelques} \cite{dellacherie1985quelques}, by \citeauthor*{bertsekas1978stochastic} in \cite{bertsekas1978stochastic,shreve1978alternative,shreve1978dynamic,shreve1979universally}, and by \citeauthor*{shreve1979resolution} \cite{shreve1977dynamic,shreve1978probability,shreve1979resolution} for discrete--time stochastic control problems. Later, \citeauthor*{el1987compactification} in \cite{el1987compactification} presented a framework for stochastic control problem in continuous time (accommodating general Markovian processes). Thanks to the notion of relaxed control, that is to say the interpretation of a control as a probability measure on some canonical space, and thanks to the use of the notion of martingale problems, they proved a DPP by simple and clear arguments. \citeauthor*{karoui2013capacities} \cite{karoui2013capacities,karoui2013capacities2} extended this approach to the non--Markovian case. Similar results were obtained by several authors, among which we mention \citeauthor*{nutz2012superhedging} \cite{nutz2012superhedging}, \citeauthor*{neufeld2013superreplication} \cite{neufeld2013superreplication,neufeld2016nonlinear}, \citeauthor*{nutz2013constructing} \cite{nutz2013constructing}, \citeauthor*{zitkovic2014dynamic} \cite{zitkovic2014dynamic}, and \citeauthor*{possamai2015stochastic} \cite{possamai2015stochastic}.

\medskip
Following the framework in \cite{karoui2013capacities, el1987compactification}, we develop in this paper a general analysis based upon the measurable selection argument for the non--Markovian optimal control of McKean--Vlasov equations with common noise. In particular, we investigate the case where the drift and diffusion coefficients, as well as the reward functions, are allowed to depend on the joint conditional distribution of the path of the state process and of the control, see \cite{pham2018bellman} for the case of the joint distribution of the state process and of feedback controls (see also \citeauthor*{yong2013linear} \cite{yong2013linear} for a more specific situation) in a non--common noise case.

\medskip	
Motivated by the notion of weak solution of classical SDEs, and similarly to the ideas used by \citeauthor*{karoui2013capacities} \cite{karoui2013capacities}, and \citeauthor*{carmona2014mean} \cite{carmona2014mean} in a mean--field games context, our first task is to provide an appropriate `relaxation` of the problem. We therefore introduce a notion of weak solution of controlled McKean--Vlasov equation with common noise. Notice that this is by no means a straightforward task. In standard McKean--Vlasov stochastic control problems, the controls (open loop in that case) are adapted with respect to the filtration generated by both the Brownian motion $(W,B)$ (with $B$ being the common noise) and the initial random variable $\xi$ (serving as an initial condition for the problem). Then, the conditional distributions considered are associated to the filtration of $B$, in other words the `common noise` filtration, that is, $\Lc(X_{t \wedge \cdot}, \alpha_t | B)$, where $X$ is the state and $\alpha$ the control. We call this the strong formulation. 
	{\color{black}  The strong formulation does not enjoy a good stability condition. To see this, it is enough to notice that the conditional distribution is not continuous with respect to the joint distribution $\big($for instance the function $\Lc(X_t,B)\longmapsto \E \big[ \big| \E \big[ X_t \big| B\big] \big|^2 \big]$ is not continuous$\big)$.
	To overcome this difficulty, we introduce a notion of weak solution by considering a more general filtration $\F$ describing the adaptability of the controls, and an extended common noise filtration $\G$ as in \cite{carmona2014mean}.
	Nevertheless, more conditions on $\F$ and $\G$ are needed to ensure that the formulation remains first compatible with the notion of strong solutions,
	then enjoys good stability properties for fixed control processes,
	and finally ensures that weak controls can be approximated sufficiently well by strong controls.
}

\medskip
	With the help of this notion, we can then provide a weak formulation for McKean--Vlasov control problems with common noise. 
	By interpreting controls as probability measures on an appropriate canonical space, 
	and using measurable selection arguments as in \cite{karoui2013capacities, karoui2013capacities2},
	we then move on to prove the universal measurability of the associated value function,
	and derive the stability of controls with respect to conditioning and concatenation, and finally deduce the DPP for the weak formulation under very general assumptions.
	Our next result addresses the DPP for the classical strong formulation.
	Using the DPP in weak formulation, and by adding standard Lipschitz conditions on the drift and diffusion coefficients, as in \cite{pham2016dynamic}, but without any regularity assumptions on reward functions, and in a non--Markovian context, we obtain the DPP for the strong formulation of McKean--Vlasov control problems with common noise, where the control is adapted to the `common noise` filtration ($B$ in this case of strong formulation). 
	Also, for general strong formulation, where the control is adapted to both $\xi,$ $W$ and $B$, we obtain the DPP under some additional regularity conditions on the reward functions.
	These regularity conditions may seem unexpected at first sight, but they seem unavoidable due to the non--linear dependency of the drift and volatility coefficients with respect to the conditional distribution of $X$ (see \Cref{rem:ddp_under_regularity} for a more thorough discussion).
{\color{black}
	Finally, the DPP results in the general non--Markovian context induces the same results in the Markovian one.
}

	\medskip

	{\color{black}
		We stress again that our results and techniques are not easy extensions of those in the existing literature, 
		although the very starting ideas may seem so.
		Our DPP results are much more general than those in the literature, and, as far as we can tell, are the most general.
		To interpret a control as a probability measure on the canonical space,
		one needs to formulate an adequate notion of weak solutions which should enjoy more stability properties than the strong solutions, and at the same time be able to approximate these solutions by strong ones.
		Finally, because of the presence of two filtrations (a general filtration and a common noise filtration) on the canonical space and the implicit conditions on the two filtrations,
		it become much more delicate to develop the classical measurable selection, conditioning and concatenation arguments in this new context.
	}

	\medskip

	The rest of the paper is organised as follows. After recalling briefly some notations and introducing the probabilistic structure to give an adequate and precise definition of the tools that are used throughout the paper, we introduce in \Cref{sec:Formulation} several notions of weak and strong formulation (in a fixed probability space or in the canonical space) for the McKean--Vlasov stochastic control problem with common noise in a non--Markovian framework, and prove some equivalence results. Next, in \Cref{sec:DPP}, we present the main result of this paper, the DPP for three formulations: weak formulation, strong formulation, and a {\color{black} $\B$--strong formulation} where the control is adapted with respect to the `common noise` filtration. We first provide all our results in the non--Markovian setting, and then in a Markovian framework.
	Finally, \Cref{sec:4} is devoted to the proof of our main results.

\medskip

  	{\bf Notations}.
	$(i)$
	Given a metric space $(E,\Delta)$ and $p \ge 0$, we denote by $\Pc(E)$ the collection of all Borel probability measures on $E$,
	and by $\Pc_p(E)$ the subset of Borel probability measures $\mu$ 
	such that $\int_E \Delta(e, e_0)^p  \mu(\mathrm{d}e) < \infty$ for some $e_0 \in E$.
	When $p \ge 1$, the space $\Pc_p(E)$ is equipped with the Wasserstein distance $\Wc_p$ defined by
	\[
		\Wc_p(\mu , \mu^\prime) 
		:=
		\bigg(
			\inf_{\lambda \in \Lambda(\mu, \mu^\prime)}  \int_{E \times E} \Delta(e, e^\prime)^p \lambda(\mathrm{d}e,\mathrm{d}e^\prime) 
		\bigg)^{1/p},
	\]
	where $\Lambda(\mu, \mu^\prime)$ is the collection of all Borel probability measures $\lambda$ on $E \times E$ 
	such that $\lambda(\mathrm{d}e, E) = \mu(\mathrm{d}e)$ and $\lambda(E, \mathrm{d}e^\prime) = \mu^\prime(\mathrm{d}e^\prime)$. 
	When $E$ is a Polish space, $(\Pc_p(E), \Wc_p)$ is a Polish space (see \citeauthor*{villani2008optimal} \cite[Theorem 6.16]{villani2008optimal}). 
	Given another metric space $(E^\prime,\Delta^\prime)$, we denote by $\mu \otimes \mu^\prime \in \Pc(E \x E')$ the product probability of any $(\mu,\mu^\prime) \in \Pc(E) \x \Pc(E^\prime)$.

\medskip
	
	\noindent $(ii)$
	Given a measurable space $(\Om,\Fc)$, we denote by $\Pc(\Om)$ the collection of all probability measures on $(\Om, \Fc)$.
	For any probability measure $\P \in \Pc(\Om)$, we denote by $\Fc^\P$ the $\P$--completion of the $\sigma$--field $\Fc$,
	and by $\Fc^U := \bigcap_{\P \in \Pc(\Om)} \Fc^\P$ the universal completion of $\Fc$.
	Let $\xi: \Om \longrightarrow \R \cup \{ -\infty,+\infty\}$ be a random variable and $\P \in \Pc(\Om)$, 
	we define
	\[
		\E^{\P}\big[ \xi \big] :=\E^{\P}[\xi_+] -\E^{\P}[\xi_-],
		\;\mbox{where}\;\xi_+ := \xi \vee 0,\; \xi_- := (-\xi) \vee 0,
		\;\mbox{with the convention}\; \infty - \infty:=-\infty.
	\]
We also use the following notation to denote the expectation of $\xi$ under $\P$
		\[
			\E^{\P}[\xi] 
			= \langle \P,\xi \rangle =\langle \xi, \P \rangle.
		\]
	When $\Om$ is a Polish space, a subset $A \subseteq \Om$ is called an analytic set if there is another Polish space $E$, and a Borel subset $B \subseteq \Om \x E$ such that $A = \{ \om \in \Om : \exists e \in E,\; (\om, e) \in B \}$.
	A function $f: \Om \longrightarrow \R \cup \{- \infty, \infty \}$ is called upper semi--analytic (u.s.a. for short) if $\{\om \in \Om : f(\om) > c \}$ is analytic for every $c \in \R$. Any upper semi--analytic function is universally measurable (see e.g. \cite[Chapter 7]{bertsekas1978stochastic}).
	
	\medskip

	\noindent $(iii)$	
	Let $\Om$ be a metric space, $\Fc$ its Borel $\sigma$--field and $\Gc \subset \Fc$ be a sub--$\sigma$--field which is countably generated.
	Following \cite{stroock2007multidimensional}, we say that $(\P^{\Gc}_{\om})_{\om \in \Om}$ is a family of r.c.p.d. (regular conditional probability distributions) of $\P$ knowing $\Gc$ if it satisfies 
	\begin{itemize}
		\item the map $\om \longmapsto \P^{\Gc}_{\om}$ is $\Gc-$measurable,
		and for all $A \in \Fc$ and $B \in \Gc$, one has $\P [A \cap B]=\int_B \P^{\Gc}_{\om}[A] \P(\mathrm{d}\om)$;
		\item $\P^{\Gc}_{\om} \big[ [\om]_{\Gc} \big] = 1$  for all $\om \in \Om$, 
		where $[\om]_{\Gc} := \bigcap \big\{ A\in\Fc:A \in \Gc\; \mbox{and}\; \om \in A\big\}$.
	\end{itemize}
	Let $(\Om, \Fc, \P, \G = (\Gc_t)_{t \in [0,T]})$ be a filtered probability space, $\Gc \subset \Fc$ be a sub--$\sigma$--field, and $E$ a metric space. Then given a random element $\xi: \Om \longrightarrow E$, we use both the notations $\Lc^{\P}( \xi | \Gc)(\om)$ and $\P^{\Gc}_{\om} \circ (\xi)^{-1}$ to denote the conditional distribution of $\xi$ knowing $\Gc$ under $\P$.
	Moreover, given a measurable process $X: [0,T] \x \Om \longrightarrow E$, 
	we can always define $\mu_t := \Lc^{\P}(X_t | \Gc_t)$ to be a $\Pc(E)$--valued $\G$--optional process (see for instance \Cref{lemm:Cond_Law}).

	\medskip
	
	\noindent $(iv)$
	Let $(E,\Delta)$ and $(E',\Delta')$ be two Polish spaces, we shall refer to $C_b(E,E')$ to designate the set of continuous functions $f:E \to E'$ such that $\sup_{e \in E} \Delta'(f(e),e'_0) < \infty$ for some $e'_0 \in E'$.
	Let $\N$ denote the set of non--negative integers.
	For $(k,n) \in \N^2$, we denote by $C^n_b(\R^k;\R)$ the set of maps bounded continuous maps $f: \R^k \longrightarrow \R$ possessing bounded continuous derivatives up to order $n$, 
	and by $\partial_i f$ (resp $\partial^2_{i,j}$) the partial derivative (resp. crossed second partial derivative) with respect to $x^i$ (resp $(x^i,x^j)$) for $(i,j) \in \{1, \dots, k\} \x \{ 1, \dots, k \}$, and $x:=(x^1,...,x^k) \in \R^k$.
	Given $(k,q)\in\N \times\N$, we denote by $\S^{k \x q}$ the collection of all $k \x q$--dimensional matrices with real entries, equipped with the standard Euclidean norm that we will denote $|\cdot|$ when there are no ambiguity. Let us denote by $\mathbf{0}_k,$ $\mathbf{0}_{k \times q}$ and $\mathrm{I}_k$ the null matrix of dimension $k \x k,$ the null matrix of dimension $k \x q$, and the identity matrix of dimension $k \x k$.
	Let $T > 0$, and $(E, \Delta)$ be a Polish space, we denote by $C([0,T], E)$ the space of all continuous functions on $[0,T]$ taking values in $E$, which is also a Polish space under the uniform convergence topology.
	When $E=\R^n$ for some $n\in\N$, and $\Delta$ is the usual Euclidean norm on $\R^n$, we simply write $\Cc^n := C([0,T], \R^n)$ and denote by $\|\cdot\|$ the uniform norm.
	We also denote by $\Cc^{n}_{s,t} :=C([s,t];\R^n)$ the space of all $\R^n$--valued continuous functions on $[s,t]$, for $0 \le s \le t \le T$.
	When $n = 0$, the space $\Cc^n$ and $\Cc^n_{s,t}$ both degenerate to a singleton.

	\medskip

	\noindent $(v)$
	Throughout the paper, we fix a constant $p \ge 0$,
	a nonempty Polish space $(U,\rho)$ and a point $u_0 \in U$.
	Notice that a Polish space is always isomorphic to a {\color{black}Borel subset of $[0,1]$. }
	Let us denote by $\pi$ such one (isomorphic) bijection between $U$ and {\color{black}$\pi(U) \subseteq [0,1]$}.
	We further extend the definition of $\pi^{-1}$ to $\R \cup \{-\infty, \infty\}$ by setting $\pi^{-1}(x) := \partial$ for all $x \notin \pi(U)$ and let $\Ub := U \cup \{\partial\}$, where $\partial$ is the usual cemetery point.
	Let $\nu \in \Pc(\Cc^k)$ (resp. $\nub \in \Pc(\Cc^k \x U)$) be a Borel probability measure on 
	the canonical space $\Cc^k$ (resp. $\Cc^k \x U$) equipped with the canonical process $X$ (resp. $(X, \alpha)$).
	We denote, for each $t \in [0,T]$,
	\[
		\nu(t) := \nu \circ X_{t \wedge \cdot}^{-1}
		\; \big(\mbox{resp.}\; \nub(t) := \nub \circ (X_{t \wedge \cdot}, \alpha)^{-1} \big).
	\]

\section{Weak and strong formulations of the McKean--Vlasov control problem}
\label{sec:Formulation}

	The main objective of this paper is to study the following (non--Markovian) McKean--Vlasov control problem, 
	in both strong and weak formulations, of the form
	\[
		`\sup_{\alpha} 
		\E \bigg[ 
			\int_0^T L \big(t, X^{\alpha}_{t \wedge \cdot}, \Lc\big(X^{\alpha}_{t \wedge \cdot}, \alpha_t \big| \Gc_t \big), \alpha_t \big) \mathrm{d}t 
			+ 
			g \big( X^{\alpha}_{\cdot}, \Lc \big(X^{\alpha}_{\cdot} \big| \Gc_T\big) \big) \bigg],`
	\]
	where $\G := (\Gc_t)_{0 \le t \le T}$ is a filtration modelling the commun noise, supporting a Brownian motion $B$, 
	$\Lc\big(X^{\alpha}_{t \wedge \cdot}, \alpha_t \big| \Gc_t \big)$ denotes the joint conditional distribution of $(X^{\alpha}_{t \wedge \cdot}, \alpha_t )$ knowing $\Gc_t$,
	and $(X^{\alpha}_t)_{t \in [0,T]}$ is a McKean--Vlasov type process, controlled by $\alpha = (\alpha_t)_{0 \le t \le T}$ 
	and generated by $W$ together with an independent Brownian motion $B$
	\begin{equation} \label{eq:MKV_general}
		`\mathrm{d}X^{\alpha}_t 
		=
		b \big(t, X^{\alpha}_{t \wedge \cdot}, \Lc\big(\big(X^{\alpha}_{t \wedge \cdot}, \alpha_t\big) \big| \Gc_t \big), \alpha_t \big) 	\mathrm{d}t
		+
		\sigma \big(t, X^{\alpha}_{t \wedge \cdot}, \Lc\big(\big(X^{\alpha}_{t \wedge \cdot}, \alpha_t\big) \big| \Gc_t \big), \alpha_t \big)  	\mathrm{d}W_t 
		+
		\sigma_0 \big(t, X^{\alpha}_{t \wedge \cdot}, \Lc\big(\big(X^{\alpha}_{t \wedge \cdot}, \alpha_t\big)\big| \Gc_t \big), \alpha_t \big)  	\mathrm{d}B_t.`
	\end{equation}

	\vspace{0.5em}

	We will provide in the following a precise definition to the above controlled SDE, depending on the strong/weak formulation considered. 
	Let us first specify the dimensions and some basic conditions on the coefficient functions.	
	Let $(n,\ell, d) \in \N\times\N\times\N$. The coefficient functions
	\[
		b:[0,T] \x \Cc^n   \x \Pc(\Cc^n \x U)  \x U \longrightarrow \R^n,\; 
		\sigma:[0,T] \x \Cc^n    \x \Pc(\Cc^n \x U) \x U \longrightarrow  \S^{n \x d} ,
	\]
	\[
		\sigma_0:[0,T] \x \Cc^n  \x \Pc(\Cc^n \x U)  \x U  \longrightarrow\S^{n\times\ell},\; 
		L:[0,T] \x \Cc^n  \x \Pc(\Cc^n \x U) \x U \longrightarrow \R,\; 
		g: \Cc^n \x \Pc(\Cc^n) \longrightarrow \R,
		\]
	are all assumed to be Borel measurable, and non--anticipative in the sense that
	\[
		\big(b, \sigma, \sigma_0,L \big) (t,\xb,  \nub, u) 
		=
		\big(b, \sigma, \sigma_0,L \big) (t,\xb_{t \wedge \cdot}, \nub(t), u),
		\;\mbox{for all}\;
		(t, \xb, \nub, u) \in [0,T] \x \Cc^n \x \Pc(\Cc^n \x U) \x U .
	\]

\subsection{A weak formulation}

	A weak formulation of the control problem is obtained by considering all weak solutions of the controlled McKean--Vlasov SDE \eqref{eq:MKV_general}. 
	Here the word `weak` refers to the fact that the probability space, as well as the equipped Brownian motion, is not assumed to be fixed, but is a part of the solution itself.
	This is of course consistent with the notion of the weak solution in the classical SDE theory.

	\begin{definition} \label{def:weak_control}
		Let $(t, \nu) \in [0,T] \x \Pc(\Cc^n)$. We say that a term
		\[
			\gamma 
			:= 
			\big( \Om^{\gamma}, \Fc^{\gamma},\P^{\gamma},  \F^{\gamma} = (\Fc^{\gamma}_s)_{0 \le s \le T}, \G^{\gamma} = (\Gc^{\gamma}_s)_{0 \le s \le T}, X^{\gamma}, W^{\gamma}, B^{\gamma}, \mub^{\gamma}, \mu^{\gamma}, \alpha^{\gamma} \big),
		\]
		is a weak control associated with the initial condition $(t, \nu) $ if the following conditions are satisfied:
	\begin{enumerate}[label={$(\roman*)$},leftmargin=*]
		\item $(\Om^{\gamma}, \Fc^{\gamma},\P^{\gamma})$ is a probability space, 
		equipped with two filtrations $ \F^{\gamma} $ and $\G^{\gamma}$ such that, for all $s \in [0,T]$,
		\begin{equation} \label{eq:H_Hypo}
			\Gc_s^{\gamma} \subseteq  \Fc_s^{\gamma},
			~\mbox{\rm and}~
			\E^{\P^{\gamma}}\big[ \mathbf{1}_D \big| \Gc^{\gamma}_s \big] = \E^{\P^{\gamma}}\big[ \mathbf{1}_D \big| \Gc^{\gamma}_T \big],~\P^{\gamma}\mbox{\rm--a.s.,}
			~\mbox{\rm for all}~ D \in \Fc^{\gamma}_s \vee \sigma(W^{\gamma});
		\end{equation}

		\item 
		$X^{\gamma} = (X^{\gamma}_s)_{s \in [0, T]}$ is an $\R^n$--valued, $\F^{\gamma}$--adapted continuous process,
		$\alpha^{\gamma} := (\alpha^{\gamma}_s)_{t \le s \le T}$ is an $U$--valued, $\F^{\gamma}$--predictable process,
		and with the fixed constant $p \ge 0$, one has
		\begin{equation} \label{eq:cond_integ}
			\E^{\P^{\gamma}} \big[ \| X^{\gamma} \|^p \big]
			+
			\E^{\P^{\gamma}} \bigg[ \int_t^T \big( \rho( \alpha^{\gamma}_s, u_0) \big)^p \mathrm{d}s \bigg] < \infty;
		\end{equation}
		
		\item $(W^{\gamma}, B^{\gamma})$ is an $\R^d \x \R^{\ell}$--valued, $\F^{\gamma}$--adapted continuous process; $(W^{\gamma, t}, B^{\gamma, t} ) := \big((W^{\gamma,t}_s)_{0 \le s \le T}, (B^{\gamma,t}_s)_{0 \le s \le T} \big)$, defined by $W^{\gamma,t}_{s} := W^{\gamma}_{s \vee t} - W^{\gamma}_t$, and $B^{\gamma,t}_s := B^{\gamma}_{s \vee t} - B^{\gamma}_t$, $s\in[t,T]$, is a standard $(\F^{\gamma}, \P^\gamma)$--Brownian motion on $[t,T]$;
		$B^{\gamma, t}$ is $\G^{\gamma}$--adapted;
		$\Fc^{\gamma}_t \vee \sigma(W^{\gamma})$ is $\P^\gamma$--independent of $\Gc^\gamma_T$,
		and $\mu^{\gamma} = (\mu^{\gamma}_s)_{t \le s \le T}$ $($resp. $\mub^{\gamma} = (\mub^{\gamma}_s)_{t \le s \le T})$ is a $\G^{\gamma}$--predictable $\Pc(\Cc^n)$--valued $($resp. $\Pc(\Cc^n \x U)$--valued$)$ process satisfying
		\[
			\mu^{\gamma}_s = \Lc^{\P^{\gamma}} \big( X^{\gamma}_{s \wedge \cdot} \big| \Gc^{\gamma}_s \big)\;
			\big(\mbox{\rm resp.}\; 
			\mub^{\gamma}_s
			=
			\Lc^{\P^{\gamma}} \big( \big(X^{\gamma}_{s \wedge \cdot}, \alpha^{\gamma}_{s} \big) \big| \Gc^{\gamma}_s \big) 
			\big),\; \text{\rm for}\; \mathrm{d}\P^\gamma\otimes\mathrm{d}s \mbox{\rm--a.e.}\; (s,\omega) \in [t,T] \x \Om^\gamma;
		\]		

		\item  $X^{\gamma}$ satisfies $\P^{\gamma} \circ (X^{\gamma}_{t \wedge \cdot})^{-1} = \nu(t)$ and
		\begin{equation} \label{eq:MKV_SDE_weak}
			X^{\gamma}_s 
			=
			X^{\gamma}_t + \int_t^s b(r, X^{\gamma}_{\cdot}, \mub^{\gamma}_r, \alpha^{\gamma}_r) \mathrm{d}r 
				+ \int_t^s \sigma(r, X^{\gamma}_{\cdot},  \mub^{\gamma}_r, \alpha^{\gamma}_r) \mathrm{d}W^{\gamma}_r 
				+ \int_t^s \sigma_0(r, X^{\gamma}_{\cdot},  \mub^{\gamma}_r, \alpha^{\gamma}_r) \mathrm{d}B^{\gamma}_r, 
			~s \in [t,T],
			\; \P^\gamma \mbox{\rm--a.s.,}
		\end{equation}
		where the integrals are implicitly assumed to be well--defined.
	\end{enumerate}
	\end{definition}
	For all $(t, \nu) \in [0,T] \x \Pc(\Cc^n)$, let us denote
	\[
		\Gamma_W (t, \nu)
		:=
		\big\{
			\mbox{All weak controls with  initial condition}~
			(t, \nu)
		\big\}.
	\]
	Then with the reward functions $L:[0,T] \x \Cc^n \x \Pc(\Cc^n \x U)\x U \longrightarrow \R$ and $g: \Cc^n \x \Pc(\Cc^n) \longrightarrow \R$,
	we introduce the value function of our McKean--Vlasov optimal control problem by
	\begin{equation} \label{eq:def_VW}
		V_W(t, \nu) 
		:=
		\sup_{\gamma \in \Gamma_W (t, \nu)} J(t, \gamma),
		\;\mbox{where}\;
		J(t, \gamma)
		:=
		\E^{\P^{\gamma}} \bigg[
			\int_t^T L(s, X^{\gamma}_{s \wedge \cdot}, \mub^{\gamma}_s, \alpha^{\gamma}_s) \mathrm{d}s 
			+ 
			g(X^{\gamma}_{T\wedge\cdot}, \mu^{\gamma}_T) 
		\bigg].
	\end{equation}

	\begin{remark}\label{remark:def}
		In a weak control $\gamma$, the filtration $\G^{\gamma}$ is used to model the common noise.
		In particular, $B^{\gamma, t}$ is adapted to $\G^{\gamma}$, and $W^{\gamma, t}$ is independent of $\G^{\gamma}$.
		In the classical strong formulation, $\G^{\gamma}$ is fixed as the filtration $\F^{B^{\gamma,t}}$ generated by $B^{\gamma, t}$,
		but for a general weak control, $\G^{\gamma}$ may be larger than $\F^{B^{\gamma,t}}$. This will be the main difference between the strong and weak formulations in our approach.
		Meanwhile, $(\G^{\gamma}, \F^{\gamma})$ satisfies a $(H)$--hypothesis type condition in \eqref{eq:H_Hypo},
		which is consistent with the classical strong formulation $($see {\rm \Cref{subsec:strong}} below$)$.
		This property will be crucial in our proof of the {\rm DPP} result for the strong formulation of the control problem,
		as well as in the limit theory of the McKean--Vlasov control problem in our accompanying paper {\rm\cite{djete2019general}}.
	\end{remark}

	\begin{remark}			
		\noindent $(i)$
		At this stage, the integrability condition \eqref{eq:cond_integ} could be construed as artificial.
		Depending on more concrete property of the coefficient functions $(b, \sigma, \sigma_0)$,
		it would play the role of an admissibility condition for the control process, and ensure that the stochastic integrals in \eqref{eq:MKV_SDE_weak} are well--defined.
		For a more concrete example, consider the case where $U = \R$ and $u_0 = 0$.
		When $b$, $\sigma$, and $\sigma_0$ are all uniformly bounded,
		one can choose $p=0$ so that all $\R$--valued predictable processes would then be admissible.
		When $\sigma(t,x,u, \nub) = u$, one may choose $p=2$ to ensure that the stochastic integral $\int_t^T \alpha^{\gamma}_s \mathrm{d}W^{\gamma}_s$ is well--defined and is a square--integrable martingale. It is also possible to consider more general types of integrability conditions, such as
		\[
			\E^{\P^{\gamma}} \bigg[ \Phi\bigg(\int_t^T \Psi\big( \rho(u_0,  \alpha^{\gamma}_s ) \big) \mathrm{d}s\bigg)\bigg] < \infty,
		\]
		for given maps $\Phi:[0,\infty)\longrightarrow [0,\infty)$ and $\Psi:[0,\infty)\longrightarrow [0,\infty)$. This would for instance allow to consider exponential integrability requirements. 
		For the sake of simplicity, we have chosen the condition in \eqref{eq:cond_integ}, but insist that it plays no essential role in the proof of the dynamic programming principle.

		\medskip
		
		\noindent $(ii)$	The set $\Gamma_W(t, \nu)$ could be empty, in which case $V_W(t, \nu) = - \infty$ by convention.
		For example, when $\int_{\Cc^n} \| \xb \|^p \nu(d \xb) = \infty$,
		then $\Gamma_W(t,\nu) = \emptyset$, since \eqref{eq:cond_integ} cannot be satisfied.
		Nevertheless, $\Gamma_W(t,\nu)$ is non--empty under either one of the following conditions $($see for instance {\rm \citeauthor*{djete2019general} \cite[Theorem A.2]{djete2019general}} for a brief proof$)$
		\begin{itemize}
			\item
			$(b, \sigma, \sigma_0)$ are bounded and continuous in $(\xb,\nub, u)$ and $\nu \in \Pc(\Cc^n);$
		
			\item
			$(b, \sigma, \sigma_0)$ are continuous in $(\xb,\nub, u)$, 
			and  with the fixed constant $p$ and other positive constants $C$, $p^\prime$, $\hat p$ 
			such that $p^\prime> p \ge 1 \vee \hat p $, and $p^\prime \ge 2 \ge \hat p \ge 0$, 
			we have $\nu \in \Pc_{p'}(\Cc^n)$ and
			\[
				\big| b(t,\xb,\nub, u)\big|
				\le
				C \bigg ( 1+ \| \xb \| + \bigg( \int_{\Cc^n \x U}\big( \| \xb^\prime\|^p + \rho(u_0,u^\prime)^p\big)\nub(\mathrm{d}\xb^\prime,\mathrm{d}u^\prime) \bigg)^{\frac{1}{p}} +\rho(u_0,u) \bigg ),
			\]
			\[
				\big| (\sigma,\sigma_0)(t,\xb, \nub, u)\big|^2
				\le 
				C \bigg ( 1+ \| \xb \|^{\hat p} + \bigg( \int_{\Cc^n \x U}\big( \| \xb^\prime\|^p + \rho(u_0,u')^p \big)\nub(\mathrm{d}\xb^\prime,\mathrm{d}u^\prime) \bigg)^{\frac{\hat p}{p}} + \rho(u_0,u)^{\hat p} \bigg ),
			\] 
			 for all $(t,\xb,\nub,u) \in [0,T] \x \Cc^n \x \Pc(\Cc^n \x U) \x U$.
		\end{itemize}
	\end{remark}

	\begin{remark}
		\label{remark:def2}
		For technical reasons, we introduce an $\R$--valued, $\F^{\gamma}$--adapted continuous process $A^{\gamma} := (A^{\gamma}_s)_{0 \le s \le T}$ by
		$($recall that $\pi: U \to [0,1]$ is a fixed isomorphic bijection between $U$ and $\pi(U))$
		\begin{equation} \label{eq:def_A_gamma}
			A^{\gamma}_s  :=  \int_t^{s \vee t} \pi(\alpha^{\gamma}_r) \mathrm{d}r,~ s \in [0,T],
			~\mbox{\rm so that}~
			\pi \big(\alpha^\gamma_s \big)=\underset{n\rightarrow\infty}{\overline{\lim}}n\big(A^\gamma_s-A^\gamma_{(s-1/n)\vee 0}\big),~\mathrm{d}\P^{\gamma} \otimes \mathrm{d}s \mbox{\rm --a.e. on}~ \Om^{\gamma} \x [t,T].
		\end{equation}
		We further define a $\Pc(\Cc^n \x \Cc \x \Cc^d \x \Cc^{\ell})$--valued process $\muh^{\gamma} = (\muh^{\gamma}_s)_{0 \le s \le T}$ by
		\begin{equation} \label{eq:def_muh}
			\muh^{\gamma}_s 
			:=
			\Lc^{\P^{\gamma}} \big(X^{\gamma}_{s \wedge \cdot}, A^{\gamma}_{s \wedge \cdot}, W^{\gamma},B^{\gamma}_{s \wedge \cdot} \big) \mathbf{1}_{\{ s \in [0,t] \}}
			+
			\Lc^{\P^{\gamma}} \big( \big(X^{\gamma}_{s \wedge \cdot}, A^{\gamma}_{s \wedge \cdot},W^{\gamma},B^{\gamma}_{s \wedge \cdot} \big) \big| \Gc^{\gamma}_s \big)
			\mathbf{1}_{\{ s \in (t,T]\}}.
		\end{equation}
		The process $\muh^{\gamma}$ can be defined to be a $\G^{\gamma}$--adapted and $\P^\gamma$--{\rm a.s.} continuous process $($equipping $\Pc(\Cc^n \x \Cc \x \Cc^d \x \Cc^{\ell})$ with the weak convergence topology$)$.
		Indeed, by \eqref{eq:H_Hypo}, we have 
		\[
			\muh^{\gamma}_s 
			=
			\Lc^{\P^{\gamma}} \big( \big(X^{\gamma}_{s \wedge \cdot}, A^{\gamma}_{s \wedge \cdot},W^{\gamma},B^{\gamma}_{s \wedge \cdot} \big) \big| \Gc^{\gamma}_T \big), 
			\; \P^\gamma\mbox{\rm--a.e.},
			\; \mbox{\rm for all}\; s \in (t, T].
		\]
		Then by {\rm \Cref{lemm:Cond_Law}}, $\muh^{\gamma}$ can be defined to be  $\P^\gamma$--{\rm a.s.} continuous on both $[0,t]$ and $(t,T]$.
		Moreover, using the independence property between $\Fc^{\gamma}_t \vee \sigma(W^{\gamma})$ and $\Gc^{\gamma}_T$, we have,  $\P^\gamma$--{\rm a.s.}
		\begin{align*}
			\lim_{r \searrow t} \muh^{\gamma}_r
			=
			\lim_{r \searrow t} \Lc^{\P^{\gamma}}\! \big( \big(X^{\gamma}_{r \wedge \cdot}, A^{\gamma}_{r \wedge \cdot}, W^{\gamma},B^{\gamma}_{r \wedge \cdot} \big) \big| \Gc^{\gamma}_T \big)
			&=
			\Lc^{\P^{\gamma}}\! \big( \big(X^{\gamma}_{t \wedge \cdot}, A^{\gamma}_{t \wedge \cdot}, W^{\gamma},B^{\gamma}_{t \wedge \cdot} \big) \big| \Gc^{\gamma}_T \big)
			=
			\Lc^{\P^{\gamma}}\! \big(X^{\gamma}_{t \wedge \cdot}, A^{\gamma}_{t \wedge \cdot}, W^{\gamma},B^{\gamma}_{t \wedge \cdot} \big)=
			\muh^{\gamma}_t.
		\end{align*}
	\end{remark}

	\begin{remark}
		It is perfectly possible for us to consider a slightly more general class of control problems allowing for exponential discounting. 
		More precisely, we could have an additional Borel map $k:[0,T] \x \Cc^n \x \Pc(\Cc^n \x U) \x U \longrightarrow \R$ and consider, for fixed $(t,\nu)\in [0,T]\times \Pc(\Cc^n)$, 
		the problem of maximising over $\gamma\in\Gamma_W(t,\nu)$ the functional
		\[
			\tilde J(t,\gamma):=\E^{\P^{\gamma}} \bigg[
			\int_t^T \mathrm{e}^{-\int_t^sk(u, X^{\gamma}_{u \wedge \cdot}, \mub^{\gamma}_u, \alpha^{\gamma}_u)\mathrm{d}u}L(s, X^{\gamma}_{s \wedge \cdot}, \mub^{\gamma}_s, \alpha^{\gamma}_s) \mathrm{d}s 
			+ 
			\mathrm{e}^{-\int_t^Tk(u, X^{\gamma}_{u \wedge \cdot}, \mub^{\gamma}_u, \alpha^{\gamma}_u)\mathrm{d}u}g(X^{\gamma}_{T\wedge\cdot}, \mu^{\gamma}_T) 
			\bigg].
		\]
		We refrained from working at that level of generality for notational simplicity, but our results extend directly to this context.
	\end{remark}

\subsection{A strong formulation }\label{subsec:strong}

	To obtain a strong formulation of the control problem, 
	the usual approach is to consider a fixed probability space, equipped with fixed Brownian motions and fixed Brownian filtrations.
	In fact, this is equivalent to fix the filtrations, in the weak control $\gamma$, to be the Brownian filtrations (see \Cref{prop:strong_ctrl_fixed_space} below). We will therefore present the two equivalent definitions one after the other.
	
	\subsubsection{Strong formulation as a special case of weak formulation}
	Let us start with the main definition.
	\begin{definition} \label{def:strong_ctrl}
		Let $(t, \nu) \in [0,T] \x \Pc(\Cc^n)$. A term
		$
			\gamma = \big( \Om^{\gamma}, \Fc^{\gamma}, \F^{\gamma}, \P^{\gamma}, \G^{\gamma}, X^{\gamma}, W^{\gamma}, B^{\gamma}, \mub^{\gamma}, \mu^{\gamma}, \alpha^{\gamma} \big)
		$
		is called a strong control associated with the initial condition $(t, \nu)$, if
		$\gamma \in \Gamma_W(t, \nu)$
		and the filtrations $\G^{\gamma}$ and $\F^{\gamma}$ are $\P^{\gamma}$-augmented filtrations of
		$\G^{\gamma, \circ} := (\Gc^{\gamma, \circ}_s)_{s \in [0,T]}$ and $\F^{\gamma, \circ} := (\Fc^{\gamma, \circ}_s)_{s \in [0,T]}$, which are defined by
			\[
				\Gc^{\gamma, \circ}_s  := 
				\begin{cases}
					\{ \emptyset, \Om^{\gamma} \},\; \mbox{\rm if}\;  0\leq s < t,\\
					\sigma \big (B^{\gamma,t}_r:r \in [t, s] \big), \; \mbox{\rm if}\;  0\leq t \le s \le T,
				\end{cases}
				\; \mbox{\rm and}\; 
				\Fc^{\gamma, \circ}_s :=
				\begin{cases}
					\sigma \big(X^{\gamma}_{s \wedge \cdot}\big), \; \mbox{\rm if} \;  0\leq s < t,\\
					\sigma\big( (X^{\gamma}_{t \wedge \cdot}, B^{\gamma,t}_r, W^{\gamma, t}_r): r \in [t,s] \big),
					\; \mbox{\rm if}\;  0\leq t \le s \le T.
				\end{cases}
			\]
		If, in addition, the control process $\alpha^{\gamma}$ is $\G^{\gamma}$--predictable,
		then $\gamma$ is called a $\B$--strong control.
	\end{definition}
	Let us denote by $\Gamma_S (t, \nu)$ the collection of all strong controls  with initial condition $(t, \nu)$,
	and by $\Gamma_S^\B (t, \nu)$ the collection of all $\B$--strong controls with initial condition $(t, \nu)$, i.e.
	\[
		\Gamma_S^\B (t, \nu)
		:=
		\big\{
		    \gamma \in \Gamma_S (t, \nu)
		    :
		    \alpha^\gamma\; \mbox{is}\; \G^\gamma\mbox{--predictable}
		\big\}.
	\]

	\begin{remark}
		For a strong control $\gamma \in \Gamma_S(t, \nu)$,
		the filtration $\G^{\gamma}$ is generated by $B^{\gamma, t}$,
		and $\F^{\gamma}$ is generated by the initial variable $X^{\gamma}_{t \wedge \cdot}$ and the Brownian motion $W^{\gamma, t}_{\cdot}$ and $B^{\gamma, t}_{\cdot}$.
		Consequently, the control process $\alpha^{\gamma}$ is adapted to the filtration generated by $(X^{\gamma}_{t \wedge \cdot}, W^{\gamma, t}_{\cdot}, B^{\gamma, t}_{\cdot})$, and the common noise comes only from $B^{\gamma, t}$.
		We will show in {\rm \Cref{prop:strong_ctrl_fixed_space}} that this is equivalent to the case with a fixed probability space equipped with fixed Brownian motions and the initial random variable.
		We here define the strong control rules as special cases of weak control rules in order to avoid repeating all the technical conditions in {\rm\Cref{def:weak_control}}.
	\end{remark}

	Our proof of the dynamic programming principle for the strong formulation of the McKean--Vlasov problem relies essentially on its equivalence to the weak formulation, 
	which requires the following standard Lipschitz condition on the coefficient functions.
	Moreover, this condition ensures the existence and uniqueness of the solution to SDE \eqref{eq:MKV_SDE_weak} (see e.g. \Cref{theorem_Existence/uniqueness-SDE}).
	
	\begin{assumption} \label{assum:Lip}
		Let the constant in \eqref{eq:cond_integ} be $p=2$. 
		There exists a constant $C > 0$ such that, for all $(t, \xb, \xb^\prime, \nub, \nub^\prime, u) \in [0,T] \x \Cc^n \x \Cc^n \x \Pc_2(\Cc^n \x U)\x \Pc_2(\Cc^n \x U) \x U$, one has
		\begin{equation} \label{eq:Lip_cond}
			\big\| \big(b, \sigma, \sigma_0 \big) (t,\xb, \nub, u) - \big(b, \sigma, \sigma_0 \big) (t,\xb', \nub^\prime, u) \big\| 
			\le C \big(\| \xb - \xb'\| + \Wc_2(\nub, \nub') \big),
		\end{equation}
		\[
			\big\| (b, \sigma, \sigma_0) (t,\xb, \nub, u) \big\|^2 
			\leq
			C \bigg( 1+ \| \xb \|^2 + \int_{\Cc^n \x U} \big(\| \yb\|^2 + \rho(\hat u,u_0)^2\big) \nub(\mathrm{d}\yb, \mathrm{d}\hat u)  + \rho(u,u_0)^2 \bigg).
	    \]
	\end{assumption}

	Under \Cref{assum:Lip}, the set $\Gamma_S(t,\nu)$ is nonempty for all $(t,\nu) \in [0, T] \x \Pc_2(\Cc^n)$ (see e.g. \Cref{theorem_Existence/uniqueness-SDE}).
	We then introduce the following strong formulation (resp. $\B$--strong formulation) of the McKean--Vlasov control problem
	\begin{equation}
	\label{eq:Strong_Value_function-McK}
		V_S(t,\nu)
		:=
	    	\sup_{\gamma \in \Gamma_S (t, \nu)} 
		J(t, \gamma),
		~\mbox{\rm and}~
		V_S^\B(t,\nu)
		:=
	    	\sup_{\gamma \in \Gamma_S^\B (t, \nu)} 
		J(t, \gamma).
	\end{equation}

	\begin{remark}[The case without common noise: $\ell = 0$]
		In the literature, the McKean--Vlasov control problem without common noise has also largely been studied. This contained as a special case in our setting. Indeed, when $\ell = 0$, the process $B^{\gamma,t}$ degenerates to be a singleton and hence 
		$ \Gc^{\gamma}_s = \{ \emptyset, \Om^{\gamma} \}$ for all $s \in [0,T]$.
		It follows that $\mub^{\gamma}$ appearing in \eqref{eq:MKV_SDE_weak} turns out to satisfy
		\[
			\mub^{\gamma}_s = \Lc^{\P^{\gamma}} \big(X^{\gamma}_{s \wedge \cdot}, \alpha^{\gamma}_s),\; \text{\rm for}\; \mathrm{d}\P^\gamma\otimes\mathrm{d}t \mbox{\rm--a.e.}\; (s,\omega) \in [t,T] \x \Om^\gamma,
		\]
		and the value function $V_S(t,\nu)$ in \eqref{eq:Strong_Value_function-McK} is the standard formulation of the control problem without common noise {\color{black} $($see e.g. {\rm\cite{carmona2018probabilisticI}}$)$.}
	\end{remark}

\subsubsection{Strong formulation on a fixed probability space}
\label{subsubsec:strong_form}

	For $0 \le s \le t \le T$, let $\Cc^n_{s,t} := C([s,t], \R^n)$ denote the space of all $\R^n$--valued continuous paths on $[s,t]$,
	we then introduce a first canonical space, for every $t \in [0,T]$
	\begin{equation} \label{eq:Omt}
		\Om^t
		:=
		\Cc^n_{0,t} \x \Cc^d_{t, T} \x \Cc^{\ell}_{t,T},
		~
		\Fc^t := \Bc(\Om^t),
	\end{equation}
	with corresponding canonical processes $\zeta := (\zeta_s)_{0 \le s \le t}$, $W := (W_s)_{t \le s \le T}$, and $B := (B_s)_{t \le s \le T}$.
	Let $W^t_s := W_{s \vee t} - W_t$ and $B^t_s := B_{s \vee t} - B_t$ for all $s \in [0, T]$, and
	define $\F^{t, \circ} = (\Fc^{t,\circ}_s)_{0 \le s \le T}$ and $\G^{t,\circ} = (\Gc^{t,\circ}_s)_{0 \le s \le T}$ by
	\[
		\Fc^{t,\circ}_s :=
		\begin{cases}
			\sigma \big(\zeta_{s \wedge \cdot}\big), \; \mbox{\rm if}\; 0\leq s < t,\\
			\sigma \big((\zeta_{t \wedge \cdot}, W^t_r, B^t_r):r \in [t,s] \big),\; \mbox{\rm if}\; 0\leq t \le s \le T,
		\end{cases}
		\; \mbox{\rm and}\;
		\Gc^{t, \circ}_s := 
		\begin{cases}
			\{\emptyset, \Om^t\}, \; \mbox{\rm if}\; 0\leq s < t, \\
			\sigma \big( B^t_r:r \in [t,s] \big),\; \mbox{\rm if}\;  0\leq t \le s \le T.
		\end{cases}
	\]
	Let $(t,\nu) \in [0,T] \x \Pc_2(\Cc^n)$, we fix a probability measure $\P^t_{\nu}$ on $\Om^t$, such that $\Lc^{\P^t_{\nu}}\big(\zeta_{t \wedge \cdot}\big)= \nu(t)$, 
	and $(W^t, B^t)$ are standard Brownian motions on $[t,T]$, independent of $\zeta$.
	Let $\F^t = (\Fc^t_s)_{0 \le s \le T}$ and $\G^t = (\Gc^t_s)_{0 \le s \le T}$ be the $\P^t_{\nu}$--augmented filtration of $\F^{t,\circ}$ and $\G^{t,\circ}$,
	we denote by $\Ac_2(t,\nu)$ (resp. $\Ac_2^\B(t,\nu)$) the collection of all $U$--valued processes $\alpha = (\alpha_s)_{t \le s \le T}$ which are $\F^t$--predictable (resp. $\G^t$--predictable) and such that
	\[
		\E^{\P^t_{\nu}} \bigg[ \int_t^T \big( \rho(u_0, \alpha_s) \big)^2 \mathrm{d}s \bigg] 
		< \infty.
	\]
	Then, given $\alpha \in \Ac_2(t, \nu)$, let $X^{\alpha}$ be the unique strong solution (in sense of \Cref{def:strong_sol}) of the SDE, 
	with initial condition $X^{\alpha}_{t \wedge \cdot}= \zeta_{t \wedge \cdot}$,
	\begin{equation} \label{eq:MKV_strong}
		X^{\alpha}_s
		= 
		X^{\alpha}_t 
		+
		\int_t^s b \big(r, X^{\alpha}_{r \wedge \cdot}, \mub^{\alpha}_r, \alpha_r \big) \mathrm{d}r
		+
		\int_t^s \sigma\big(r, X^{\alpha}_{r \wedge \cdot}, \mub^{\alpha}_r, \alpha_r \big) \mathrm{d} W^t_r
		+ 
		\int_t^s \sigma_0 \big(r, X^{\alpha}_{r \wedge \cdot}, \mub^{\alpha}_r, \alpha_r \big) \mathrm{d} B^t_r,\; t\leq s\leq T, \; \P^t_{\nu} \mbox{--a.s.},
	\end{equation}
	where $\mub^{\alpha}_r =  \Lc^{\P^t_{\nu}}\big((X^{\alpha}_{r \wedge \cdot}, \alpha_r )\big| \Gc^t_r \big)$, $\mathrm{d} \P^t_{\nu} \x \mathrm{d}r$--a.e. on $\Om^t \x [t, T]$.
	Notice that the existence and uniqueness of a solution to SDE \eqref{eq:MKV_strong} is ensured by \Cref{assum:Lip} (see \Cref{theorem_Existence/uniqueness-SDE}).
	Finally, we denote for any $\alpha\in \Ac_2(t, \nu)$, $\mu^{\alpha}_s :=  \Lc^{\P^t_{\nu}}\big(X^{\alpha}_{s \wedge \cdot} \big| \Gc^t_s \big)$,  $s\in[t,T]$.
	
	\vspace{0.5em}
	
	We next show that the above strong formulation of the control problem with fixed probability space is equivalent to that in \Cref{def:strong_ctrl} as a special case of the weak control rules.

	\begin{proposition} \label{prop:strong_ctrl_fixed_space}
		Let {\rm\Cref{assum:Lip}} hold true.
		Then for all $(t,\nu) \in [0,T] \x \Pc_2(\Cc^n)$, one has
		\begin{equation} \label{eq:Strong_Value_function-McK_equiv}
			V_S(t,\nu)
			=
			\sup_{\alpha \in \Ac_2(t, \nu)} 
			J(t,\nu, \alpha),
			~\mbox{\rm and}~
			V_S^\B(t,\nu)
			=
			\sup_{\alpha \in \Ac_2^\B(t, \nu)} 
			J(t,\nu, \alpha),
		\end{equation}
		where
		\[
			J(t,\nu, \alpha)
			:=
			\E^{\P^t_{\nu}} \bigg[
				\int_t^T L(s, X^{\alpha}_{s \wedge \cdot}, \mub^{\alpha}_s, \alpha_s) \mathrm{d}s 
				+ 
				g(X^{\alpha}_{\cdot}, \mu^{\alpha}_T) 
			\bigg].
		\]

	\end{proposition}
	\begin{proof}
		We will only consider the case of $V_S$, since the arguments for the case of $V^{\B}_S$ are exactly the same. First, given $\alpha \in \Ac_2(t,\nu)$, let us define 
	        \[
			\gamma
			:=
			\big( \Om^t, \Fc^t, \P^t_{\nu}, \F^t,  \G^t, X^\alpha, W^t, B^t, \mub^\alpha, \mu^\alpha, \alpha \big).
		\]
		Then it is straightforward to check that $ \gamma$ is a strong control rule
		(i.e. $ \gamma \in \Gamma_S(t, \nu)$) such that $J(t,  \gamma) = J(t, \nu, \alpha)$.
		
\medskip

		Next, let $\gamma \in \Gamma _S(t, \nu)$.
		Notice that $(X^{\gamma}, \alpha^\gamma)$ is $\F^\gamma$--predictable, and $(\mu^{\gamma}, \mub^{\gamma})$ is $\G^{\gamma}$--predictable.
		Using for instance \citeauthor*{claisse2016pseudo} \cite[Proposition 10]{claisse2016pseudo} (with a slight extension consisting simply in having a larger $\Fc_0$),
		there exists two Borel measurable functions
		$\Psi_1: [0,T] \x \Om^t \longrightarrow \R^n \x U$ and $\Psi_2: [0,T] \x \Cc^{\ell} \longrightarrow \Pc(\Cc^n) \x \Pc(\Cc^n \x U)$ such that
		\[
			\big(X^{\gamma}_{s}, \alpha^\gamma_s\big)
			=
			\Psi_1 \big(s, X^\gamma_{t \wedge \cdot}, W^{\gamma,t}_{s \wedge \cdot}, B^{\gamma,t}_{s \wedge \cdot} \big),
			~
			\big(\mu^{\gamma}_s, \mub^{\gamma}_s \big)
			=
			\Psi_2(s, B^{\gamma,t}_{s \wedge \cdot}),
			\; s \in [0,T],
			~\P^{\gamma}\mbox{\rm --a.s.}
		\]
		Then, on $\Om^t$,
		let us define $( X^{\star}_s, \alpha^{\star}_s) := \Psi_1(s, \zeta{t \wedge \cdot}, W^{t}_{s \wedge \cdot}, B^{t}_{s \wedge \cdot})$ 
		and $(\mu^{\star}_s, \mub^{\star}_s) := \Psi_2(s, B^t_{s \wedge \cdot})$,
		so that
		\[
			 \alpha^{\star} \in \Ac_2(t,\nu),
			 ~\mbox{\rm and}~
			\P^t_{\nu} \circ
			\big(  X^{\star}, W^{t}, B^{t}, \alpha^{\star}, \mub^{\star}, \mu^{\star}
			\big)^{-1}
			=
			\P^\gamma \circ
			\big( X^\gamma, W^{\gamma,t},B^{\gamma,t},  \alpha^\gamma, \mub^\gamma, \mu^\gamma
			\big)^{-1}.
		\]
		This implies that $X^{\star}$ is the unique strong solutions to SDE \eqref{eq:MKV_strong} with control $\alpha^{\star}$,
		such that 
		$\mu^{\star}_s = \Lc^{\P^t_{\nu}}(X^{\star}_{s \wedge \cdot} | \Gc^t_s)$, $\mub^{\star}_s = \Lc^{\P^t_{\nu}}((X^{\star}_{s \wedge \cdot}, \alpha^{\star}_s) | \Gc^t_s)$
		and
		$J(t, \nu, \alpha^{\star}) = J(t, \gamma)$.
	\end{proof}

\section{The dynamic programming principle}

\label{sec:DPP}

	The main results of our paper consist in the dynamic programming principle (DPP) for the previously introduced formulations of the McKean--Vlasov control problem.
	We will first prove the DPP for the general strong and weak control problems introduced in \Cref{sec:Formulation},
	and then show how they naturally induce the associated results in the Markovian case.
	Finally, we also discuss heuristically the Hamilton--Jacobi--Bellman (HJB for short) equations which can be deduced for each formulation.

\subsection{The dynamic programming principle in the general case}

\subsubsection{Dynamic programming principle for the weak control problem}

	To provide the dynamic programming principle of the McKean--Vlasov control problem \eqref{eq:def_VW},
	let us introduce another canonical space 
	\[
		\Om^{\star} := \Cc^{\ell} \x C\big([0,T], \Pc(\Cc^n \x \Cc \x \Cc^d \x \Cc^{\ell})\big),
		~\mbox{with canonical process}~(B^{\star}, \muh^{\star}) := (B^{\star}_s, \muh^{\star}_s)_{s \in [0,T]},
	\]
	and canonical filtration $\G^{\star} := (\Gc^{\star}_s)_{0 \le s \le T}$ defined by $\Gc^{\star}_s := \sigma \big \{(\muh^{\star}_r, B^{\star}_r):r \in [0, s] \big\}$, $s\in[0,T]$.
	Then, for every $\G^{\star}$--stopping time $\tau^{\star}$ (which can then be written as a function of $B$ and $\muh$), 
	for all $(t,\nu) \in [0,T]\times\Pc(\Cc^n)$ and $\gamma\in\Gamma_W(t,\nu)$, we define (recall that $\muh^{\gamma}$ is defined by \eqref{eq:def_muh})
	\begin{equation} \label{eq:tau_from_taut}
		\tau^{\gamma} 
		:=
		\tau^{\star} \big(B^{\gamma, t}_{\cdot}, \muh^{\gamma}_\cdot \big).
	\end{equation}
	
	\begin{theorem} \label{thm:WeakDPP}
		The value function $V_W:[0,T] \x \Pc(\Cc^n) \longrightarrow \R \cup\{-\infty, \infty\}$
		of the weak McKean--Vlasov control problem \eqref{eq:def_VW} is upper semi--analytic.
		Moreover, let $(t, \nu) \in [0,T] \x \Pc(\Cc^n)$, $\tau^{\star}$ be a $\G^{\star}$--stopping time taking values in $[t, T]$,
		and $\tau^{\gamma}$ be defined in \eqref{eq:tau_from_taut},
		one has
		\begin{equation} \label{eq:WeakDPP}
			V_W(t, \nu) 
			=
			\sup_{\gamma \in \Gamma_W(t, \nu)} 
			\E^{\P{^\gamma}} \bigg[
				\int_t^{\tau^{\gamma}} L(s, X^{\gamma}_{s \wedge \cdot}, \mub^{\gamma}_s, \alpha^{\gamma}_s) \mathrm{d}s 
				+ 
				V_W\big(\tau^{\gamma}, \mu^{\gamma}_{\tau^{\gamma}} \big) 
			\bigg].
		\end{equation}
	\end{theorem}

\subsubsection{Dynamic programming for the strong control problems}

	We now consider the two strong formulations of the control problems introduced in \eqref{eq:Strong_Value_function-McK}, or equivalently in \eqref{eq:Strong_Value_function-McK_equiv}.
	To formulate the DPP results, we will rather use the fixed probability space context in \eqref{eq:Strong_Value_function-McK_equiv}.
	Recall that, given initial condition $(t, \nu) \in [0,T] \x \Pc_2(\Cc^n)$, a fixed probability space $(\Om^t, \Fc^t, \P^t_{\nu})$ is defined in and below \eqref{eq:Omt}.
	Let us first consider the strong control problem $V^{\mathbb{B}}_S$.

	\begin{theorem} \label{thm:B-StrongDPP}
		Let {\rm\Cref{assum:Lip}} hold. Then the value function $V^\B_S:[0,T] \x \Pc_2(\Cc^n) \longrightarrow \R \cup\{-\infty, \infty\}$ is upper semi--analytic.
		Moreover, let $(t, \nu) \in [0,T] \x \Pc_2(\Cc^n)$, and $\tau$ be a $\G^{t, \circ}$--stopping time on $(\Om^t, \Fc^t, \P^t_{\nu})$, taking values in $[t, T]$, one has
		\begin{equation} \label{eq:B-strongDPP}
			V_S^\B(t, \nu) 
			=
			\sup_{\alpha \in \Ac_2^\B(t, \nu)} 
			\E^{\P^t_{\nu}} \bigg[
				\int_t^{\tau}  L(s, X^{\alpha}_{s \wedge \cdot}, \mub^{\alpha}_s, \alpha_s) \mathrm{d}s 
				+ 
				V_S^\B\big(\tau, \mu^{\alpha}_{\tau} \big) 
			\bigg].
		\end{equation}
	\end{theorem}

	For the strong control problem $V_S$, we need some additional regularity conditions on the coefficient functions.
	
	\begin{assumption} \label{assum:Growth}
		For all $t \in [0,T]$, the functions
		\[
			(b,\sigma,\sigma_0):(\xb, \nub, u) \in \Cc^n  \x \Pc(\Cc^n \x U) \x U
			\longmapsto 
			(b,\sigma,\sigma_0)(t,\xb, \nub, u)
			\in \R^n \x \S^{n\x d} \x \S^{n \x \ell},
		\] 
		are continuous, and there exists a constant $C > 0$ such that, for all $(t,\xb, u, \nub) \in [0,T] \x \Cc^n \x U \x \Pc(\Cc^n \x U)$,
		\[
			\big| (L, g) (t,\xb, \nub, u) \big|^2 
			\leq
			C \bigg ( 1+ \| \xb \|^2 + \int_{\Cc^n \x U} \big(\| \yb\|^2 + \rho( u',u_0)^2\big) \nub(\mathrm{d}\yb, \mathrm{d} u')  + \rho(u,u_0)^2 \bigg ).
	    \]
Moreover, the map
	    \[
	        (\xb, \nub, u) \in \Cc^n  \x \Pc_2(\Cc^n \x U) \x U
		\longmapsto
	        (L,g)(t,\xb, \nub, u) \in \R \x \R,
	    \]
	    is lower semi--continuous for all $t \in [0,T]$.
	\end{assumption}

	\begin{theorem} \label{thm:StrongDPP}
		Let {\rm \Cref{assum:Lip}} and {\rm \Cref{assum:Growth}} hold true.
		Let  $(t, \nu) \in [0,T] \x \Pc_2(\Cc^n)$, 
		and $\tau$ be a $\G^{t, \circ}$--stopping time on $(\Om^t, \Fc^t, \P^t_{\nu})$ taking values in $[t, T]$. 
		Then
		\[
			V_S(t, \nu)
			=
			V_W(t,\nu), 
		\]
		so that the value function $V_S:[0,T] \x \Pc_2(\Cc^n) \longrightarrow \R \cup\{-\infty, \infty\}$ is upper semi--analytic,
		and one has
		\begin{equation} \label{eq:strongDPP}
			V_S(t, \nu) 
			=
			\sup_{\alpha \in \Ac_2(t, \nu)} 
			\E^{\P^t_{\nu}} \bigg[
				\int_t^{\tau}  L(s, X^{\alpha}_{s \wedge \cdot}, \mub^{\alpha}_s, \alpha_s) \mathrm{d}s 
				+ 
				V_S\big(\tau, \mu^{\alpha}_{\tau} \big) 
			\bigg].
		\end{equation}
	\end{theorem}

	\begin{remark}
		$(i)$ Our results for the dynamic programming for $V_W$ and $V_S$ in {\rm\Cref{thm:WeakDPP}} and {\rm\Cref{thm:StrongDPP}} are new in this general framework. 
		For the result in {\rm\Cref{thm:B-StrongDPP}}, where the control is adapted to the common noise $B$,
		the same {\rm DPP} result has been obtained in {\rm \citeauthor*{pham2016dynamic} \cite[Proposition $3.1$]{pham2016dynamic}.}
However, our result is more general for two reasons.
		First, we do not require any regularity conditions on the reward functions $L$ and $g$, thanks to our use of measurable selection arguments.
		Second, we are able to stay in a generic non--Markovian framework with interaction terms given by the law of both control and controlled processes,
		while the results of {\rm \cite{pham2016dynamic}} are given in a Markovian context with interaction terms given by the law of controlled process.

\medskip		
		$(ii)$ From our point of view, the formulations $V_W$ and $V_S$ in \eqref{eq:def_VW} and \eqref{eq:Strong_Value_function-McK} seem to be more natural, because they should be the ones arising naturally as limit of finite population control problems $($see {\rm \citeauthor*{lacker2017limit} \cite{lacker2017limit}} for the case without common noise, and {\rm \citeauthor*{djete2019general} \cite{djete2019general}} for the context with common noise$)$.
		Indeed, for the problem with a finite population $N$, when the controller observes the evolution of the empirical distribution of $(X^1, \dots, X^N)$, it is more reasonable to assume that he/she uses the information generated by both $(X_{t \wedge \cdot},W,B)$ $($as in the definition of $V_S)$, rather than just the information from $B$ $($as in the definition $V_S^{\mathbb{B}})$, to control the system. In this sense, the formulation $V_S^\B$ may not be the most natural strong formulation for McKean--Vlasov control problems with common noise. 
	\end{remark}

	\begin{remark} \label{rem:ddp_under_regularity}
		The {\rm DPP} result for $V_S$ in {\rm \Cref{thm:StrongDPP}} has been proved under additional regularity conditions, namely the ones given in {\rm \Cref{assum:Growth}}. This should appear as a surprise to readers familiar with the measurable selection approach to the {\rm DPP} for classical stochastic control problems. 	We will try here to give some intuition on why, at least if one uses our method of proof, there does not seem to be any way to make do without these aditional assumptions.

\medskip		
Let us consider the classical conditioning argument in the proof of the {\rm DPP}. Given a control process $\alpha := (\alpha_s)_{s \in [t,T]} \in \Ac_2(t,\nu)$, which is adapted to the filtration generated by $(X_{t \wedge \cdot},W^t_s, B^t_s )_{s\in[t,T] }$,
		we consider some time $t_o \in( t,T]$, and the filtration $\Gt := (\Gct_s)_{s \in [t,T]}$, generated by $B^t$.
		Then, under the {\rm r.c.p.d.} of $\P^t_{\nu}$ knowing $\Gct_{t_o}$, the process $(\alpha_s )_{s \in [t_o, T]}$ will be adapted to the filtration generated by $(X_{t_o \wedge \cdot}, W^{t_o}_s,B^{t_o}_s)_{ s\in[t_o,T]}$ together with $(W^t_s)_{s \in [t, t_o]}$.
		Because of the randomness of $(W^t_s )_{s \in [t, t_o]}$, we cannot consider $(\alpha_s )_{s \in [t_o, T]}$ as a `strong` control process under the {\rm r.c.p.d.} of $\P^t_{\nu}$ knowing $\Gct_{t_o}$.
		
\medskip		
		To bypass this difficulty, we will need to use the equivalence result $V_S=V_W$ together with the {\rm DPP} results for $V_W$ given by {\rm\Cref{thm:WeakDPP}}.
		The equivalence result will be proved in our accompanying paper {\rm\cite{djete2019general}} under the integrability and regularity conditions in {\rm\Cref{assum:Lip}} and {\rm\Cref{assum:Growth}}.
	\end{remark}

\subsection{Dynamic programming principle in the Markovian case}

	With the DPP results in the general non--Markovian context of \Cref{thm:WeakDPP}, \Cref{thm:B-StrongDPP} and \Cref{thm:StrongDPP},
	we can easily establish the DPP results for the control problems in the Markovian setting.
	In fact, we will consider a framework which is slightly more general than the classical Markovian formulation, by considering the so--called updating functions, as in \citeauthor*{brunick2013mimicking} \cite{brunick2013mimicking}.

\medskip
Let $E$ be a non--empty Polish space. A Borel measurable function $\Phi: \Cc^n \longrightarrow C([0,T], E)$ is called an updating function if it satisfies
	\[
		\Phi_t(\xb) = \Phi_t(\xb(t \wedge \cdot)),
		\; \mbox{\rm for all}\; (t, \xb) \in [0,T] \x \Cc^n,
	\]
	and for all $0 \le s \le t \le T$
	\[
		\big( \Phi_r (\xb) \big)_{r \in [s,t]} = \big( \Phi_r(\xb^\prime) \big)_{r \in [s,t]},
		\; \mbox{\rm whenever}\; \Phi_s(\xb) = \Phi_s(\xb^\prime),
		\;\mbox{\rm and}\;\big(\xb(r) - \xb(s) \big)_{r \in [s,t]} = \big(\xb^\prime(r) - \xb^\prime(s) \big)_{r \in [s,t]}.
	\]
	The intuition of the updating function $\Phi$ is the following:
	the value of $\Phi_t(\xb)$ depends only on the path of $\xb$ up to time $t$,
	and for $0\leq s < t$, $\Phi_t(\xb)$ depends only on $\Phi_s(\xb)$ and the increments of $\xb$ between $s$ and $t$.
	On the canonical space $\Cc^n$, let $X := (X_t)_{t \in [0,T]}$ be the canonical process. We also define a new process $Z_t := \Phi_t(X)$, $t\in[0,T]$.
	Let us borrow some examples of updating functions from \cite{brunick2013mimicking}.
	
	\begin{example}
		$(i)$ The most simple updating function is the running process itself, that is
		\[
			\Phi_t(\xb) := \xb(t), \; \mbox{\rm with}\; E = \R^n.
		\]

		$(ii)$ Let $M^i_t(\xb) := \max_{0 \le s \le t} \xb^i(s)$ for $i=1, \cdots, n$, $t\in[0,T]$, and 
		$A_t(\xb) := \int_0^t \xb(s) \mathrm{d}s$, $t\in[0,T]$. 
		Then the running process, together with the running maximum and running average process, is also an example of updating functions
		\[
			\Phi_t(\xb) := \big( \xb(t), M_t(\xb), A_t(\xb) \big),\; \mbox{\rm with}\; E = \R^n \x \R^n \x \R^n.
		\]
	\end{example}

	Throughout this subsection, we fix an update function $\Phi$. In this context, one can in fact define the value function on $[0,T] \x \Pc(E)$ under some additional conditions.
	Given $\nub \in \Pc(\Cc^n \x U)$ (resp. $\nu \in \Pc(\Cc^n)$), 
	let us consider $X$ (resp. $(X, \alpha)$) as canonical element on the canonical space $\Cc^n$ (resp. $\Cc^n \x U$),
	and then define
	\[
		[\nub]^{\circ}_t 
		:=
		\nub \circ (\Phi_t(X), \alpha)^{-1} \in \Pc(E\x U)
		\; \big(\mbox{resp.}\; [\nu]^{\circ}_t := \nu \circ (\Phi_t(X))^{-1} \in \Pc(E) \big),\; t\in[0,T].
	\]

	\begin{assumption} \label{assum:Markov}
		For a fixed updating function $\Phi: \Cc^n \longrightarrow C([0,T], E)$,
		there exist Borel measurable functions 
		$(b^{\circ}, \sigma^{\circ}, \sigma^{\circ}_0, L^{\circ}, g^{\circ}): [0,T] \x E \x U \x \Pc(E \x U) \longrightarrow \R^n \x \S^{n \x d} \x \S^{n \x \ell} \x \R \x \R,$ such that
		\[
			\big( b,\sigma,\sigma_0, L, g \big)(t,\xb, \nub, u)
			=
			\big(b^{\circ}, \sigma^{\circ}, \sigma^{\circ}_0, L^{\circ}, g^{\circ} \big) (t,\Phi_t(\xb), [\nub]^{\circ}_t, u ),\; \mbox{\rm for all}\; (t,\xb,u, \nub)\in [0,T] \x \Cc^n \x U \x \Pc(\Cc^n \x U).
		\]
	\end{assumption}
	
	Let $t \in [0,T]$, and $\nu^\circ \in \Pc(E)$, we define first the following sets
	\[
		\Vc(t,\nu^\circ):=\big\{\nu \in \Pc(\Cc^n) : [\nu]^\circ_t = \nu^{\circ}\big\}, 
	\]
	\[
		\Gamma^{\circ}_W(t, \nu^{\circ})
		:=
		\bigcup_{\nu \in \Vc(t,\nu^\circ)} \Gamma_W(t, \nu),
		~
		\Gamma^{\circ}_S(t, \nu^{\circ})
		:=
		\bigcup_{\nu \in \Vc(t,\nu^\circ)} \Gamma_S(t, \nu),
		~
		\Gamma^{\B,\circ}_S(t, \nu^{\circ})
		:=
		\bigcup_{\nu \in \Vc(t,\nu^\circ)} \Gamma_S^\B(t, \nu),
	\]
	as well as the value functions, with $J(t, \gamma)$ defined in \eqref{eq:def_VW},
	\[
		V_W^{\circ} (t,\nu^{\circ})
		:=
		\sup_{\gamma \in \Gamma_W^{\circ} (t, \nu^{\circ})} J(t,\gamma),
		~
		V_S^{\circ} (t,\nu^{\circ})
		:=
		\sup_{\gamma \in \Gamma_S^{\circ} (t, \nu^{\circ})} J(t,\gamma)
		~\text{\rm and}~
		V_S^{\B,\circ} (t,\nu^{\circ})
		:=
		\sup_{\gamma \in \Gamma_S^{\B,\circ} (t, \nu^{\circ})} J(t,\gamma).
	\]

	\begin{remark}
		When the updating function is the running process given by $\Phi_t(\xb) := \xb(t)$, 
		the problems $V_W^{\circ}$, $V_S^{\circ}$ and $V_S^{\B,\circ}$ are of course exactly the classical Markovian formulation of the control problems.
	\end{remark}

	\begin{lemma} \label{lemm:equality_Markov}
		Let {\rm \Cref{assum:Markov}} hold true, and fix some $t\in[0,T]$. 
		Then, for any $(\nu_1, \nu_2) \in \Pc(\Cc^n)\times \Pc(\Cc^n)$
		such that $[\nu_1]^{\circ}_t = [\nu_2]^{\circ}_t$,
		one has
		\[
			V_W(t,\nu_1) =V_W(t,\nu_2),
			~
			V_S(t,\nu_1) = V_S(t,\nu_2)
			~\mbox{\rm and}~
			V_S^\B(t,\nu_1) = V_S^\B(t,\nu_2).
		\]
		Consequently, for all $\nu \in \Pc(\Cc^n)$, one has
		\[
			V_W(t,\nu) = V_W^{\circ} (t, [\nu]^{\circ}_t ),
			~
			V_S(t,\nu) = V_S^{\circ} (t, [\nu]^{\circ}_t ),
			~\mbox{\rm and}~
			V_S^\B(t,\nu) = V_S^{\B, \circ} (t, [\nu]^{\circ}_t).
		\]
	\end{lemma}

	\begin{proof}
		We will only consider the equality for $V_W$, the arguments for $V_S$ and $V_S^\B$ will be the same.
		First, we can consider $\nu_2$ as a probability measure defined on the canonical space $\Cc^n$ with canonical process $X$, and containing the random variable $Z_t := \Phi_t(X)$.
		Then, on (a possible enlarged) probability space $(\Cc^n, \Bc(\Cc^n), \nu_2)$,
		there exists a Borel measurable function $\psi: E \x [0,1] \longrightarrow \Cc^n$, together with a random variable $\eta$ with uniform distribution on $[0,1]$, which is independent of $Z_t$, such that $
			\nu_2 \circ \big( Z_t , X_{\cdot} \big)^{-1} 
		=
			\nu_2 \circ \big( Z_t, \psi(Z_t, \eta) \big)^{-1}.$ Next, consider an arbitrary $
		    \gamma_1 := \big( \Om^1, \Fc^1, \P^1, \F^1,  \G^1, X^1, W^1, B^1,  \mub^1, \mu^1, \alpha^1 \big) \in \Gamma_W(t,\nu_1).$	Without loss of generality (that is up to enlargement of the space), we assume that there exists a random variable $\eta$ with uniform distribution on $[0,1]$ in the probability space $ ( \Om^1, \Fc^\mathbf{1}_0, \P^1 )$,
		and which is independent of the random variables $(X^1, W^1, B^1, \mub^1, \mu^1, \alpha^1)$.
		
		\medskip
		We then define $\gamma^2$ as follows. Let $Z^\mathbf{1}_s := \Phi_s(X^1)$, for all $s \in [0,T]$, so that, by definition, $\P^1 \circ \big(Z^\mathbf{1}_t \big)^{-1} =[\nu_1]^{\circ}_t = [\nu_2]^{\circ}_t$. Next, let
		\[
			X^2_s
			:=
			\begin{cases}
				\psi_s \big(Z^\mathbf{1}_t, \eta \big), \; \mbox{\rm if}\; s \in[0, t],\\[0.5em]
				X^2_t + X^\mathbf{1}_{s} - X^\mathbf{1}_t, \;\mbox{if}\; s \in (t,T].
			\end{cases}
		\]
		It follows by the properties of $\psi$ and those of the updating function $\Phi$ that
		\begin{equation} \label{eq:equiv_g1g2}
			\P^1 \circ \big(X^2_{t \wedge \cdot} \big) = \nu_2(t),
			\; \mbox{\rm and}\;
			\Phi_s(X^2) = \Phi_s(X^1),\; s \in [t,T].
		\end{equation}
		Let $\mub^2_s := \Lc^{\P^1} \big((X^2_{s \wedge \cdot}, \alpha^\mathbf{1}_s)| \Gc^\mathbf{1}_s \big)$, $\mu^2_s := \Lc^{\P^1} \big(X^2_{s \wedge \cdot}| \Gc^\mathbf{1}_s \big)$, for $s \in [t,T]$, and $\gamma_2 
		    :=
		    \big( \Om^1, \Fc^1, \F^1, \P^1, \G^1, X^2, W^1, B^1, \mub^2, \mu^2, \alpha^1 \big).$
		Using \Cref{assum:Markov} and \eqref{eq:equiv_g1g2},
		we have $\gamma_2 \in \Gamma_W(t, \nu_2)$
		and $J(t, \gamma_2) = J(t, \gamma_1)$,
		 implying $V_W(t, \nu_1) = V_W(t, \nu_2)$.
	\end{proof}

	Now we provide the dynamic programming principle for the Markovian control problem under \Cref{assum:Markov}.

	\begin{corollary}
		Let {\rm \Cref{assum:Markov}} hold true, $t \in [0,T]$ and $\nu^{\circ} \in \Pc(E)$. 
		Let $ \tau^{\star}$ be a $\G^{\star}$--stopping time taking values in $[t, T]$ on $\Om^{\circ}$
		and $(\tau^{\gamma})_{\gamma \in \Gamma^{\circ}_W(t, \nu^{\circ})}$ be defined from $\tau^{\star}$ as in \eqref{eq:tau_from_taut},
		and $\tau$ be a $\G^{t, \circ}$--stopping time taking values in $[t, T]$ on $\Om^t$.
		Then one has the following dynamic programming results.
		
\medskip
		
		$(i)$
		The function $V_W^{\circ}:[0,T] \x \Pc(E) \longrightarrow \R \cup\{-\infty, \infty\}$ is upper semi--analytic and, with $Z^{\gamma}_s := \Phi_s(X^{\gamma}_{\cdot})$,
		\begin{equation} \label{eq:WeakDPP_Markov}
			V_W^{\circ}(t, \nu^{\circ}) 
			=
			\sup_{\gamma \in \Gamma_W^{\circ}(t, \nu^{\circ})} 
			\E^{\P^\gamma} \bigg[
				\int_t^{\tau^{\gamma}} L^{\circ}\big(s, Z^{\gamma}_s, [\mub^{\gamma}]^{\circ}_s , \alpha^{\gamma}_s\big) \mathrm{d}s 
				+ 
				V_W^{\circ}\big(\tau^{\gamma}, [\mu^{\gamma}]^{\circ}_{\tau^{\gamma}} \big) 
			\bigg].
		\end{equation}

		$(ii)$
		Let {\rm\Cref{assum:Lip}} hold true, then $V^{\B,\circ}_S:[0,T] \x \Pc(E) \longrightarrow \R \cup\{-\infty, \infty\}$ is upper semi--analytic,
		and with $Z^{\alpha}_s := \Phi_s(X^{\alpha}_{\cdot})$, one has
		\begin{equation} \label{eq:B-StrongDPP_Markov}
			V^{\B, \circ}_S(t, \nu^{\circ}) 
			=
			\sup_{\alpha \in \Ac^\B_2(t, \nu^{\circ})} 
			\E^{\P^t_{\nu}} \bigg[
				\int_t^{\tau}  L\big(s, Z^{\alpha}_s, [\mub^{\alpha}]^{\circ}_s, \alpha_s \big) \mathrm{d}s 
				+ 
				V^{\B, \circ}_S\big(\tau, [\mu^{\alpha}]^{\circ}_{\tau} \big) 
			\bigg].
		\end{equation}
		
		$(iii)$ 
		Let {\rm \Cref{assum:Lip} }and {\rm\Cref{assum:Growth}} hold,
		then
		$V^{\circ}_S(t, \nu^{\circ}) = V^{\circ}_W(t, \nu^{\circ})$, and with $Z^{\alpha}_s := \Phi_s(X^{\alpha}_{\cdot})$,
		\begin{equation} \label{eq:StrongDPP_Markov}
			V^{\circ}_S(t, \nu^{\circ}) 
			=
			\sup_{\alpha \in \Ac_2(t, \nu^{\circ})} 
			\E \bigg[
				\int_t^{\tau}  L\big(s, Z^{\alpha}_s, [\mub^{\alpha}]^{\circ}_s, \alpha_s\big) \mathrm{d}s 
				+ 
				V^{\circ}_S\big(\tau, [\mu^{\alpha}]^{\circ}_{\tau} \big) 
			\bigg].
		\end{equation}		
	\end{corollary}
	\begin{proof}
		We will only consider the case $V_W$, the arguments for $V_S$ and $V_S^\B$ will be the same.

\medskip
		
		Let $\llbracket \Vc \rrbracket:= \{ (t,\nu,\nu^\circ) \in [0,T] \x \Pc(\Cc^n) \x \Pc(E):[\nu]^{\circ}_{t}=\nu^\circ \}$.
		Notice that $\Phi: \Cc^n \longrightarrow C([0,T], E)$ is Borel, then $(t, \nu) \longmapsto [\nu]^{\circ}_{t}$ is also Borel,
		and hence $\llbracket \Vc \rrbracket$ is a Borel subset of $[0,T] \x \Pc(\Cc^n) \x \Pc(E)$. 
		Further, one has $V_W^\circ (t,\nu^\circ)=\sup_{(t,\nu,\nu^\circ) \in \llbracket \Vc \rrbracket} V_W(t,\nu)$ from \Cref{lemm:equality_Markov},
		and $V_W$ is upper semi--analytic by \Cref{thm:WeakDPP}.
		It follows by the measurable selection theorem (e.g. \cite[Proposition 2.17]{karoui2013capacities}) that $V_W^{\circ}:(t,\nu^\circ) \in [0,T] \x \Pc(E) \longrightarrow V_W^{\circ}(t,\nu^\circ) \in \R \cup \{-\infty, \infty \}$  is also upper semi--analytic.
		Finally, using the DPP results in \Cref{thm:WeakDPP}, it follows that
		\begin{align*}
			V_W^{\circ}(t, \nu^{\circ})
			=
			\sup_{	\nu\in \Vc(t,\nu^\circ)}
			V_W(t, \nu) 
			&=
			\sup_{\nu\in \Vc(t,\nu^\circ)}
			\sup_{\gamma \in \Gamma_W(t, \nu)} 
			\E^{\P^\gamma} \bigg[
				\int_t^{\tau^{\gamma}} L\big(s, X^\gamma_{s \wedge \cdot}, \mub^{\gamma}_s, \alpha^{\gamma}_s \big) \mathrm{d}s 
				+ 
				V_W\big(\tau^{\gamma}, \mu^{\gamma}_{\tau^{\gamma}} \big) 
			\bigg]
			\\
			&=
			\sup_{\nu\in \Vc(t,\nu^\circ)}
			\sup_{\gamma \in \Gamma_W(t, \nu)} 
			\E^{\P^\gamma}  \bigg[
				\int_t^{\tau^{\gamma}} L^{\circ}\big(s, Z^{\gamma}_s, [\mub^{\gamma}]^{\circ}_s, \alpha^{\gamma}_s\big) \mathrm{d}s 
				+ 
				V_W^{\circ}\big(\tau^{\gamma}, [\mu^{\gamma}]^{\circ}_{\tau^{\gamma}} \big) 
			\bigg]
			\\
			&=
			\sup_{\gamma \in \Gamma^\circ_W(t, \nu)} 
			\E^{\P^\gamma}  \bigg[
				\int_t^{\tau^{\gamma}} L^{\circ}\big(s, Z^{\gamma}_s, [\mub^{\gamma}]^{\circ}_s, \alpha^{\gamma}_s\big) \mathrm{d}s 
				+ 
				V_W^{\circ}\big(\tau^{\gamma}, [\mu^{\gamma}]^{\circ}_{\tau^{\gamma}} \big) 
			\bigg].
		\end{align*}
	\end{proof}
	
\subsection{Discussion: from dynamic programming to the HJB equation}
One of the classical applications of the DPP consists in giving some local characterisation of the value function, such as in proving that it is the viscosity solution of an HJB equation.
	This was achieved in \citeauthor*{pham2018bellman} \cite{pham2018bellman} for the control problem $V^\circ_S$ in the setting with $\sigma_0 \equiv 0$),
	and in \citeauthor*{pham2016dynamic} \cite{pham2016dynamic} for the control problem $V^{\circ,\B}_S$ ({\color{black} for $\Phi_t(\xb) := \xb(t), \; \mbox{\rm with}\; E = \R^n$}).
	It relies essentially on the notion of differentiability with respect to probability measures due to \citeauthor*{lions2007theorie} (see e.g. \cite{lions2007theorie} and \citeauthor*{cardaliaguet2010notes}'s notes \cite[Section $6$]{cardaliaguet2010notes}), 
	and It\=o's formula along a measure (see e.g. \citeauthor*{carmona2014master} \cite[Proposition $6.5$ and Proposition $6.3$]{carmona2014master}).
	We will now provide some heuristic arguments to derive the HJB equation from our DPP results for both $V^{\B,\circ}_S$ and $V_S^{\circ}$, with updating function $\Phi_t(\xb) = \xb(t)$.

\medskip	
	Let us first recall briefly the notion of the derivative, in sense of Fr\'echet, $\partial_{\nu} V(\nu)$ for a function $V: \Pc_2(\R^n) \longrightarrow \R$.
	Consider a probability space $(\Om,\Fc,\P)$ rich enough so that, for any $\nu \in \Pc_2(\R^n)$, there exists a random variable $Z: \Om \longrightarrow \R^n$ such that $\Lc^\P(Z)=\nu$. We denote by $\Lc^2(\Omega,\Fc,\P)$ the space of square--integrable random variables on $(\Om,\Fc,\P)$. Let $V: \Pc_2(\R^n) \longrightarrow \R$, we consider $\widetilde{V}: \Lc^2(\Om,\Fc,\P) \longrightarrow \R^n$, the lifted version of $V$, defined by $\widetilde{V}(X):= V(\Lc^\P(X))$.
	Recall that $\widetilde{V}$ is said to be continuously Fr\'echet differentiable, if there exists a unique continuous application $D\widetilde{V}: \Lc^2(\Om,\Fc,\P) \longrightarrow \Lc^2(\Om,\Fc,\P)$, 
	such that, for all $Z \in \Lc^2(\Om,\Fc,\P)$,
	\[
	    \lim_{\|Y\|_2 \to 0}
	    \frac{\big|\widetilde{V}(Z+Y)-\widetilde{V}(Z)-\E\big[Y^\top D\widetilde{V}(Z) \big] \big|}{\|Y\|_2},
	\]
	where $\|Y\|_2:=\E[|Y|^2]^{1/2}$ for any $Y\in\Lc^2(\Om,\Fc,\P)$.
	We say that $V$ is of class $C^1$ if $\widetilde{V}$ is continuously Fr\'echet differentiable,
	and denote for any $\nu\in \Pc_2(\R^n)$, $
	    \partial_{\nu}V(\nu)(Z):=D\widetilde{V}(Z),\;
	    \P\mbox{\rm --a.s.}, \; \mbox{\rm for any $Z\in\Lc^2(\Omega,\Fc,\P)$ such that} \;\Lc^\P(Z)=\nu.$
	Notice that one has $\partial_{\nu}V(\nu): \R^n\ni y  \longmapsto\partial_{\nu}V(\nu)(y) \in \R^n$ and this function belongs to 
	$\Lc^2(\R^n,\Bc(\R^n),\nu).$ Besides, the law of $D\widetilde{V}(Z)$ is independent of the choice of $Z$.
	Similarly, we also define the derivatives $\Pc_2(\R^n) \x \R^n\ni (\nu,y) \longmapsto \partial_y \partial_{\nu}V(\nu)(y) \in \S^n$ and
	\[
	 \Pc_2(\R^n) \x \R^n \x \R^n \ni	(\nu,y,y^\prime)
		\longmapsto
		\partial_{\nu}^2 V(\nu)(y,y^\prime):=\partial_{\nu}\big[\partial_{\nu}V(\nu)(y)\big](y^\prime) \in \S^n.
	\]
	In the following, we say $V$ is a `smooth function`, if all the above Fr\'echet derivatives are well defined and are continuous.
	
\subsubsection{HJB equation for the common noise strong formulation}	
	Let us consider the control problem $V^{\B,\circ}_S$ and repeat the arguments in \cite{pham2016dynamic} in a heuristic way.
	Given a ``smooth function'' $V: [0,T] \x \Pc_2(\R^n) \longrightarrow \R$, $(t,\nu) \in [0,T] \x \Pc_2(\R^n)$, and $\gamma \in \Gamma^{\B}_S(t,\nu),$
	it follows from It\^o's formula that, for $s\in[t,T]$,
	\begin{align}  \label{eq:itoFormula-V^B_S}
		V\big(s,\mu^{\gamma}_s \big)
		=&\
		V\big(t,\nu \big)
		+
		\int_t^s  \int_{\R^n} \Big(  \partial_t V\big(r,\mu^{\gamma}_r \big) + \partial_{\nu} V\big(r,\mu^{\gamma}_r \big)(y) \cdot b\big(r,y,\mu^{\gamma}_r \otimes \delta_{\alpha^\gamma_r},\alpha^\gamma_r \big)  \Big)\mu^{\gamma}_r(\mathrm{d}y)\mathrm{d}r\nonumber
		\\
		&+
		\frac{1}{2}   \int_t^s  \int_{\R^n}\mbox{Tr} \big[ \partial_x \partial_{\nu} V\big(r,\mu^{\gamma}_r \big)(y) \big( \sigma ^\top\sigma + \sigma_0^\top\sigma_0 \big)(r,y,\mu^{\gamma}_r \otimes \delta_{\alpha^\gamma_r},\alpha^\gamma_r) \big] \mu^{\gamma}_r(\mathrm{d}y)\mathrm{d}r \nonumber
		\\
		&+
		\frac{1}{2}   \int_t^s  \int_{(\R^n)^2}\mbox{Tr} \big[ \partial^2_{\nu}V\big(r,\mu^{\gamma}_r \big)(y,y^{\prime}) \sigma_0^\top(r,y,\mu^{\gamma}_r \otimes \delta_{\alpha^\gamma_r},\alpha^\gamma_r) \sigma_0(r,y^\prime,\mu^{\gamma}_r \otimes \delta_{\alpha^\gamma_r},\alpha^\gamma_r )  \big]
		\mu^{\gamma}_r(\mathrm{d}y) \mu^{\gamma}_r(\mathrm{d}y^\prime) \mathrm{d}r \nonumber
		\\
		&+
		\int_t^s  \int_{\R^n\times U}\partial_{\nu} V\big(r, \mu^{\gamma}_r \big)(x)\cdot \sigma_0(r,y,\mu^{\gamma}_r \otimes \delta_{\alpha^\gamma_r},\alpha^\gamma_r) \mu^{\gamma}_r(\mathrm{d}y)\mathrm{d}B_r.
	\end{align}
	As $\gamma \in \Gamma^{\circ,\B}_S(t,\nu)$, for Lebesgue--almost every $r\in[t,T]$, $\alpha^{\gamma}_r$ is a measurable function of $(B_u - B_t)_{u \in [t,r]}$.
	Considering piecewise constant control process, $\alpha^{\gamma}$ would be a deterministic constant on a small time horizon $[t, t+\eps]$.
	By replacing $V$ in \eqref{eq:itoFormula-V^B_S} by $V^{\B,\circ}_S$ 
	and taking supremum as in DDP \eqref{eq:B-strongDPP} (but over constant control processes), 
	this leads to the Hamiltonian
	\begin{align*} 
		H^\B[V] \big(t,\nu \big)
		:=
		\sup_{u \in U} \bigg\{&\int_{\R^n} \Big(\big(L+ [V]^1\big)\big(t,y,\nu \otimes \delta_{u}, u \big) \Big)\nu(\mathrm{d}y)+
		\int_{(\R^n)^2} [V]^2 \big(t,y,u,y^\prime, \nu \otimes \delta_{u}, u \big)\nu(\mathrm{d}y)\nu(\mathrm{d}y^\prime)
		\bigg\},
	\end{align*}
	where for any $(r,y,u,y^\prime,u^\prime,\nub)\in[0,T]\times\R^n\times U\times\R^n\times U\times \Pc(\R^n\times U)$
	\[
		[V]^1(r,y,\nub, u)
		:=
		\partial_{\nu} V(r,\nu )(y) \cdot b(r,y,\nub, u)
		+
		\frac{1}{2} \mbox{Tr} \big[ \partial_x \partial_{\nu} V(r,\nu)(y) ( \sigma ^\top\sigma + \sigma_0^\top\sigma_0 )(r,y,\nub, u)\big],
	\]
	and
	\[
		[V]^2(r,y,u,y^\prime,u^\prime,\nub)
		:=
		\frac{1}{2} \mbox{Tr} \big[ \partial^2_{\nu}V\big(r,\nu \big)(y,y^{\prime}) \sigma_0^\top(r,y,\nub, u){\sigma}_0(r,y^\prime,\nub,u^\prime)\big].
	\]
	Heuristically, $V^{\B, \circ}$ should satisfy the HJB equation
	\[
		- \partial_t V^{\B,\circ}_S(t, \nu) - H^\B \big[V^{\B,\circ}_S \big] \big(t,\nu \big) = 0,\; (t,\nu)\in[0,T)\times\Pc_2(\R^n),\; V^{\B,\circ}_S(T, \cdot)=g(\cdot).
	\]
	We refer to \cite{pham2016dynamic} for a detailed rigorous proof of the fact that $V^{\B, \circ}$ is a viscosity solution of the above HJB equation under some technical regularity conditions.

\subsubsection{HJB equation for the general strong formulation}	

	Similarly, for the control problem $V^\circ_S$, we consider a strong control rule $\gamma \in \Gamma^\circ_S(t,\nu)$,
	where for Lebesgue--almost every $r\in[t,T]$, the control process $\alpha^{\gamma}_r$ is a measurable function of both $(B_u - B_t)_{u \in [t,r]},$ $(W_u - W_t)_{u \in [t,r]}$ and $X_{t \wedge \cdot}$.
	As the control process $\alpha^{\gamma}$ is adapted to the filtration generated by $(X_{t \wedge \cdot}, W^{\gamma,t}, B^{\gamma,t})$,
	by considering adapted piecewise constant control processes, 
	the control process on the first small interval $[t, t+\eps)$ should be a measurable function of $X_{t \wedge \cdot}$.
	Similarly to \citeauthor*{pham2018bellman} \cite{pham2018bellman} in a non--common noise setting and by considering the It\^o formula \eqref{eq:itoFormula-V^B_S}, this would formally lead to the Hamiltonian
	\begin{align*} 
		H[V] \big(t,\nu \big)
		:=
		\sup_{a \in \L^2_\nu} \bigg\{ \int_{\R^n}\big( L+ [V]^1\big)\big(t,y,\nu \circ (\hat a)^{-1},a(y) \big)\nu(\mathrm{d}y) +
		\int_{(\R^n)^2} [V]^2 \big(t,y,a(y),y^\prime,a(y^\prime),\nu \circ (\hat a)^{-1} \big)\nu(\mathrm{d}y)\nu(\mathrm{d}y^\prime) \bigg\},
	\end{align*}
	where $\hat a: \R^n\ni x  \longmapsto (x,a(x)) \in \R^n \x U$,
	and $\L^2_\nu$ of all $\nu$--square integrable functions $a: (\R^n, \Bc(\R^n), \nu) \longrightarrow U$.
	Heuristically, $V^\circ_S$ should be a solution of the HJB equation
	\[
		- \partial_t V^\circ_S(t, \nu) - H[V^\circ_S] \big(t,\nu \big) = 0 ,\; (t,\nu)\in[0,T)\times\Pc_2(\R^n),\; V^{\B,\circ}_S(T, \cdot)=g(\cdot).
	\]

	As explained above, the difference between the HJB equations for $V^{\B, \circ}_S$ and $V^\circ_S$
	comes mainly from the fact that the control process $\alpha^{\gamma}$, for $\gamma \in \Gamma^\circ_S(t,\nu)$,
	depends on the initial random variable condition, which in turn modifies the Hamiltonian function which appears in the PDE.
	Finally, we also refer to \citeauthor*{wu2018viscosity} \cite{wu2018viscosity} for a discussion of the McKean--Vlasov control problem in a non--Markovian framework without common noise.

\section{Proofs of the main results}\label{sec:4}

	We now provide the proofs of our main DPP results in Theorems \ref{thm:WeakDPP}, \ref{thm:B-StrongDPP} and \ref{thm:StrongDPP},
	{\color{black} where a key ingredient is the measurable selection argument.}
	We will first reformulate the control problems on an appropriate canonical space in \Cref{subsec:Reform_canon},
	and then provide some technical lemmata for the problems formulated on the canonical space in \Cref{subsec:TechLemma},
	and finally give the proofs of the main results themselves in \Cref{subsec:Proofs}.

\subsection{Reformulation of the control problems on the canonical space}
\label{subsec:Reform_canon}

\subsubsection{Canonical space}
	In order to prove the dynamic programming results in \Cref{sec:DPP}, 
	we first reformulate the controlled McKean--Vlasov SDE problems on an appropriate canonical space. This is going to be achieved by the usual way, that is to say by considering appropriately defined controlled martingale problems.
	Recall that $n,$ $d$ and $\ell$ are the dimensions of the spaces in which $X$, $W$ and $B$ take values,
	$\Ub = U \cup \{\partial\}$ and that $\pi^{-1}$ maps $\R \cup \{\infty, -\infty\}$ to $\Ub$.
	Let us introduce a first canonical space by
	\[
		\Omh := \Cc^n \x \Cc \x \Cc^d \x \Cc^\ell
		~\mbox{with canonical process}~
		(\Xh, \Ah, \Wh, \Bh),
		~\mbox{and}~
		\alphah_t := \pi^{-1} \Big(\underset{n\rightarrow+\infty}{ \overline{\lim}} n \big(\Ah_t - \Ah_{0 \vee (t-1/n)}\big) \Big),
		~t \in [0,T].
	\]
	Denote by $C( [0,T], \Pc(\Omh))$ be the space of all continuous paths on $[0,T]$ taking values in $\Pc(\Omh)$, which is a Polish space for the uniform topology (see e.g. \cite[Lemmata 3.97, 3.98, and 3.99]{aliprantis2006infinite}), 
	we introduce a second canonical space by
	\[
		\Omb
		:=
		\Omh \x C \big([0,T], \Pc(\Omh) \big),
		~\mbox{with canonical process}~
		(X, A, W, B, \muh)
		~\mbox{and canonical filtration}~\Fb = (\Fcb_t)_{0 \le t \le T}.
	\]
	Let $\Fcb := \Bc(\Omb)$ be the Borel $\sigma$--field on $\Omb$.
	Notice that for any $t\in[0,T]$, $\muh_t$ is a probability measure on $\Omh$, 
	we then define two processes $\mu = (\mu_t)_{0 \le t \le T}$ and $\mub = (\mub_t)_{0 \le t \le T}$ on $\Omb$ by
	\[
		\mub_t := \muh_t \circ \big(\Xh_{t \wedge \cdot}, \alphah_t \big)^{-1},
		~
		\mu_t := (\muh_t) \circ \big(\Xh_{t \wedge \cdot} \big)^{-1}.	
	\]
	Define also the $\Ub$--valued process $\alphab = (\alphab_t)_{0 \le t \le T}$ on $\Omb$ by
	\[
		\alphab_t := \pi^{-1} \Big( \underset{n\rightarrow+\infty}{ \overline{\lim}} n \big(A_t - A_{0 \vee (t-1/n)}\big) \Big),\; t\in[0,T].
	\]
	Finally, for any $t \in [0,T]$, we introduce the processes $W^t := (W^t_s)_{s \in [0,T]}$ and $B^t := (B^t_s)_{s \in [0,T]}$ by
	\[
		B^t_s := B_{s \vee t} - B_t,
		~\mbox{\rm and}~
		W^t_s := W_{s \vee t} - W_t,\; s\in[0,T],
	\]
	and the filtration $\Gb^t := (\Gcb^t_s)_{0 \le s \le T}$ by
	\begin{equation} \label{eq:def_Gt}
		\Gcb^t_s := 
		\begin{cases} 
			\{\emptyset, \Om \}, \; \mbox{\rm if}\; s \in [0,t),\\
			\sigma \big((B^t_r, \muh_r): r \in [0, s] \big),  \; \mbox{\rm if}\; s \in [t,T].
		\end{cases}
	\end{equation}

\subsubsection{Controlled martingale problems on the canonical space}
	We now reformulate the strong/weak control problem as a controlled martingale problem on the canonical space $\Omb$,
	where a control (term) can be considered as a probability measure on $\Omb$.
	To this end, let us first introduce the corresponding generator.
	Given the coefficient functions $b,$ $\sigma$ and $\sigma_0$,
	for all $(t, \xb,\wb,\bb, \nub, u) \in [0,T] \x \Cc^n \x \Cc^d \x \Cc^{\ell} \x \Pc(\Cc^n \x U) \x U$, let
	\begin{equation} \label{cond:coeff-b}
		\bar b \big(t, (\xb,\wb,\bb), \nub, u \big)
		:=
		\big(b,0_d,0_\ell \big)(t,\xb, \nub, u),
	\end{equation}
	and
	\begin{equation} \label{cond:coeff-a}
		\bar a \big(t,(\xb,\wb,\bb), \nub, u \big)
		:=
		\begin{pmatrix} 
			\sigma & \sigma_0 \\ 
			I_{d \x d} & 0_{d \x \ell} \\ 
			0_{\ell \x d} & I_{\ell \x \ell} 
		\end{pmatrix}
		\begin{pmatrix} 
			\sigma & \sigma_0 \\ 
			I_{d \x d} & 0_{d \x \ell} \\ 
			0_{\ell \x d} & I_{\ell \x \ell}
		\end{pmatrix}^{\top}
		(t,\xb,  \nub, u),
	\end{equation}
	and then introduce the generator $\Lc$, for all $\varphi \in C^2_b(\R^{n+d+\ell})$,
	\[
		\Lc_t \varphi \big( \xb,\wb,\bb,  \nub, u \big)
		:=
		\sum_{i=1}^{n+d+\ell} \bar b_{i}(t,(\xb,\wb,\bb), \nub, u) \partial_{i} \varphi(\xb_t,\wb_t,\bb_t)
		+ 
		\frac{1}{2}\sum_{i,j=1}^{n+d+\ell} \bar a_{i,j}(t,(\xb,\wb,\bb), \nub, u) \partial^2_{i,j} \varphi(\xb_t,\wb_t,\bb_t).
	\]
	We next define a process $|\Sb| = \big(|\Sb|_t\big)_{0 \le t \le T}$ by
	\[
		|\Sb|_t := \int_0^t \big( |\bar b| + | \bar a| \big) (s, X, W, B, \mub_s, \alphab_s) \mathrm{d}s,
	\]
	and then, for all $\varphi \in C^2_b(\R^{n+d+\ell})$,
	let $\Sb^{\varphi} = (\Sb^{\varphi}_t)_{t \in [0,T]}$ be defined by
	\begin{equation} \label{eq:associate-martingale}
		\Sb^{\varphi}_t
		:=
		\varphi(X_t, W_t, B_t) 
		-
		\int_0^t  \Lc_s \varphi \big(X, W, B,  \mub_s, \alphab_s \big) \mathrm{d}s,~ t\in[0,T],
	\end{equation}
	where for $\phi: [0,T] \to \R$, $\int_0^t \phi(s) \mathrm{d}s := \int_0^t \phi^+(s) \mathrm{d}s - \int_0^t \phi^-(s) \mathrm{d}s$ with the convention $\infty - \infty = - \infty$.
	Notice that on $\{|\Sb|_T < \infty\}$, the process $\Sb^{\varphi}$ is $\R$--valued.
	To localise the process $\Sb^{\varphi}$, we also introduce, for each $m \ge 1$,
	\begin{equation} \label{eq:def_tau_m}
		\tau_m := \inf \big\{ t : |\Sb|_t \ge m \big\},
		~\mbox{and}~
		S^{\varphi, m}_t := \Sb^{\varphi}_{t \wedge \tau_m} = \Sb^{\varphi}_t \mathbf{1}_{\{\tau_m \ge t\}} + \Sb^{\varphi}_{\tau_m} \mathbf{1}_{\{\tau_m < t\}},
		~t \in [0,T].
	\end{equation}
	Notice that the process $|\Sb|$ is left--continuous, $\tau_m$ is a $\F^+$--stopping time on $\Omb$, and $S^{\varphi, m}$ is an $\F$--adapted uniformly bounded process.

	\begin{definition} \label{def:weak_rule}
		Let  $(t, \nuh) \in [0,T] \x \Pc(\Omh)$. 
		A probability $\Pb$ on $(\Omb, \Fcb)$ is called a weak control rule with initial condition $(t,\nuh)$ if
		\begin{enumerate}[label={$(\roman*)$}]
		
			\item the process $\alphab = (\alphab_s)_{t \le s \le T}$ satisfies
			\[
				\Pb \big[ \alphab_s \in U \big] = 1, \; \mbox{\rm for Lebesgue--a.e.}\; s \in [t,T],
				\; \mbox{\rm and}\;
				\E^{\Pb} \bigg[ \int_t^T \big( \rho(u_0, \alphab_s) \big)^p \mathrm{d}s \bigg] < \infty;
			\]
			\item the process $\muh = (\muh_s)_{0 \le s \le T}$ satisfies
			\begin{equation} \label{eq:property_muh}
				\muh_s 
				=
				\Pb  \circ (X_{s \wedge \cdot} , A_{s \wedge \cdot}, W, B_{s \wedge \cdot} )^{-1} \mathbf{1}_{\{s \in [0,t]\}} 
				+ 
				\Pb^{\Gcb_T^t} \circ (X_{s \wedge \cdot} , A_{s \wedge \cdot}, W,B_{s \wedge \cdot} )^{-1} \mathbf{1}_{\{s \in (t, T]\}},\; \Pb-{\rm a.s.}
			\end{equation}
			with $\Pb  \circ (X_{t \wedge \cdot} , A_{t \wedge \cdot}, W_{t \wedge \cdot}, B_{t \wedge \cdot} )^{-1} = \nuh(t);$

			\item $\E^{\Pb}[ \|X\|^p] < \infty$, $\Pb[|\Sb|_T < \infty] = 1$,
			the process $\Sb^{\varphi}$ is an $(\Fb, \Pb)$--local martingale on $[t,T]$, for all $\varphi \in C^2_b \big(\R^n \x \R^d \x \R^{\ell} \big)$.

		\end{enumerate}
	\end{definition}
	Given $\nu \in \Pc(\Cc^n)$, we denote by $\Vc(\nu)$ the collection of all probability measures $\nuh \in \Pc(\Omh)$ such that $\nuh \circ \Xh^{-1} = \nu$,
	and let
	\[
		{\widehat{\Pc}}_W(t, \nuh)
		:=
		\big\{ \mbox{All weak control rules}\; \P \;\mbox{with initial condition}\; (t, \nuh) \big\},\; \mbox{\rm and}\;
		\Pcb_W(t, \nu)
		:=
		\bigcup_{\nuh\in \Vc(\nu)}
		\widehat{\Pc}_W(t, \nuh).
	\]

	\begin{remark} \label{rem:def_weak_rule}
		Let $\Pb \in \Pcb_W(t, \nu)$ for some $t \in [0,T]$ and $\nu \in \Pc(\Cc^n)$.
		Notice that for $s \in (t,T]$, $\muh_s$ is $\Gcb^t_s$--measurable, then by \eqref{eq:property_muh}, one has
		\[
			\muh_s =\Pb^{\Gcb_s^t} \circ (X_{s \wedge \cdot} , A_{s \wedge \cdot}, W,B_{s \wedge \cdot} )^{-1},~\Pb \mbox{\rm--a.s.}
		\]
		Further, as the canonical process $(\muh_s)_{s \in [0,T]}$ is continuous, it follows that
		\[
			\Lc^{\Pb} \big(X_{t \wedge \cdot}, A_{t \wedge \cdot}, W,B_{t \wedge \cdot} \big)
			=
			\muh_t
			=
			\lim_{s \searrow t} \muh_s
			=
			\lim_{s \searrow t} \Lc^{\Pb} \big( \big(X_{s \wedge \cdot}, A_{s \wedge \cdot}, W,B_{s \wedge \cdot} \big) \big| \Gcb^t_T \big)
			=
			\Lc^{\Pb} \big( \big(X_{t \wedge \cdot}, A_{t \wedge \cdot}, W,B_{t \wedge \cdot} \big) \big| \Gcb^t_T \big),
			~\Pb \mbox{\rm --a.s.}
		\]
		This implies that $\Fc_t \vee \sigma(W) = \sigma(X_{t \wedge \cdot},A_{t \wedge \cdot}, W,B_{t \wedge \cdot})$ is independent of $\Gcb^t_T$, which is consistent with the conditions in \Cref{def:weak_control}.
		
\medskip
		
		Finally, under $\Pb$, $(\muh_s)_{s \in [0,t]}$ is completely determined by $\nuh(t)$. More precisely, one has
		\[
			\muh_t(\mathrm{d}\xb,\mathrm{d}a,\mathrm{d}\wb,\mathrm{d}\bb)
			=
			\int_{\Omh \x \Cc^d} \delta_{(\xb^\prime,a^\prime,\wb^\prime \oplus_t \wb^\star,\bb^\prime)}(\mathrm{d}\xb,\mathrm{d}a,\mathrm{d}\wb,\mathrm{d}\bb)
			\;\nuh(t)(\mathrm{d}\xb^\prime,\mathrm{d}a^\prime,\mathrm{d}\wb^\prime,\mathrm{d}\bb^\prime)  \Lc^{\Pb} \big(W^{t} \big)(\mathrm{d}\wb^\star),
			~\Pb \mbox{\rm--a.s.}
		\]
		where $W^t$ is a $(\Fb, \Pb)$--Brownian motion on $[t,T]$, by the martingale problem property in \Cref{def:weak_rule}.
	\end{remark}

	\begin{definition} \label{def:PS}
		Let $(t, \nu) \in [0,T] \x \Pc_2(\Cc^n)$. 
		A probability $\Pb$ on $(\Omb, \Fcb)$ is called a strong control rule $($resp. $\B$--strong control rule$)$ with initial condition $(t,\nu)$, 
		if $\Pb  \in \Pcb_W(t,\nu)$ and moreover there exists some Borel measurable function $\phi: [0,T] \x \Om^t \longrightarrow U$ $\big($resp. $\phi: [0,T] \x \Cc^\ell_{t,T} \longrightarrow U\big)$ such that
		\[
			\alphab_s 
			=
			\phi\big(s, X_{t \wedge \cdot},  W^t_{s \wedge \cdot}, B^t_{s \wedge \cdot}\big)
			\;
			\big(\mbox{\rm resp.}\; \phi\big(s,B^t_{s \wedge \cdot}\big)\big), \; 
			\Pb\mbox{\rm--a.s., for all}\; s \in [t,T].
		\]
	\end{definition}
	Let us then denote by $\Pcb_S(t,\nu)$ (resp. $\Pcb_S^\B(t,\nu)$) the collection of all strong (resp. $\B$--strong) control rules with initial condition $(t,\nu)$.

\subsubsection{Equivalence of the reformulation}

	We now show that every strong/weak control (term) induces a strong/weak control rule on the canonical space,
	and any strong/weak control rule on the canonical space can be induced by a strong/weak control (term).

	\begin{lemma}\label{lemma:equivalence}
		$(i)$ Let $t \in [0,T]$ and $\nu \in \Pc(\Cc^{n})$.
		Then for every $\gamma \in \Gamma_W(t,\nu)$, one has
		\begin{equation} \label{eq:ctrl_repres}
			\Pb^\gamma
			:=
			\P^{\gamma}
			\circ 
			\big( X^{\gamma}, A^{\gamma}, W^{\gamma}, B^{\gamma}, \muh^{\gamma} \big)^{-1}
			\in
			\Pcb_W(t, \nu).
		\end{equation}
		Conversely, given $\Pb \in \Pcb_W(t, \nu)$, 
		there exists some $\gamma \in \Gamma_W(t,\nu)$
		such that $\P^{\gamma} \circ \big( X^{\gamma}, A^{\gamma}, W^{\gamma}, B^{\gamma}, \muh^{\gamma} \big)^{-1} = \Pb$.
		
	\medskip
		
		$(ii)$ Let $t \in [0,T]$, $\nu \in \Pc_2(\Cc^n)$, and {\rm\Cref{assum:Lip}} hold true.
		Then for every $\gamma \in \Gamma_S(t,\nu)$ $\big($resp. $ \Gamma^\B_S(t,\nu)\big)$, one has
		\begin{equation} 
			\Pb^\gamma
			:=
			\P^{\gamma}
			\circ 
			\big( X^{\gamma}, A^{\gamma}, W^{\gamma}, B^{\gamma}, \muh^{\gamma} \big)^{-1}
			\in
			\Pcb_S(t, \nu)\; 
			\big(\mbox{\rm resp.}\; 
				 \Pcb^{\B}_S(t,\nu)
			\big).
		\end{equation}
		Conversely, given $\Pb \in \Pcb_S(t, \nu)$ 
		$\big(\mbox{\rm resp.}\; \Pcb^{\B}_S(t,\nu) \big)$, 
		there exists some $\gamma \in \Gamma_S(t,\nu)$ $\big( \mbox{\rm resp.}\; \Gamma^\B_S(t,\nu)\big)$
		such that $\P^{\gamma} \circ \big( X^{\gamma}, A^{\gamma}, W^{\gamma}, B^{\gamma}, \muh^{\gamma} \big)^{-1} = \Pb$.
	\end{lemma}

\begin{proof}
		$(i)$
		First, let $\gamma \in \Gamma_W(t,\nu)$
		and $\Pb^{\gamma}:=\P^{\gamma}\circ \big( X^{\gamma}, A^{\gamma},W^{\gamma}, B^{\gamma}, \muh^{\gamma}\big)^{-1}$. 
		First,  it is straightforward to check that
		\[
			\Pb^{\gamma} [\alphab_s \in U]=1,\;
			\mbox{\rm for}\; \mathrm{d}t\mbox{\rm --a.e.}\; s \in [t,T],
			~\E^{\Pb^{\gamma}} \big[ \|X\|^p \big] < \infty
			~\mbox{\rm and}~
			\E^{\Pb^{\gamma}} \bigg[\int_t^T (\rho(u_0,\alphab_s))^p \mathrm{d}s \bigg] < \infty.
		\]
		Further, as the integrals in \eqref{eq:MKV_SDE_weak} are well defined, one has $|\Sb|_T < \infty$, $\Pb^{\gamma}$--a.s.
		Moreover, by It\^o's formula, the process $\big( \Sb^{\varphi}_s \big)_{s \in [t,T]}$ defined in \eqref{eq:associate-martingale} is an $(\Fb,\Pb^{\gamma})$--local martingale, for every $\varphi \in C^2_b \big(\R^n \x \R^d \x \R^{\ell} \big)$.
		 
		 \vspace{0.5em}
		 
		Next, notice that $B^{t,\gamma}$ and $\muh^\gamma$ are adapted to $\G^\gamma$, one has, for all $(s,\beta,\psi) \in (t,T]\times C_b\big(\Omh \big)\times C_b\big(\Cc^\ell \x C([0,T]; \Pc(\Omh))\big)$,
		\begin{align*}
		    \E^{\Pb^{\gamma}} \big[\langle \beta,\muh_s \rangle \psi \big(B^{t}_{T \wedge \cdot},\muh_{T \wedge \cdot} \big) \big] 
		    &=
		    \E^{\P^\gamma} \big[\langle \beta,\muh^\gamma_s \rangle \psi \big(B^{\gamma,t}_{T \wedge \cdot},\muh^\gamma_{T \wedge \cdot} \big) \big]
		    =
		    \E^{\P^\gamma} \big[\langle \beta,\Lc^{\P^\gamma} \big( \big( X^\gamma_{s \wedge \cdot}, A^\gamma_{s \wedge \cdot}, W^\gamma, B^\gamma_{s \wedge \cdot} \big) \big| \Gc^\gamma_T\big) \rangle \psi \big(B^{\gamma,t}_{T \wedge \cdot},\muh^\gamma_{T \wedge \cdot} \big) \big]
		    \\
		    &=
		    \E^{\P^\gamma} \big[\beta \big( X^\gamma_{s \wedge \cdot}, A^\gamma_{s \wedge \cdot},W^\gamma,B^\gamma_{s \wedge \cdot} \big) \psi \big(B^{\gamma,t}_{T \wedge \cdot},\muh^\gamma_{T \wedge \cdot} \big) \big]
		   \!=\!
		   \E^{\Pb^{\gamma}} \big[\beta \big( X_{s \wedge \cdot}, A_{s \wedge \cdot}, W, B_{s \wedge \cdot} \big) \psi \big(B^{t}_{T \wedge \cdot},\muh_{T \wedge \cdot} \big) \big]
		    \\
		    &=
		    \E^{\Pb^{\gamma}} \big[ \big\langle \beta,\Lc^{\Pb^{\gamma}} \big( \big( X_{s \wedge \cdot}, A_{s \wedge \cdot}, W,B_{s \wedge \cdot} \big) \big| \Gcb^t_T\big) \big\rangle 
		    	\psi \big(B^{t}_{T \wedge \cdot},\muh_{T \wedge \cdot} \big) \big].
		\end{align*}
		This implies that $\muh_s=\Lc^{\Pb^{\gamma}} \big( \big( X_{s \wedge \cdot}, A_{s \wedge \cdot}, W^\gamma, B^\gamma_{s \wedge \cdot} \big) \big| \Gcb^t_T\big)$, $\Pb^{\gamma}$--a.s. for all $s \in (t,T]$.
		By the same argument and using the fact that $\Fc^{\gamma}_t \vee \sigma(W^{\gamma})$ is independent of $\Gc^{\gamma}_T$, 
		one can easily check that $\muh_s = \Lc^{\Pb^{\gamma}}(X_{s \wedge \cdot}, A_{s \wedge \cdot}, W_{\cdot}, B_{s \wedge \cdot})$ for $s \in [0,t]$,
		and that $\Pb^{\gamma} \circ X_{t \wedge \cdot}^{-1} = \nu$.
		This implies that $\Pb^{\gamma} \in \Pcb_W(t,\nu)$.

		\medskip 

		Assume in addition that $\gamma \in \Gamma_S(t,\nu)$ so that $\alpha^{\gamma}$ is $\F^{\gamma}$--predictable.
		Then there exists a Borel measurable function $\phi: [t,T] \x \Om^t \longrightarrow U$ such that
		$\alpha^{\gamma}_s=\phi \big(s,X^{\gamma}_{t \wedge \cdot},W^{\gamma,t}_{s \wedge \cdot},B^{\gamma, t}_{s \wedge \cdot} \big)$, for all $s \in [t,T]$, $\P^{\gamma}$--a.s. (see e.g. \citeauthor*{claisse2016pseudo} \cite[Proposition 10]{claisse2016pseudo}).
		This implies that $\alphab_s=\phi \big(s,X_{t \wedge \cdot},W^{t}_{s \wedge \cdot},B^{t}_{s \wedge \cdot} \big)$, $\Pb^{\gamma}$--a.s. for all $s \in [t,T]$,
		and it follows that $\Pb^{\gamma} \in \Pcb_S(t, \nu)$.
		
		\medskip 
		
		\noindent $(ii)$ Let $\Pb \in \Pcb_W(t,\nu)$ for some $\nu \in \Pc(\Cc^n)$.
		By \citeauthor*{stroock2007multidimensional} \cite[Theorem 4.5.2]{stroock2007multidimensional}, one knows that $(W, B)$ are $(\Fb, \Pb)$--Brownian motions on $[t,T]$,
		and 
		\begin{align*}
			X_s 
			&=
			X_t + \int_t^s b(r,   X_{\cdot}, \mub_r, \alphab_r) \mathrm{d}r 
			+ \int_t^s \sigma(r,   X_{\cdot}, \mub_r,   \alphab_r) \mathrm{d}   W_r 
			+ \int_t^s \sigma_0(r,   X_{\cdot}, \mub_r,   \alphab_r) \mathrm{d}   B_r, 
			\;   \Pb \mbox{\rm--a.s.},
		\end{align*}  
		Moreover, with the filtration $\Gb^t$ defined in \eqref{eq:def_Gt}, and in view of \Cref{rem:def_weak_rule}, it is straightforward to check that
		\[
			\gamma := \big(   \Omb,   \Fcb,   \Fb,  \Pb,   \Gb^t,   X,   W,   B, \mub, \mu,   \alphab \big)
			\in
			\Gamma_W(t,\nu).
		\]

		If, in addition, $\P \in \Pcb_S(t,\nu)$, so that $  A$ is a $(\sigma \big(  X_{t \wedge r \wedge \cdot},  W^t_{r \wedge \cdot},   B^t_{r \wedge \cdot} \big))_{r \in [t,T]}$--adapted continuous process.
		Using \Cref{corollary_doubleSDE} and the fact that $\muh_s= \Lc^\P \big((  X_{s \wedge \cdot},  A_{s \wedge \cdot},   W,   B_{s \wedge \cdot}) \big|   B^t_{r \in [t,s]},\muh_{s \wedge \cdot} \big)$, $\P$--a.s., for all $s \in [t,T]$, 
		one can deduce  that $\muh_s= \Lc^\P \big((  X_{s \wedge \cdot},  A_{s \wedge \cdot},   W,   B_{s \wedge \cdot}) \big|   B^t_{r \in [t,s]}\big)$, $\P$--a.s., for all $s \in [t,T]$.
		Let $\tilde \G^t$ be the filtration generated by $B^t$, $\tilde \F^t$ be the filtration generated by $(X_{t \wedge \cdot}, W^t, B^t)$,
		and $\tilde \G^{t,\Pb}$, $\tilde \F^{t, \Pb}$ be the corresponding $\Pb$--augmented filtrations.
		Then $\muh$ is $\tilde \G^{t,\Pb}$--predictable, and $X$ is $\tilde \F^{t, \Pb}$--predictable.
		Then it follows that
		\[
			\gamma' :=  \big( \Omb,   \Fcb,   \tilde \F^{t, \Pb},  \Pb,   \tilde \G^{t, \Pb},   X,   W^t,   B^t, \mub, \mu, \alphab \big) \in \Gamma_S(t,\nu).
		\]

		\noindent $(iii)$ Finally, the results related to $\Pcb^{\B}_S(t, \nu)$ and $\Gamma^{\B}_S(t,\nu)$ can be deduced by almost the same arguments as for $\Pcb_S(t, \nu)$ and $\Gamma_S(t,\nu)$.
	\end{proof}

	\begin{remark}
		From {\rm\Cref{lemma:equivalence}}, we can easily deduce that
		under {\rm\Cref{assum:Lip}}, for $\Pb \in \Pcb_S(t,\nu)$ or $\Pb \in \Pcb_S^\B(t,\nu)$, 
		the canonical process $\muh$ satisfies
		\[
		    \muh_s
		    =
		    \Lc^\Pb \big( (X_{s \wedge \cdot},A_{s \wedge \cdot}, W, B_{s \wedge \cdot}) \big| B^t_{s \wedge \cdot} \big)
		    =
		    \Lc^\Pb \big( (X_{s \wedge \cdot},A_{s \wedge \cdot}, W,B_{s \wedge \cdot}) \big| B^t \big),\; \Pb\mbox{\rm --a.s.},\; \mbox{\rm for all}\;s \in [0,T].
		\]
	\end{remark}

	A direct consequence of \Cref{lemma:equivalence} is that we can now reformulate equivalently the weak/strong formulation of the McKean--Vlasov control problem on the canonical space. 

	\begin{corollary}
		Let $(t, \nu) \in [0,T] \x \Pc(\Cc^n)$, one has 
		\begin{equation} \label{eq:def_JP}
			V_W(t, \nu) = \sup_{\Pb \in \Pcb_W(t, \nu)} J(t,\Pb),		
			~\mbox{where}~
			J(t, \Pb) 
			:=
			\E^{\Pb} \bigg[
				\int_t^T L\big(s, X, \mub_s, \alphab_s \big) \mathrm{d}s 
				+
				g\big( X, \mu_T \big)
			\bigg].
		\end{equation}
		Moreover, when {\rm\Cref{assum:Lip}} holds true and $\nu \in \Pc_2(\Cc^n)$, one has
		\[
			V_S(t, \nu) 
			=
			\sup_{\Pb \in \Pcb_S(t, \nu)} J(t,\Pb),
			~\mbox{\rm and}~
			V^\B_S(t, \nu) 
			=
			\sup_{\Pb \in \Pcb_S^\B(t, \nu)} J(t,\Pb).
		\]
	\end{corollary}

\subsection{Technical lemmata}
\label{subsec:TechLemma}

	We provide in this section some technical results related to the sets $\Pch_W(t, \nuh)$.

	\begin{lemma} \label{graph:weak-formulation}
		Both graph sets
		\[
			\llbracket \widehat{\Pc}_W \rrbracket
			:=
			\big\{(t, \nuh, \Pb) : \Pb \in \widehat{\Pc}_W(t, \nuh)
			\big\}
			\;\mbox{\rm and}\;
			\llbracket \Pcb_W \rrbracket
			:=
			\big\{(t, \nu, \Pb) : \Pb \in \Pcb_W(t, \nu)
			\big\}
		\]
		are analytic subsets of, respectively, $[0,T] \x \Pc(\Omh) \x \Pc(\Omb)$
		and $[0,T] \x \Pc(\Cc^n) \x \Pc(\Omb)$.
		Moreover, the value function 
		\[
		    V_W: [0,T] \x \Pc(\Cc^n) \longrightarrow \R \cup \{- \infty,\infty \},
		\]
		is upper semi--analytic.
	\end{lemma}
	\begin{proof}
	For $0\leq r \le s \leq T$, $m \ge 1$,
	$\chi \in C_b(\Omb)$,
	$\varphi \in C^2_b(\R^n \x \R^d \x \R^\ell)$, 
	$\phi \in C_b(\Omh)$, 
	$\psi \in C_b(\Cc^\ell \x C([0,T];\Pc(\Omh)))$, we define
	\[
		\xi_{r \wedge \cdot}:= \chi \big( X_{r \wedge \cdot}, A_{r \wedge \cdot}, W_{r \wedge \cdot}, B_{r \wedge \cdot}, \muh_{r \wedge \cdot} \big),
	\] 
	and the following Borel measurable subsets of $[0,T] \x \Pc(\Omh) \x \Pc(\Omb)$:
	\begin{align*}
		K^1
			&:=\bigg\{ (t,\nuh,\Pb):
				\int_t^T \Pb [\alphab_\theta \in U] \mathrm{d} \theta = T-t,
				~\E^{\P} \big[ \|X\|^p \big] < \infty,
				~\E^{\P} \bigg[\int_t^T  (\rho(u_0, \alphab_\theta))^p \mathrm{d}\theta \bigg] < \infty
			\bigg\}, \\
		K^{2,m}_{r,s}[\chi, \varphi]
			&:=\big\{ (t,\nuh,\P):
				\E^{\P}\big[\Sb_r^{\varphi, m} \xi_{r \wedge \cdot}\big]
				=
				\E^{\P}\big[\Sb_{s}^{\varphi, m} \xi_{r \wedge \cdot}\big]				
			\big\}, \\
		K^3_{s}[\phi]
			&:= \Big\{ (t,\nuh,\P):
				\E^{\P} \big[ \big|\langle \phi, \muh_{t \wedge s}\rangle
				-
				\E^{\P}\big[\phi(X_{[t \wedge s] \wedge \cdot},A_{[t \wedge s] \wedge \cdot},W,B_{[t \wedge s] \wedge \cdot})\big] \big| \big]=\E^{\P} \big[ \big|\langle \phi, \muh_t(t) \rangle
				-
				\langle \phi, \nuh(t) \rangle \big| \big]=0
			\Big\},\\
		K^4_s[\phi,\psi]
			&:=\Big\{ (t,\nuh,\P):
				\E^{\P}\big[\langle \phi, \muh_{t \vee s} \rangle \psi(B^t, \muh)\big]
				= 
				\E^{\P}\big[\phi(X_{[t \vee s] \wedge \cdot},A_{[t \vee s] \wedge \cdot},W,B_{[t \vee s] \wedge \cdot}) \psi(B^t, \muh)\big]
			\Big\}.
	\end{align*}
	The above Borel measurable sets allow to characterise the graph set $\llbracket \widehat{\Pc}_W \rrbracket$.
	Indeed, $K^1$ contains the probabilities on $\Omb$ such that the canonical element $\alphab$ takes its values in $U$ and not in $U \cup \{ \partial \}$,
	$K^{2,m}_{r,s}[\chi,\varphi]$ reduces the set $\Pc(\Omb)$ to the set of probabilities on $\Omb$ that solves a (local) martingale problem,
	while the probabilities which satisfy the `fixed point property`, i.e. the canonical process $\muh$ is equal to the conditional distribution of canonical process $(X, A, W, B)$, are contained in $K^3_s[\phi]$ and $K^4_s[\phi,\psi]$.

\medskip
	Let us consider a countable dense subset $\X$ of $\big(r, s, m, \chi, \varphi, \phi,\psi \big)$ in 
	\begin{align*}
		[0,T]^2 \x \N \x C_b(\Omb) \x C^2_b(\R^n \x \R^d \x \R^{\ell}) \x C_b(\Omh)
		\x \Cc_b(\Cc^\ell \x C([0,T];\Pc(\Omh))),
	\end{align*}
	where $0\leq r \le s\leq T$.
	By the above remarks, it is straightforward to check that
	\[
		\llbracket \widehat{\Pc}_W \rrbracket
		=
		\bigcap_{\X}
		\Big(
			K^1[h] \cap K^{2,m}_{r,s}[ \chi, \varphi]  \cap K^3_s[\phi] \cap K^4_s[\phi, \psi] \cap K^5_s[\phi, \psi]
		\Big),
	\]
	and hence it is a Borel subset of $[0,T] \x \Pc(\Omh) \x \Pc(\Omb)$. Furthermore, since $\Pc(\Omh)\ni\nuh  \longmapsto \nuh \circ (X)^{-1} \in \Pc(\Cc^n)$ is continuous, the set
	\[
		\llbracket \Pcb_W \rrbracket
	=
		\big\{
			(t, \nu, \Pb) : (t, \nuh, \Pb) \in \llbracket \widehat{\Pc}_W \rrbracket,\; \nuh \circ (X)^{-1} = \nu
		\big\},
	\]
	is an analytic subset of $[0,T] \x \Pc(\Cc^n) \x \Pc(\Omb)$. Finally, use the (analytic) measurable selection theorem 
	{\color{black} (see e.g. \citeauthor*{karoui2013capacities} \cite[Proposition 2.17]{karoui2013capacities}),}
	it follows that 
	\[
		V_W(t,\nu)
		=
		\sup_{(t,\nu,\Pb) \in \llbracket \Pcb_W \rrbracket} J(t,\Pb),
	\]
	is upper semi--analytic as desired.
\end{proof}

	We next prove a stability result w.r.t. the ``conditioning'' of $\Pch_W(t,\nuh)$.	
	\begin{lemma} \label{lemm:conditioning}
		Let $(t, \nuh) \in [0,T] \x \Pc(\Omh)$,
		$\Pb \in \Pch_W(t, \nuh)$, $\taub$ be a $\Gb^t$--stopping time taking values in $[t, T]$, and 
		$ \big( \Pb^{\Gcb^t_{\taub}}_{\omb} \big)_{\omb \in \Omb}$ be a family of {\rm r.c.p.d.} of $\Pb$ knowing $\Gcb^t_{\taub}$.
		Then
		\[
			\Pb^{\Gcb^t_{\taub}}_{\omb} \in \Pch_W \big(\taub(\omb), \muh_{\taub(\omb)}(\omb)  \big),
			~\mbox{\rm for}~
			\Pb\text{\rm--a.e.}\; \omb  \in \Omb.
		\]
	\end{lemma}

\begin{proof}

	Let $\Pb \in \Pch_W(t, \nuh)$.
	First, it is easy to check that for $\Pb$--a.e. $\omb \in \Omb$,
	one has $\Pb^{\Gcb^t_{\taub}}_{\omb} [\alphab_s \in U]=1$, for Lebesgue--almost every $s \in [\taub(\omb),T]$, 
	and $\E^{\Pb^{\Gcb^t_{\taub}}_{\omb}} \big[ \int_{\taub(\omb)}^T \big( \rho(u_0, \alphab_s) \big)^p \mathrm{d}s \big] < \infty$.

	\vspace{0.5em}

	Next, notice that 
	for all $s\in[0,T]$, $\beta \in C_b(\Cc^n \x \Cc \x \Cc^d \x \Cc^\ell)$, $\psi \in C_b(\Cc^\ell \x C([0,T],\Pc(\Omh)))$ and $Z \in \Gcb^t_{\taub}$,
	\begin{align*}
		\E^{\Pb} \big[\langle \beta, \muh_s \rangle \psi(B^{\taub}, \muh) \mathbf{1}_{Z \cap \{ \taub \le s\} } \big]
		=
		\E^{\Pb} \big[\beta(X_{s \wedge \cdot},A_{s \wedge \cdot},W, B_{s \wedge \cdot}) \psi(B^{\taub}, \muh)\mathbf{1}_{Z \cap \{ \taub \le s\} } \big],
	\end{align*}
	so that, for $\Pb$--a.e. $\omb \in \Omb$ and any $\taub(\omb)\leq s\leq T$,
	\begin{align*}
		\E^{\Pb^{\Gcb^t_{\taub}}_{\omb}} \big[\langle \beta, \muh_s \rangle \psi(B^{\taub}_{s \wedge \cdot}, \muh_{s \wedge \cdot}) \big]
		=
		\E^{\Pb^{\Gcb^t_{\taub}}_{\omb}} \big[\beta(X_{s \wedge \cdot},A_{s \wedge \cdot},W, B_{s \wedge \cdot}) \psi(B^{\taub}_{s \wedge \cdot}, \muh_{s \wedge \cdot}) \big].
	\end{align*}
	By considering a countable dense set of maps $(\beta,\psi) \in C_b(\Cc^n \x \Cc \x \Cc^d \x \Cc^\ell) \x C_b(\Cc^\ell \x C([0,T],\Pc(\Omh)))$, 
	it follows that 
	\begin{align*}
	    \muh_s=\Lc^{\Pb^{\Gcb^t_{\taub}}_{\omb}}  \big(X_{s \wedge \cdot} , A_{s \wedge \cdot}, W, B_{s \wedge \cdot}| \Gcb^{\taub(\omb)}_T\big),\; \Pb^{\Gcb^t_{\taub}}_{\omb} \mbox{--a.s.},\; \mbox{for all}\; s \ge \taub(\omb),\; \text{for}\; \Pb\mbox{--a.e.}\; \omb \in \Omb.
	\end{align*}
	Similarly, one can prove that, for $\Pb$--a.e. $\omb$, and $s \le \taub(\omb)$,
	\begin{align*}
	    \muh_s
	    =
	    \Lc^{\Pb^{\Gcb^t_{\taub}}_{\omb}}  (X_{s \wedge \cdot} , A_{s \wedge \cdot}, W, B_{s \wedge \cdot} ),\; \Pb^{\Gcb^t_{\taub}}_{\omb} \mbox{--a.s.},
	\end{align*}
	and hence, for $\Pb$--a.e. $\omb \in \Omb$, 
	\[
		\muh_s =
			\Lc^{\Pb^{\Gcb^t_{\taub}}_{\omb}}  (X_{s \wedge \cdot} , A_{s \wedge \cdot}, W, B_{s \wedge \cdot} ) \mathbf{1}_{\{ s \in [0,\taub(\omb)] \}}
			+
			\Lc^{\Pb^{\Gcb^t_{\taub}}_{\omb}} (X_{s \wedge \cdot} , A_{s \wedge \cdot}, W,B_{s \wedge \cdot}| \Gcb^{\taub(\omb)}_T ) \mathbf{1}_{\{s \in (\taub(\omb), T] \}},
			~
			\Pb^{\Gcb^t_{\taub}}_{\omb} \mbox{--a.s.}
	\]

	Finally, it is clear that for $\Pb$--a.e. $\omb$, one has $\E^{\Pb^{\Gcb^t_{\taub}}_{\omb}} \big [\| X\|^p \big] < \infty$ and $\Pb^{\Gcb^t_{\taub}}_{\omb} \big[ |\Sb|_T < \infty \big] = 1$.
	Moreover, let $\varphi \in C_b^2(\R^n \x \R^d \x \R^\ell)$, 
	so that the localised process $S^{\varphi, m} = \Sb^{\varphi}_{\tau_m \wedge \cdot}$ is a $(\Fb,\Pb)$--martingale on $[t,T]$.
	Fix $T\geq r > s \ge t$, $J \in \Fcb_s$ and $K \in \Gcb^t_{\taub}$, we have
	\[
		\E^{\Pb} \Big[\E^{\Pb^{\Gcb^t_{\taub}}_.}\big [ \Sb^{\varphi}_{\tau_m \wedge r}  \mathbf{1}_J \big ] \mathbf{1}_{K \cap \{ \taub \le s\} } \Big]   
		=
		\E^{\Pb}\big [ \Sb^{\varphi}_{\tau_m \wedge r}  \mathbf{1}_{J \cap K \cap \{ \taub \le s\} } \big]
		=
		\E^{\Pb}\big [ \Sb^{\varphi}_{\tau_m \wedge s}  \mathbf{1}_{J \cap K \cap \{\taub \le s\} } \big] 
		=
		\E^{\Pb}\Big [\E^{\Pb^{\Gcb^t_{\taub}}_.}\big [ \Sb^{\varphi}_{\tau_m \wedge s}  \mathbf{1}_J \big ] \mathbf{1}_{K \cap \{ \taub \le s\} } \Big].
	\]
	This implies that 
	\[ 
		\E^{\Pb^{\Gcb^t_{\taub}}_{\omb}}\big [ \Sb^{\varphi}_{\tau_m \wedge r}  \mathbf{1}_J \big ] 
		=
		\E^{\Pb^{\Gcb^t_{\taub}}_{\omb}}\big [ \Sb^{\varphi}_{\tau_m \wedge s}  \mathbf{1}_J \big ],
		~\mbox{\rm for}~\Pb \mbox{\rm--a.e.}~\omb.
	\]
	By considering countably many $s, r, J$, 
	it follows that $\Sb^{\varphi}$ is a $\big( \Fb, \Pb^{\Gcb^t_{\taub}}_{\omb}\big)$--local martingale on $[\taub(\omb), T]$ for $\Pb$--a.e. $\omb \in \Omb$.
	We hence conclude the proof.
\end{proof}

	We next provide a stability result for $\Pcb_W$ under concatenation.
	For any constant $M > 0$, let us introduce
	\[
		\Pcb^M_t := \Big \{ 
			\Pb \in \Pc(\Omb)~:
			\E^{\Pb} \big[ \| X \|^p \big] +
			\E^{\Pb} \Big[ \!\! \int_t^T \!\! \big( \rho(u_0, \alphab_s) \big)^p \mathrm{d}s \Big] \le M
		\Big\},
		~
		\Pcb_W^M(t, \nu)
		:= 
		\Pcb_W(t, \nu) \cap \Pcb^M_t,
		~
		\Pch_W^M(t, \nuh) 
		:=
		\Pch_W(t, \nuh) \cap \Pcb^M_t,
	\]
	and
	\[
		V_W^M(t, \nu) 
		:=
		\sup_{\Pb \in \Pcb_W^M(t, \nu)} J(t, \Pb).
	\]
	Notice that $V_W^M(t, \nu) \nearrow V_W(t, \nu)$ as $M \nearrow \infty$. 
	Moreover, as in \Cref{graph:weak-formulation}, the graph set 
	\begin{equation} \label{eq:analytic_PCM}
		\big\{ (t, \nu, M, \Pb) ~: \Pb \in \Pcb^M_W(t, \nu) \big\}
		~\mbox{is analytic, and}~
		(t, \nu, M) \longmapsto V^M_W(t,\nu) \in \R \cup \{- \infty,\infty \}
		~\mbox{is upper semi--analytic.}
	\end{equation}

	\begin{lemma} \label{lemm:equiv_nuh_nu}
		Let $t \in [0,T],$ $\nuh_1, \nuh_2 \in \Pc(\Omh)$ and $\nu \in \Pc(\Cc^n)$ be such that
		$\nuh_1 \circ X^{-1}_{t \wedge \cdot} = \nuh_2 \circ X^{-1}_{t \wedge \cdot} = \nu(t)$.
		Then for all $\Pb_1 \in \Pch_W(t, \nuh_1)$, there exists $\Pb_2 \in \Pch_W(t, \nuh_2)$ satisfying
		\[
			\Pb_1 \circ  \big(X, A^t, W^t, B^t\big)^{-1}
			=
			\Pb_2 \circ  \big(X, A^t, W^t, B^t\big)^{-1},
		\]
		where $A^t_{\cdot}:=A_{\cdot \vee t}-A_t,$ so that $J(t, \Pb_1) = J(t, \Pb_2)$.
		Consequently, one has
		\[
			V_W(t, \nu) = \sup_{\Pb \in \Pch_W(t, \nuh_1)} J(t, \Pb),
			\;\mbox{\rm and}\;
			V_W^M(t, \nu) = \sup_{\Pb \in \Pch_W^M(t, \nuh_1)} J(t, \Pb).
		\]
	\end{lemma}
	The proof is almost the same as that of \Cref{lemm:equality_Markov}, and hence it is omitted.

	\begin{lemma} \label{lemm:concatenation}
		Let $(t, \nu) \in [0,T] \x \Pc(\Cc^n)$, $\Pb \in \Pcb_W(t, \nu)$,
		$\taub$ be a $\Gb^t$--stopping time taking values in $[t, T]$,
		$\eps > 0$.
		Then there exists a family of probability measures $(\Qb_{t, \nuh, M}^{\eps})_{(t, \nuh, M) \in [0,T] \x \Pc(\Omh) \x \R_+}$ such that $(t, \nuh, M) \longmapsto \Qb_{t, \nuh, M}^{\eps}$ is universally measurable,
		and for every $(t,\nuh, M)$ s.t. $ \Pch^M_W(t, \nuh) \neq \emptyset$, one has
		\begin{equation} \label{eq:Q_eps_select}
			\Qb^{\eps}_{t, \nuh, M} \in \Pch^M_W(t, \nuh)
			\; \mbox{\rm and}\; 
			J \big(t, \Qb_{t, \nuh, M}^{\eps} \big) \ge 
			\begin{cases}
				V^M_W(t, \nu) - \eps,\; \mbox{\rm when}\; V^M_W(t,\nu) < \infty,\\
				1/\eps,\; \mbox{\rm when}\; V^M_W(t, \nu) = \infty,
			\end{cases}
			\mbox{with}~\nu := \nuh \circ \Xh^{-1}.
		\end{equation}
		Moreover, there exists a $\Pb$--integrable, $\Gcb^t_{\taub}$--measurable r.v. $\widehat M: \Omb \to \R_+$ such that
		for all constant $M \ge 0$, one can find a probability measure $\Pb^{M, \eps} \in \Pcb_W(t, \nu)$ satisfying 
		$\Pb^{M, \eps}|_{\Fcb_{\taub}} = \Pb|_{\Fcb_{\taub}}$ and
		\[
			\big( \Qb^{\eps}_{\taub(\omb), \muh_{\taub(\omb)} (\omb), M+ \widehat M(\omb)} \big)_{\omb \in \Omb}
			~\mbox{\rm is a version of {\rm r.c.p.d.} of}~ 
			\Pb^{M, \eps}
			~\mbox{\rm knowing}~
			\Gcb^t_{\taub}.
		\]
	\end{lemma}
	\begin{proof}
	The existence of the family of probability measures $((\Qb_{t, \nuh, M}^{\eps})_{(t, \nuh, M) \in [0,T] \x \Pc(\Omh)\x \R_+})$ satisfying \eqref{eq:Q_eps_select} follows by \eqref{eq:analytic_PCM} and \Cref{lemm:equiv_nuh_nu}, together with the measurable selection theorem {\color{black} (see e.g. \cite[Proposition 2.21]{karoui2013capacities})}.
	
	\medskip

	With $\Pb \in \Pcb_W(t, \nu)$, we consider a family of r.c.p.d. $(\Pb_{\omb})_{\omb \in \Omb}$ of $\Pb$ knowing $\Gcb^t_{\taub}$, and define
	\[
		\widehat M(\omb) := \E^{\Pb_{\omb}} \bigg[ \| X\|^p + \int_{\taub}^T \big( \rho(\alphab_s, u_0) \big)^p \mathrm{d}s \bigg],
	\]
	so that $\Pb_{\omb} \in  \Pch^{\widehat M(\omb)}_W \big(\taub(\omb), \muh_{\taub(\omb)}(\omb) \big)$ for $\Pb$--a.e. $\omb$, by Lemma \ref{lemm:conditioning}.
	In particular, $\Pch^{\widehat M(\omb)}_W \big(\taub(\omb), \muh_{\taub(\omb)}(\omb) \big)$ is nonempty for $\Pb$--a.e. $\omb \in \Omb$.
	For a fixed constant $M \ge 0$, let
	\[
		\Qb^{\eps}_{\omb} := \Qb^{\eps}_{\taub(\omb), \muh_{\taub(\omb)}(\omb), \widehat M(\omb)+M}.
	\]
	Notice that, for $\Pb$--a.e. $\omb \in \Omb$,
	\[
		\muh_{\taub(\omb)}( \omb)
		=
		\Pb_{\omb} \circ (X_{\taub(\omb) \wedge \cdot},A_{\taub(\omb) \wedge \cdot}, W,  B_{\taub(\omb) \wedge \cdot})^{-1}
		=
		\Qb^{\eps}_{\omb} \circ (X_{\taub(\omb) \wedge \cdot},A_{\taub(\omb) \wedge \cdot}, W,  B_{\taub(\omb) \wedge \cdot})^{-1},
	\]
	then
	\begin{align} \label{eq:loi_Xmu}
		\Lc^{\Qb^\eps_{\omb}} \big(X_{\taub(\omb) \wedge \cdot},A_{\taub(\omb) \wedge \cdot},W_{\taub(\omb) \wedge \cdot},B_{\taub(\omb) \wedge \cdot},\muh_{\taub(\omb) \wedge \cdot} \big)
		&=
		\Lc^{\Qb^\eps_{\omb}} \big(X_{\taub(\omb) \wedge \cdot},A_{\taub(\omb) \wedge \cdot},W_{\taub(\omb) \wedge \cdot},B_{\taub(\omb) \wedge \cdot} \big) 
		\otimes \Lc^{\Qb^\eps_{\omb}} \big(\muh_{\taub(\omb) \wedge \cdot} \big)
		\nonumber \\
		&=
		\Lc^{\Pb_{\omb}} \big(X_{\taub(\omb) \wedge \cdot},A_{\taub(\omb) \wedge \cdot},W_{\taub(\omb) \wedge \cdot},B_{\taub(\omb) \wedge \cdot},\muh_{\taub(\omb) \wedge \cdot} \big).
	\end{align}
	In particular, one has
	\begin{equation} \label{fix_value-B}
	    \Qb^\eps_{\omb} \big[B^t_{\taub(\omb) \wedge \cdot}=(\omb^b)^t_{\taub(\omb) \wedge \cdot},\;\muh_{\taub(\omb) \wedge \cdot}=\omb^{\hat \nu}_{\taub(\omb) \wedge \cdot} \big]=1,
	    \; \mbox{for}\; \Pb \mbox{--a.e.}\; \omb = (\omb^x, \omb^a, \omb^w, \omb^b, \omb^{\hat \nu}) \in \Omb.
	\end{equation}
	
	
	Let us then define  a probability measure $\Pb^{M, \eps}$ on $\Omb$ by
	\begin{align*}
		\Pb^{M, \eps} [K]
		:=
		\int_{\Omb} \Qb_{\omb}(K) \Pb(\mathrm{d}\omb),
		~\mbox{\rm for all}~
		K \in \Fcb.
	\end{align*}
	By \eqref{eq:loi_Xmu}, one has $\Pb^{M, \eps} = \Pb$ on $\Fcb_{\taub}$, and moreover, $(\Qb^{\eps}_{\omb})_{\omb \in \Omb}$ is a family of r.c.p.d. of $\Pb^{M, \eps}$ knowing $\Gcb^t_{\taub}$.
	To conclude the proof, it is enough to check that $\Pb^{M, \eps} \in \Pcb_W(t, \nu)$.
	
	\medskip
	First, it is clear that $\Pb^{M,\eps}[\alphab_s \in U]=1$, for Lebesgue--almost every $s \in [t,T]$, and
	\[
		\E^{\Pb^{M, \eps}} \bigg[ \int_t^T \big( \rho(u_0, \alphab_s) \big)^p \mathrm{d}s \bigg]
		\le
		\E^{\Pb} \bigg[ \int_t^{\taub} \big( \rho(u_0, \alphab_s) \big)^p \mathrm{d}s \bigg]
		+
		\E^{\Pb} \big[ \widehat M \big]
		< \infty. 
	\]

	Next, for each $\beta \in C_b(\Omh)$, $\psi \in C_b(\Cc^{\ell} \x \Pc(\Omh))$, $h \in C_b(\Cc^{\ell})$ and $s \in [t,T]$, one has
	\begin{align*}
		&\ 
		\E^{\Pb^{M, \eps}}\big[ \beta(X_{s \wedge \cdot},A_{s \wedge \cdot},W,B_{s \wedge \cdot})\psi(B^{\taub},\muh) h(B^t_{\taub \wedge \cdot}) \big]
		=
		\int_{\omb} \E^{\Qb^\eps_{\omb}}\big[\beta(X_{s \wedge \cdot},A_{s \wedge \cdot},W,B_{s \wedge \cdot})\psi(B^{\taub},\muh) h(B^t_{\taub \wedge \cdot})\big] \Pb(\mathrm{d}\omb)
		\\
		=&\ 
		\int_{\omb} \E^{\Qb^\eps_{\omb}}\big[\beta(X_{s \wedge \cdot},A_{s \wedge \cdot},W,B_{s \wedge \cdot})\psi(B^{\taub},\muh) \big] h(B^t_{\taub(\omb) \wedge \cdot}(\omb)) \Pb(\mathrm{d}\omb) \\
		=&\
		\int_{\omb} \E^{\Qb^{\eps}_{\omb}}\big[\E^{\Qb^\eps_{\omb}}\big[\beta(X_{s \wedge \cdot},A_{s \wedge \cdot},W,B_{s \wedge \cdot}) \big|\Gcb^{\taub(\omb)}_{T} \big]\psi(B^{\taub},\muh) \big] h(B^t_{\taub(\omb) \wedge \cdot}(\omb)) \Pb(\mathrm{d}\omb)
		\\
		=&\
		\int_{\omb} \E^{\Qb^\eps_{\omb}}\big[\langle \beta, \muh_s \rangle \psi(B^{\taub},\muh) \big] h(B^t_{\taub(\omb) \wedge \cdot} (\omb))) \Pb(\mathrm{d}\omb)\\
		=&\
		\int_{\omb} \E^{\Qb^\eps_{\omb}}\big[\langle \beta, \muh_s \rangle \psi(B^{\taub},\muh) h(B^t_{\taub \wedge \cdot}))\big] \Pb(\mathrm{d}\omb)
		\\
		= &\
		\E^{\Pb^{M, \eps}}\big[\langle \beta,\muh_s \rangle \psi(B^{\taub},\muh) h(B^t_{\taub \wedge \cdot})\big],
	\end{align*}
	where the second and fifth equalities are due to \Cref{fix_value-B}, 
	and the fourth follows by the fact that $\Qb^\eps_{\omb} \in \Pch_W(\taub(\omb), \muh_{\taub(\omb)}(\omb))$.
	Notice that $B^t_u = B^t_{\taub \wedge u} + B^{\taub}_u$ for any $u \in [t,T]$, the above equality implies that
	\[
		\muh_s = \Lc^{\Pb^{M, \eps}} \big(X_{s \wedge \cdot}, A_{s \wedge \cdot}, W, B_{s \wedge \cdot} \big| \Gcb^t_T \big),
		\; \Pb^{M, \eps}\mbox{\rm--a.s.}
	\]
Finally, we easily check that $\Pb^{M, \eps}[ |\Sb|_T < \infty] = 1$ and $\E^{\Pb^{M, \eps}}[ \|X\|^p] \le \E^{\Pb}[ \|X\|^p] + M < \infty$.
	For a fixed test function $ \varphi \in C^2_b(\R^{n+d+\ell})$, we consider the localised stopping times
	$\tau_m$ defined in \eqref{eq:def_tau_m} and $\tau^{\omb}_k(\omb')  :=  \taub(\omb) \vee \tau_k(\omb')$ for each $\omb \in \Omb$.
	We know that $\tau^{\omb}_k \le \tau^{\omb}_{k+1}$, for any $k\in\N$, that $\tau^{\omb}_k \underset{k\rightarrow\infty}{\longrightarrow} \infty$, and that $(\Sb^{\varphi}_{s \wedge \tau^{\omb}_k})_{s \in [\tau(\omb),T]}$ is an $(\Fb, \Qb^\eps_{\omb})$--martingale for all $k \in \N$. 
	Notice that for all $s \in [t,T]$ and $A \in \Fcb_s$, the map
	\[
		\omb \longmapsto \E^{\Qb^\eps_{\omb}}\Big[\Sb^{\varphi}_{s \wedge \tau_m \wedge \tau_k^{\omb}} \mathbf{1}_A \mathbf{1}_{s > \taub(\omb)} \Big]
		~\mbox{is}~\Gcb^t_{\taub} \mbox{--measurable}.
	\]
	Then for $s \le r \le T,$
	\begin{align*}
		&\
		\E^{\Pb^{M, \eps}}\big[\Sb^{\varphi}_{s \wedge \tau_m} \mathbf{1}_A\big]
		= 
		\E^{\Pb^{M, \eps}}\big[\Sb^{\varphi}_{s \wedge \tau_m} \mathbf{1}_A \mathbf{1}_{s \le \taub}\big]
		+
		\E^{\Pb^{M, \eps}}\big[\Sb^{\varphi}_{s \wedge \tau_m} \mathbf{1}_A \mathbf{1}_{s > \taub}\big]
		\\
		=&\
		\E^{\Pb}[\Sb^{\varphi}_{s \wedge \tau_m} \mathbf{1}_A \mathbf{1}_{s \le \taub}]
		+
		\lim_{k\rightarrow\infty}
		\int_{\Omb}\E^{\Qb^\eps_{\omb}}\big[\Sb^{\varphi}_{s \wedge \tau_m \wedge \tau_k^{\omb}} \mathbf{1}_A \mathbf{1}_{s > \taub(\omb)}\big] \Pb(d\omb) 
		=
		\E^{\Pb}\big[\Sb^{\varphi}_{s \wedge \tau_m} \mathbf{1}_A \mathbf{1}_{s \le \taub}\big]
		+
		\E^{\Pb^{M, \eps}}\big[\Sb^{\varphi}_{r \wedge \tau_m} \mathbf{1}_A \mathbf{1}_{s > \taub}\big]
		\\
		=&\ 
		\E^{\Pb}\big[\Sb^{\varphi}_{\taub \wedge \tau_m} \mathbf{1}_A \mathbf{1}_{s \le \taub}\mathbf{1}_{r < \taub} \big]
		+
		\E^{\Pb}\big[\Sb^{\varphi}_{\taub \wedge \tau_m} \mathbf{1}_A \mathbf{1}_{s \le \taub}\mathbf{1}_{\taub \le r}\big]
		+
		\E^{\Pb^{M, \eps}}\big[\Sb^{\varphi}_{r \wedge \tau_m} \mathbf{1}_A \mathbf{1}_{s > \taub}\big]
		\\
		=&\ 
		\E^{\Pb}\big[\Sb^{\varphi}_{r \wedge \tau_m} \mathbf{1}_A \mathbf{1}_{s \le \taub}\mathbf{1}_{r < \taub}\big]
		+
		\E^{\Pb^{M, \eps}}\big[\Sb^{\varphi}_{\taub \wedge \tau_m} \mathbf{1}_A \mathbf{1}_{s \le \taub}\mathbf{1}_{\taub \le r}\big]
		+
		\E^{\Pb^{M, \eps}}\big[\Sb^{\varphi}_{r \wedge \tau_m} \mathbf{1}_A \mathbf{1}_{s > \taub}\big]
		\\
		=&\ 
		\E^{\Pb^{M, \eps}}\big[\Sb^{\varphi}_{r \wedge \tau_m} \mathbf{1}_A \mathbf{1}_{s \le \taub}\mathbf{1}_{r < \taub}\big]
		+
		\E^{\Pb^{M, \eps}}\big[\Sb^{\varphi}_{r \wedge \tau_m} \mathbf{1}_A \mathbf{1}_{s \le \taub}\mathbf{1}_{\taub \le r}\big]
		+
		\E^{\Pb^{M, \eps}}\big[\Sb^{\varphi}_{r \wedge \tau_m} \mathbf{1}_A \mathbf{1}_{s > \taub}\big]
		= 
		\E^{\Pb^{M, \eps}}\big[\Sb^{\varphi}_{r \wedge \tau_m} \mathbf{1}_A\big],
	\end{align*}
	which means that $(\Sb_{u}^{\varphi})_{u \in [t,T]}$ is an $(\Fb,\Pb^{M, \eps})$--local martingale,
	and hence $\Pb^{M, \eps} \in \Pcb_W(t, \nu)$.
	\end{proof}

\subsection{Proof of the main results}
\label{subsec:Proofs}

\subsubsection{Proof of Theorem \ref{thm:WeakDPP} }
\label{Proof_of_WeakDPP}

	First, $V_W$ is upper semi--analytic by \Cref{graph:weak-formulation}.
	Further, let $\taub$ be a $\Gb^t$--stopping time taking value in $[t,T]$,
	 it follows by \Cref{lemm:conditioning} that, for every $\Pb \in \Pcb_W(t,\nu)$
	\begin{align*}
		J(t, \Pb)
		&=
		\E^{\Pb} \bigg[
			\int_t^{\taub} L(s, X_{s \wedge \cdot} , \mub_s, \alphab_s) \mathrm{d}s
			+
		    \E^\Pb 
		    \bigg[	
			\int_{\taub}^T L(s, X_{s \wedge \cdot} , \mub_s, \alphab_s) ds
			+ 
			g\big( X_{T \wedge \cdot}, \mu_{T} \big)
		    \bigg|\Gcb^t_{\tau} \bigg]
		\bigg] 
		\\
		&\le
		\E^\Pb \bigg[
			\int_t^{\taub} L(s, X_{s \wedge \cdot}, \mub_s, \alphab_s) \mathrm{d}s
			+
		    V_W\big(\taub, \mu_{\taub} \big)
		\bigg] 
		\\
		&\le 
		\sup_{\Pb^\prime \in \Pcb_W(t, \nu)} 
		\E^{\Pb^\prime} \bigg[
			\int_t^{\taub} L(s, X_{s \wedge \cdot}, \mub_s, \alphab_s) \mathrm{d}s
			+ 
			V_W\big(\taub, \mu_{\taub} \big) 
		\bigg].
	\end{align*}
	Notice that a $\Gb^t$--stopping time on $\Omb$ can be considered as a $\G^{\star}$--stopping time $\tau^{\star}$ on $\Om^{\star}$.
	Then by the way how $\tau^{\gamma}$ is defined from $\tau^{\star}$ in \eqref{eq:tau_from_taut} and \Cref{lemma:equivalence}, we obtain the inequality
	\begin{align} \label{eq:DPP-firstEquality}
		V_W(t, \nu) 
		\le 
		\sup_{\gamma \in \Gamma_W(t, \nu)} 
		\E^{\P^{\gamma}} \bigg[
			\int_t^{\tau^{\gamma}} L(s, X^{\gamma}_{s \wedge \cdot} , \mub^{\gamma}_s, \alpha^{\gamma}_s) \mathrm{d}s
			+ 
			V_W\big(\tau^{\gamma}, \mu^{\gamma}_{\tau^{\gamma}} \big) 
		\bigg].
	\end{align}
	
	We now consider the reverse inequality, for which one can assume w.l.o.g. that 
	\begin{equation} \label{eq:V_finite}
		V_W(t, \nu) < \infty,
		~\mbox{and}~
		\sup_{\Pb^\prime \in \Pcb_W(t, \nu)} 
		\E^{\Pb^\prime} \bigg[
			\int_t^{\taub} L(s, X_{s \wedge \cdot}, \mub_s, \alphab_s) \mathrm{d}s
			+ 
			V_W\big(\taub, \mu_{\taub} \big) 
		\bigg]
		> - \infty.
	\end{equation}
	Let $\Pb \in \Pcb_W(t,\nu)$ be a weak control rule,
	then by \Cref{lemm:concatenation}, for some $\Fcb_{\taub}$--measurable $\Pb$--integrable r.v. $\widehat M : \Omb \to \R_+$, 
	one has a family of  probability measures $(\Pb^{M, \eps})_{M \ge 0}$ in $\Pcb_W(t, \nu)$ such that
	\begin{align*}
		&\ \E^\Pb \bigg[
			\int_t^{\taub} L(s, X_{s \wedge \cdot}, \mub_s, \alphab_s) \mathrm{d}s
			+ 
			\big(V^{M + \widehat M(\omb)}_W\big(\taub, \mu_{\taub} \big)-\varepsilon\big)\mathbf{1}_{\{V^{M +\widehat M(\omb)}_W(\taub, \mu_{\taub} )<\infty\}}
			\bigg]
			+
			\frac{1}{\varepsilon}\Pb\big[V^{M + \widehat M(\omb)}_W(\taub, \mu_{\taub})=\infty\big]
		\\
		\le&\
		\E^\Pb \bigg[
			\int_t^{\taub} L(s, X_{s \wedge \cdot}, \mub_s, , \alphab_s ) \mathrm{d}s
			+ 
			\E^{\Qb^\eps_{\taub,\hat \mu, M+\widehat M}} 
		    \bigg[	
			\int_{\taub}^T L(s, X_{s \wedge \cdot} , \mub_s, \alphab_s) \mathrm{d}s
			+ 
			g\big( X_{T \wedge \cdot}, \mu_{T} \big)
		    \bigg] 
		\bigg]
		\\
		=&\
		\E^{\Pb^{M, \eps}} \bigg[
			\int_t^T L(s, X_{s \wedge \cdot}, \mub_s, \alphab_s) \mathrm{d}s
			+ 
			g\big( X_{T \wedge \cdot}, \mu_{T} \big)
		\bigg]
		\le
		V_W(t,\nu).
	\end{align*}
	If $\Pb\big[V^{M+ \widehat M}_W(\taub, \mu_{\taub})=\infty\big)]> 0$ for some $M \ge 0$, 
	then by taking $\eps \longrightarrow 0$, one finds $V_W(t,\nu)=\infty$ which is in contradiction to \eqref{eq:V_finite}.
	When $\Pb\big[V^{M+\widehat M}_W(\taub, \mu_{\taub})=\infty\big] = 0$ for all $M \ge 0$,
	let $M  \longrightarrow \infty$ and then take the supremum over all $\Pb \in \Pcb_W(t,\nu)$, it follows that
	\begin{align*}
		&\sup_{\Pb^\prime \in \Pcb_W(t,\nu)}\E^{\Pb^\prime} \bigg[
			\int_t^{\taub} L(s, X_{s \wedge \cdot}, \mub_s, \alphab_s ) \mathrm{d}s
			+ 
			V_W\big(\taub, \mu_{\taub} \big) 
		\bigg]
		-
		\eps
		\le
		V_W(t,\nu).
	\end{align*}
	Notice that $\eps > 0$ is arbitrary, and again by the way how $\tau^{\gamma}$ is defined from $\tau^{\star}$ (equivalent to $\taub$ on $\Omb$) and \Cref{lemma:equivalence}, 
	we can conclude the proof with \eqref{eq:DPP-firstEquality}.
	\qed

\subsubsection{Proof of Theorem \ref{thm:StrongDPP}}
\label{proof:StrongDPP}

	Let $(t,\nu) \in [0,T] \x \Pc_2(\Cc^n)$ and \Cref{assum:Growth} hold,  by  \cite[Theorem 3.1]{djete2019general} (letting $\hat p = p =2$ in their Assumption 2.1), one has
	\[
		V_S(t,\nu) =V_W(t,\nu).
	\]
	Therefore $V_S: [0,T] \x \Pc_2 (\Cc^n) \longrightarrow \R \cup \{-\infty,\infty \}$ has the same measurability as $V_W: [0,T] \x \Pc_2(\Cc^n) \longrightarrow \R \cup \{-\infty,\infty \}$.

	\medskip
	
	Next, let $\tau$ be a $\G^{t, \circ}$--stopping time on $(\Om^t, \Fc^t, \P^t_{\nu})$ taking value in $[t,T]$,
	we denote  $\tau^{\gamma} := \tau(B^{\gamma, t})$ for $\gamma \in \Gamma_S(t,\nu)$.
	Then by the formulation equivalence result in \Cref{prop:strong_ctrl_fixed_space}, the DPP result \eqref{eq:strongDPP} is equivalent to
	\[
		V_S(t,\nu) 
		=
		\sup_{\gamma \in \Gamma_S(t,\nu)} 
		\E^{\P^\gamma} \bigg[
			\int_t^{\tau^{\gamma}} L\big(s, X^\gamma_{s \wedge \cdot}, \mub^\gamma_s, \alpha^\gamma_s \big) \mathrm{d}s
			+
			V_S({\tau^{\gamma}}, \mu^\gamma)
		\bigg].
	\]
	
	Recall that under \Cref{assum:Lip} and by \Cref{theorem_Existence/uniqueness-SDE}, one has 
	\[
		\E^{\P^\gamma}
		\bigg[ 
			\sup_{s \in [0,T]} |X^\gamma_s|^2 
		\bigg] 
		< 
		\infty.
	\]
	Then by \Cref{lemm:conditioning} and the fact that $V_S = V_W$, it follows that
	\begin{align*}
		V_S(t,\nu) 
		=
		\sup_{\gamma \in \Gamma_S(t,\nu)} J(t,\gamma)
		&=
		\sup_{\gamma \in \Gamma_S(t,\nu)}
		\E^{\P^\gamma} \bigg[
			\int_t^{\tau^{\gamma}} L\big(s, X^\gamma_{s \wedge \cdot}, \mub^\gamma_s, \alpha^\gamma_s \big) \mathrm{d}s
			+
			\int_{\tau^{\gamma}}^T L\big(s, X^\gamma_{s \wedge \cdot}, \mub^\gamma_s, \alpha^\gamma_s \big) \mathrm{d}s 
			+
			g\big( X^\gamma_{T \wedge \cdot}, \mu^\gamma_T \big)
		\bigg]
		\\
		&\le
		\sup_{\gamma \in \Gamma_S(t,\nu)} 
		\E^{\P^\gamma} \bigg[
			\int_t^{\tau^{\gamma}} L\big(s, X^\gamma_{s \wedge \cdot}, \mub^\gamma_s, \alpha^\gamma_s \big) \mathrm{d}s
			+
			V_S({\tau^{\gamma}}, \mu^\gamma)
		\bigg].
	\end{align*}
	Further, by \Cref{thm:WeakDPP}, we have
	\begin{align*}
		V_S(t,\nu)
		=
		V_W(t,\nu)
		&=
		\sup_{\gamma \in \Gamma_W(t,\nu)}
		\E^{\P^\gamma} \bigg[
			\int_t^{\tau^{\gamma}} L\big(s, X^\gamma_{s \wedge \cdot}, \mub^\gamma_s, \alpha^\gamma_s \big) \mathrm{d}s
			+
		V_W({\tau^{\gamma}},\mu^\gamma)
		\bigg]
		\\
		&\ge
		\sup_{\gamma \in \Gamma_S(t,\nu)}
		\E^{\P^\gamma} \bigg[
			\int_t^{\tau^{\gamma}} L\big(s, X^\gamma_{s \wedge \cdot}, \mub^\gamma_s, \alpha^\gamma_s \big) \mathrm{d}s
			+
		V_S({\tau^{\gamma}},\mu^\gamma)
		\bigg],
	\end{align*}
	and hence the proof is concluded.
	\qed

\subsubsection{Proof of Theorem \ref{thm:B-StrongDPP}}

In this part, we use the results and techniques of \Cref{thm:WeakDPP} to show the $\mathrm{DPP}$ for $V^\B_S$. We start by proving the universal measurability of $V^\B_S$. For this, we consider an equivalent formulation of $V^\B_S$, which is more appropriate for our purpose.

\paragraph{An equivalent reformulation for $V_S^\B$}
	Let $\Omt^\star:=\Cc^\ell$ be the canonical space with canonical process $\Bt^\star$,
	and $\Pt^\star$ be the Wiener measure, under which $\Bt^\star$ is an $\ell$--dimensional standard Brownian motion.
	Let $\Ft^{\star} = (\Fct^{\star}_t)_{t \in [0,T]}$ be the canonical filtration.
	Recall that we consider a fixed Borel map $\pi: U \cup \{\partial\} \to \R \cup \{-\infty,\infty \}$.
	We denote by $\Uc$ the set of $\Ft^{\star}$--predictable processes $\theta$ taking values in $\R$, such that $\E^{\Pt^{\star}} \big[\int_0^T \big|\theta_t \big|^2 \mathrm{d}t \big]< \infty.$	
	Define a metric $d^{\star}$ on $\Uc$ by
	\[
		d^{\star}(\eta,\theta)^2
		:=
		\E^{\Pt^{\star}}\bigg[\int_0^T  \big|\eta_t-\theta_t\big|^2  \mathrm{d}t \bigg],
		\; \mbox{for all}\; 
		(\eta,\theta) \in \Uc \x \Uc,
	\]
	so that $(\Uc, d^{\star})$ is a Polish space (see e.g. \citeauthor*{brezis2011functional} \cite[Theorems 4.8 and 4.13]{brezis2011functional}). 
	Next, let $\theta \in \Uc$, and define $A^{\theta}_t := \int_0^t \theta_s \mathrm{d}s$, $t\in[0,T]$. We consider then the map $\Upsilon: \Uc \longrightarrow \Pc(\Cc^\ell \x \Cc)$ defined by
	\[
		\Upsilon (\theta)
		:=
		\Pt^{\star} \circ \Big(\Bt^{\star}_{\cdot}, A^{\theta(\Bt^{\star}_{\cdot})}_{\cdot} \Big)^{-1},\; \theta\in\Uc.
	\]
	
	Let us introduce, for all $t \in [0,T]$, $\nu \in \Pc_2(\Cc^n)$ and $\nuh \in \Pc_2(\Omh)$ such that $\nu = \nuh \circ \Xh^{-1}$, 
	\[
		\Pcb_S^\star(t,\nu)
		:=
		\big\{
			\Pb \in \Pcb_W(t,\nu)  :
			\Pb \circ \big(B_{\cdot}, A_{\cdot} \big)^{-1} \in \Upsilon \big(\Uc \big),\; \mbox{and}\; B_{t \wedge \cdot}\; \mbox{is}\; \Pb \mbox{--independent of}\;(B^t,A)
		\big\},
	\]
	and
	\[
		\Pch_S^\star(t, \nuh) := \Pch_W(t, \nuh) \cap \Pcb_S^\star(t, \nu).
	\]

	\begin{lemma} \label{lemm:VS_star}
		Let $(t,\nu) \in [0,T] \x \Pc_2(\Cc^n)$ and $\nuh \in \Pc(\Omh)$ be such that $\nuh \circ \Xh^{-1} = \nu$.
		Then under {\rm \Cref{assum:Lip}}, one has $\Pcb_S^\star(t,\nu) \subseteq \Pcb_S^\B(t,\nu) $ and
		\begin{equation} \label{eq:equiv_VSB}
			V^\B_S(t,\nu)
			=
			\sup_{\Pb \in \Pcb^\star_S(t,\nu)} J(t,\Pb).
		\end{equation}
	\end{lemma}
	\begin{proof}
	
	First, take $\gamma \in \Gamma_S^\B(t,\nu)$. 
	W.l.o.g., we can assume that there exists an independent Brownian motion $\widetilde B$ in the space $(\Om^{\gamma}. \Fc^{\gamma}, \P^{\gamma})$,
	and let  $B_{\cdot}^{\star,\gamma}:=B^\gamma_{t \vee \cdot}-B^\gamma_t+\widetilde{B}_{t \wedge \cdot}$,
	then
	\[
		\gamma' := \big( \Om^{\gamma}, \Fc^{\gamma},\P^{\gamma},  \F^{\gamma}, \G^{\gamma}, X^{\gamma}, W^{\gamma}, B^{\star,\gamma},  \mub^{\gamma}, \mu^{\gamma}, \alpha^{\gamma} \big) 
		\in
		\Pcb^{\B}_S(t,\nu).
	\]
	Recall that  $\alpha^{\gamma}$ is $\G^{\gamma}$--predictable and $\G^{\gamma}$ is the augmented filtration generated by $B^{\gamma,t}$,
	then for some Borel function $\phi: [t,T] \x \Cc^{\ell} \to U$,
	one has $\alpha^{\gamma}_s = \phi(s, B^{\gamma, t}_{s \wedge \cdot})$, $s \in [t,T]$, $\P^{\gamma}$--a.s.
	Let $A^{\gamma'}_\cdot := \int_0^{\cdot} \pi \big( \phi(s,B^{\gamma,t}_{s \wedge \cdot})  \big) \mathrm{d}s$ and $\muh^{\gamma'}$ be defined as in \eqref{eq:def_muh}, 
	it follows that 
	$
		\Pb':= \P^{\gamma} \circ \big( X^\gamma,A^{\gamma'},W^\gamma, B^{\gamma'}, \muh^{\gamma'} \big)^{-1}  \in \Pcb_W(t, \nu)
	$
	satisfies $\Pb' \circ (B, A)^{-1} \in \Gamma(\Uc)$
	and $J(t, \Pb') = J(t, \gamma)$.
	Then  $J(t,\gamma)=J(t,\Pb') \le \sup_{\Pb \in \Pcb^\star_S(t,\nu)} J(t,\Pb)$ and hence $V^{\B}_S(t,\nu) \le \sup_{\Pb \in \Pcb^\star_S(t,\nu)} J(t,\Pb)$.

	\medskip

	Next, given $\Pb \in \Pcb^{\star}_S(t,\nu)$, since $\Pb \circ \big(B, A \big)^{-1} \in \Upsilon \big(\Uc \big)$, there exists $\theta^\star \in \Uc$ such that $\Pb \circ \big(B_{\cdot}, A_{\cdot} \big)^{-1}= \Pt^\star \circ \big(\Bt^{\star}_{\cdot}, A^{\theta^\star(\Bt^{\star}_{\cdot})}_{\cdot}) \big)^{-1}.$ 
	Thus $\pi\big(\alphab_s(\om) \big)=\theta^\star_s(B_{s \wedge \cdot}(\om))$, for $\mathrm{d}\Pb \otimes \mathrm{d}t$--a.e. $(s,\omb) \in [t,T] \x \Omb$. 
	As $\Pb \in \Pcb_W(t,\nu),$ we know $\Pb[\alphab_s \in U]=1$ for $\mathrm{d}t$--a.e. $s \in [0,T],$ therefore $\pi\big(\alphab_s(\omb) \big)=\theta^\star_s(B_{s \wedge \cdot}(\omb)) \in \pi(U)$ and $\alphab_s(\omb)= \pi^{-1} \big(\theta^\star_s(B_{s \wedge \cdot}(\omb)) \big) \in U$, for $\mathrm{d}\Pb \otimes \mathrm{d}t$--a.e. $(s,\omb) \in [0,T] \x \Omb$.
	 Further, since $(B^t,A)$ is $\Pb$--independent of $B_{t \wedge \cdot},$ 
	it follows that there is a Borel measurable function $\phi: [0,T] \x \Cc^\ell \longrightarrow \R$ such that 
	$A_s=\phi(s,B^t_{s \wedge \cdot})$, $s \in  [0,T]$, $\Pb$--a.s.,
	and therefore $\Pb \in \Pcb^{\B}_S(t,\nu)$. 
	This implies that $\Pcb_S^\star(t,\nu) \subseteq \Pcb_S^\B(t,\nu)$, and the equality \eqref{eq:equiv_VSB}.
	\end{proof}

	We are now ready to prove the measurability of $V_S^\B$.
	\begin{lemma}
	 \label{lemm:measurability_B-Strong}
		The graph sets
		\begin{align*}
			\llbracket \Pcb^\star_S \rrbracket
			:=
			\big \{
				(t,\nu,\Pb) \in [0,T] \x \Pc_2(\Cc^n) \x \Pc(\Omb): \Pb \in \Pcb_S^\star(t,\nu)
			\big \},
			\;\mbox{\rm and}\;
			\llbracket \Pch^\star_S \rrbracket
			:=
			\big \{
				(t,\nuh,\Pb) : \Pb \in \Pch_S^\star(t,\nuh)
			\big \},
			\end{align*}
		are analytic sets in respectively $[0,T] \x \Pc_2(\Cc^{n}) \x  \Pc(\Omb)$ and $ [0,T] \x \Pc_2(\Omh) \x \Pc(\Omb)$.
		Consequently, 
			$V_S^\B
			:
			[0,T] \x \Pc_2(\Cc^n)
			\longrightarrow
			\R \cup \{-\infty,\infty\}$
		is upper semi--analytic.
	\end{lemma}    
	\begin{proof}
	We will only consider the case of $\Pcb^{\star}_S$, while the proof is almost the same for $\Pch^{\star}_S$. First, notice that 
	\[
		\Upsilon:\;
		\Uc
		\longrightarrow 
		\Pc(\Cc^\ell \x \Cc),
	\]
	is continuous and injective, so that $\Upsilon \big(\Uc \big)$ is a Borel subset of $\Pc(\Cc^\ell \x \Cc )$ {\color{black} (see e.g. \citeauthor*{kechris1995classical} \cite[Theorem 15.1]{kechris1995classical})}. 
	It follows that
	\begin{align*}
		 \D^1
		:=
		\big \{
			(t,\nu,\Pb) \in [0,T] \x \Pc_2(\Cc^n) \x \Pc(\Omb) : \Pb \circ \big(B_{\cdot}, A_{\cdot} \big)^{-1} \in \Upsilon \big(\Uc \big)
		\big \},
	\end{align*}
	is a Borel subset of $[0,T] \x \Pc(\Cc^n) \x \Pc(\Omb)$,
	as the map
	\[
	    \Gamma_1:[0,T] \x \Pc_2(\Cc^n) \x \Pc(\Omb)\ni (t,\nu,\Pb)  \longmapsto \Pb \circ \big(B_{\cdot}, A_{\cdot} \big)^{-1} \in \Pc(\Cc^\ell \x \Cc),
	\]
	is Borel measurable. 
	Similarly
	\begin{align*}
		 \D^2
		:=
		\big \{
			(t,\nu,\Pb) \in [0,T] \x \Pc_2(\Cc^n) \x \Pc(\Omb) : B_{t \wedge \cdot}\; \mbox{is}\; \Pb\mbox{--independent of}\;(B^t,A,\muh)
		\big \},
	\end{align*}
	is also a Borel subset of $[0,T] \x \Pc(\Cc^n) \x \Pc(\Omb)$.
	Indeed, for all $(h,\psi) \in C_b(\Cc^\ell) \x C_b \big( \Cc^\ell \x \Cc\big)$, the function
	\[
	    \Gamma_{h,\psi}:[0,T] \x \Pc_2(\Cc^n) \x \Pc(\Omb)\ni (t,\nu,\Pb)  \longmapsto \big(\E^\Pb [h(B_{t \wedge \cdot}) \psi(B^t,A)] 
	    -
	    \E^\Pb [h(B_{t \wedge \cdot})] \E^\Pb [ \psi(B^t,A)] \big) \in \R,
	\]
	is continuous. 
	By consider a countable dense subset $\Rc \subset C_b(\Cc^\ell) \x C_b \big( \Cc^\ell \x \Cc \big)$, it follows that
	\[
	    \D^2= \bigcap_{(h,\psi) \in \Rc} \Gamma_{\varphi,\psi}^{-1}\{ 0 \},
	\]
	is a Borel set. Finally, notice that 
	\begin{align*}
		\llbracket \Pc^\star_S \rrbracket
		=
		\llbracket\Pcb_W \rrbracket
		\cap
		\D^1
		\cap
		\D^2,
	\end{align*}
	
	and we then conclude the proof by \Cref{graph:weak-formulation} and \Cref{lemm:VS_star}.
	\end{proof}
	
	Recall that for each $M > 0$, $\Pc^M_t$ is defined in Section \ref{subsec:TechLemma}, we similarly introduce,
	for $(t, \nu) \in [0,T] \x \Pc_2(\Cc^n)$ and $\nuh \in \Pc_2(\Omh)$ such that $\nu = \nuh \circ \Xh^{-1} $,
	\[
		\Pcb_S^{\B,M}(t, \nu)
		:= \Pcb_S^\B(t, \nu) \cap \Pc^M_t,
		~
		\Pcb_S^{\star,M}(t, \nu)
		:=
		 \Pcb_S^\star \cap \Pcb_S^{\B, M}(t,\nu),
		~
		\Pch_S^{\B,M}(t,\nuh)
		:= 
		\Pcb_S^\B(t, \nuh) \cap \Pc^M_t,
		~
		\Pch_S^{\star,M}(t, \nuh) 
		:=
		 \Pcb_S^\star \cap \Pch_W^{\B,M}(t,\nuh).
	\]
	By \Cref{lemm:equiv_nuh_nu}, it is clear that
	\[
		V_S^{\B,M}(t, \nu) 
		:=\!\!
		\sup_{\Pb \in \Pcb_S^{\B,M}(t, \nu)} \!\! J(t, \Pb)
		= \!\!
		\sup_{\Pb \in \Pcb_S^{\star,M}(t, \nu)} \!\! J(t, \Pb)
		=\!\!
		\sup_{\Pb \in \Pch_S^{\B,M}(t, \nuh)} \!\! J(t, \Pb)
		=\!\!
		\sup_{\Pb \in \Pch_S^{\star,M}(t, \nuh)} \!\! J(t, \Pb)
		~\nearrow~ V_S^\B(t, \nu),
		~\mbox{as}~M \nearrow \infty.
	\]

	\begin{lemma} \label{lemm:conditioning-concatenation_StrongFormulation}
		$(i)$
		Let $(t,\nu) \in [0,T] \x \Pc_2(\Cc^n)$, $\Pb \in \Pcb_S^{\B}(t,\nu)$, 
		$\taub$ a $\Gb^t$--stopping time taking values in $[t,T]$, 
		and $ \big( \Pb^{\Gcb^t_{\taub}}_{\omb} \big)_{\omb \in \Omb}$ be a family of {\rm r.c.p.d.} of $\Pb$ knowing $\Gcb^t_{\taub}$.
		Then $\Pb^{\Gcb^t_{\taub}}_{\omb} \in \Pcb_S^\B \big(\taub(\omb), \mu_{\taub(\omb)}(\omb) \big)$, for $\Pb$--a.e. $\omb \in \Omb$.

		\vspace{0.5em}
		
		$(ii)$ 
		The graph set $\big\{ (t, \nuh, M, \Pb) ~: \Pb \in \Pch_S^{\star,M}(t, \nuh) \big\}$ is analytic.
		Further, let $(t, \nu) \in [0,T] \x \Pc_2(\Cc^n)$, $\Pb \in \Pcb_S^\B(t, \nu)$,
		$\taub$ be a $\Gb^t$--stopping time taking values in $[t, T]$, and 
		$\eps > 0$.
		Then there exists a family of probability measures 
		and a family of probability measures $(\Qb_{t, \nuh, M}^{\eps})_{(t, \nuh, M) \in [0,T] \x \Pc(\Omh) \x \R_+}$ such that $(t, \nuh, M) \longmapsto \Qb_{t, \nuh, M}^{\eps}$ is universally measurable,
		and for every $(t,\nuh, M)$ s.t. $ \Pch^{\B, M}_S(t, \nuh) \neq \emptyset$, 
		one has
		\begin{equation} \label{ineq:eps-optimal}
			\Qb^{\eps}_{t, \nuh, M} \in \Pch^{\B,M}_S(t, \nuh),
			\; \mbox{\rm and}\; 
			J \big(t, \Qb_{t, \nuh,M}^{\eps} \big) \ge 
			\begin{cases}
				V^{\B,M}_S(t, \nu) - \eps,\; \mbox{\rm when}\; V^{\B,M}_S(t,\nu) < \infty,\\
				\frac{1}{\eps}, \; \mbox{\rm when}; V^{\B,M}_S(t, \nu) = \infty,
			\end{cases}
			\mbox{for}~\nu = \nuh \circ \Xh^{-1}.
		\end{equation}
		Moreover, there is a $\Gcb^t_{\taub}$--measurable and $\Pb$--integrable r.v. $\widehat M : \Omb \to \R_+$ such that for all constant $M>0$, 
		there exists $\Pb^{M, \eps} \in \Pcb^{\B}_S(t, \nu)$ such that $\Pb^{M, \eps}|_{\Fcb_{\taub}} = \Pb|_{\Fcb_{\taub}}$ and
		\[
			\Big( \Qb^{\eps}_{\taub(\omb), \muh(\omb), M+\widehat M(\omb)} \Big)_{\omb \in \Omb}
			~\mbox{\rm is a version of the r.c.p.d. of}~
			\Pb^{\eps,M}
			~\mbox{\rm knowing}~ \Gcb^t_{\taub}.
		\]
	\end{lemma}
	\begin{proof}
	$(i)$
	Let $\Pb \in \Pcb_S^{\B}(t,\nu)$, then there exists a Borel measurable function $\phi: [t, T] \x \Cc^{\ell} \longrightarrow U$ such that
	\[
		\alphab_s
		=
		\phi\big(s,B^t_{s \wedge \cdot}\big),
		\; \mbox{for all}\; s \in [t,T],\; 
		\Pb\mbox{--a.s.}
	\]
	Let us consider the concatenated path $(\omb \otimes_t \bar \w)_s := \omb_{t \wedge s} + \bar \w_{s\vee t} - \bar \w_t$ and define a Borel measurable function $\phi^{\omb}$ by
	\[
		\phi^{\omb}\big(s, \bar \w^b\big)
		:=
		\phi\big(s, \omb^b \otimes_{\taub(\omb)} \bar \w^b \big),
		~\mbox{for}~
		s \in [\taub(\omb),T],
		~\omb = (\omb^x, \omb^a, \omb^w, \omb^b, \omb^{\muh}),
		~\w = (\w^x, \w^a, \w^w, \w^b, \w^{\muh}) \in \Omb.
	\]
	Then by a classical conditioning argument, it is easy to check that for $\Pb$--a.e. $\omb \in \Omb$,
	\[
		\alphab_s
		=
		\phi^{\omb} \big(s,B^{\taub(\omb)}_{s \wedge \cdot}\big),\; \mbox{for all} \; s \in [\taub(\omb),T],\; 
		\Pb^{\Gcb^t_{\taub}}_{\omb} \mbox{--a.s.}
	\]
	Using \Cref{lemm:conditioning} and \Cref{def:PS}, it follows that $\Pb^{\Gcb^t_{\taub}}_{\omb} \in \Pcb_S^\B \big(\taub(\omb), \mu_{\taub(\omb)}(\omb) \big)$, for $\Pb$--a.e. $\omb \in \Omb$.

	\vspace{0.5em}

	$(ii)$ 
	Using \Cref{lemm:measurability_B-Strong}, it is easy to see that the graph set $\big\{ (t, \nu, M, \Pb) ~: \Pb \in \Pcb_S^{\star,M}(t, \nu) \big\}$ is analytic.
	Then one can apply the same arguments as in \Cref{lemm:concatenation} to obtain a measurable family $(\Qb^{\eps}_{t, \nu, M})_{(t, \nu, M) \in [0,T] \x \Pc(\Cc^n) \x \R_+}$ 
	such that 
	\[
		\Qb^{\eps}_{t, \nu, M} \in \Pcb_S^{\star, M}(t, \nu)
		~\mbox{\rm and}~
		J(t, \Qb^{\eps}_{t, \nu, M} ) \ge (V^{\B, M}_S(t, \nu) - \eps ) \mathbf{1}_{\{V^{\B, M}_S < \infty\}} + \frac{1}{\eps} \mathbf{1}_{\{V^{\B, M}_S = \infty\}}.
	\]
	To proceed, we will define a family $(\Qb^{\eps}_{t, \nuh, M})_{(t, \nuh, M) \in [0,T] \x \Pc(\Omh) \x \R_+} $ from the family $(\Qb^{\eps}_{t, \nu, M})_{(t, \nu, M) \in [0,T] \x \Pc(\Cc^n) \x \R_+}$ as follows.
	For all $(t, \nuh) \in [0,T] \x \Pc(\Omh)$, let $ \nu := \nuh \circ \Xh^{-1}$. 
	Then on the probability space $(\Omb, \Fcb_T, \Qb^{\eps}_{t, \nu, M})$, we consider a $\Fcb_t$--measurable random element $(A'_s, W'_s, B'_s)_{s \in [0,t]}$ such that 
	\[
		\Qb^{\eps}_{t, \nu, M} \circ (X_{t \wedge \cdot} , A'_{t \wedge \cdot}, W'_{t \wedge \cdot}, B'_{t \wedge \cdot} )^{-1} = \nuh(t).
	\]
	Define
	\[
		A'_s := A'_t + A_s - A_t, ~ W'_s := W'_t + W^t_s,~ B'_s := B'_t + B^t_s,~\mbox{for}~s \in [t,T],
	\]
	\[
		\muh'_s := \Lc^{\Qb^{\eps}_{t, \nu, M}}( X_{s \wedge}, A'_{s \wedge}, W', B'_{s \wedge}) \mathbf{1}_{\{s \in [0,t]\}} + \Lc^{\Qb^{\eps}_{t, \nu, M}}( X_{s \wedge}, A'_{s \wedge}, W', B'_{s \wedge} | \Gcb^t_T) \mathbf{1}_{\{s \in (t,T]\}}.
	\]
	Let
	\[
		\Qb^{\eps}_{t, \nuh, M} := \Qb^{\eps}_{t, \nu, M} \circ  \big( X, A', W', B', \muh' \big),
		~\mbox{so that}~ J(t, \Qb^{\eps}_{t, \nuh, M} ) = J(t, \Qb^{\eps}_{t, \nu, M})
		~\mbox{and hence satisfies \eqref{ineq:eps-optimal}}.
	\]
	Let $\Pb \in \Pcb^{\B}_S(t, \nu)$,
	as in \Cref{lemm:concatenation}, for $M > 0$, we let 
	\[
		\widehat M(\omb) :=  \E\bigg[ \|X \|^p + \int_{\taub}^T \big( \rho(\alphab_s, u_0) \big)^p \mathrm{d}s  \bigg| \Gcb^t_{\taub} \bigg](\omb),
		~\mbox{\rm and}~
		\Qb^{\eps}_{\omb} := \Qb^{\eps}_{\taub(\omb), \muh_{\taub(\omb)}(\omb), \widehat M(\omb) +M}.
	\]
	Again, as in the proof of \Cref{lemm:concatenation}, one has $\Qb^{\eps}_{\omb}$ satisfies \eqref{eq:loi_Xmu} and \eqref{fix_value-B},
	which allows defining $\Pb^{M, \eps}$ by
	\[
		\Pb^{M, \eps} [K] \in \int_{\Omb} \Qb_{\omb}[K] \Pb(d \omb),\; \mbox{for all}\; K \in \Fcb,
	\]
	so that $\Pb^{M, \eps} \in \Pcb_W(t, \nu)$, $\Pb^{M, \eps} = \Pb$ on $\Fcb_{\taub}$ and $(\Qb^{\eps}_{\omb})_{\omb \in \Omb}$ is a family of r.c.p.d. of $\Pb^{M, \eps}$ knowing $\Gcb^t_{\taub}$.
	
	\vspace{0.5em}

	Finally, it is enough to prove that $\Pb^{M, \eps} \in \Pcb_S^{\B}(t,\nu)$.
	Let $s \in [t, T]$, $h  \in C_b(\R)$ and $ \psi \in C_b(\Cc^{\ell})$, then
	\begin{align*}
		\E^{\Pb^{M, \eps}} \big[A_s h(A_s) \psi \big(B^t_{s \wedge \cdot} \big) \mathbf{1}_{s > \taub}\big]
		&=
		\E^{\Pb^{M, \eps}} \big[ \E^{\Qb^\eps_\cdot} \big[A_{s} h(A_s) \psi \big(B^t_{s \wedge \cdot} \big) \big]  \mathbf{1}_{s > \taub}\big] \\
		&=
		\E^{\Pb^{M, \eps}} \Big[ \E^{\Qb^\eps_\cdot} \Big[ \E^{\Qb^\eps_\cdot} \big[A_{s} \big| B^t_{s \wedge \cdot} \big] h(A_s) \psi \big(B^t_{s \wedge \cdot} \big) \Big] \mathbf{1}_{s > \taub}\Big]
		\\
		&=
		\E^{\Pb^{M, \eps}} \Big[ \E^{\Pb^{M, \eps}} \Big[ \E^{\Pb^{M, \eps}} \big[A_{s} \big| \Gcb^t_{\taub} \vee \sigma(B^t_{s \wedge \cdot} ) \big] h(A_s) \psi \big(B^t_{s \wedge \cdot} \big) \Big| \Gcb^t_{\taub} \Big] \mathbf{1}_{s > \taub}\Big] \\
		&=
		\E^{\Pb^{M, \eps}} \big[  \E^{\Pb^{M, \eps}} \big[A_{s} \big| B^t_{s \wedge \cdot} \big] h(A_s) \psi \big(B^t_{s \wedge \cdot} \big) \mathbf{1}_{s > \taub}\big],
	\end{align*}
	where the second equality follows by the fact that $\Qb^{\eps}_{\omb} \in \Pcb_S^\B(\taub(\omb), \nu')$ for some $\nu' \in \Pc(\Cc^n)$ and hence
	$h(A_s) = h( \phi(B^{\taub(\omb)}_{s \wedge \cdot}))$, $\Qb^{\eps}_{\omb}$--a.s., for some Borel measurable function $\phi$,
	and the last quality follows by the fact that on $\{s > \taub\},$ $\E^{\Pb^{M, \eps}} \big[A_{s} \big| \Gcb^t_{\taub} \vee \sigma(B^t_{s \wedge \cdot} ) \big]=\E^{\Pb^{M, \eps}} \big[A_{s} \big| B^t_{s \wedge \cdot} \big]$.
	Further, as $\Pb^{M, \eps}|_{\Fcb_{\taub}} = \Pb|_{\Fcb_{\taub}}$ and $\Pb \in \Pcb_S^{\B}(t,\nu)$, one can use similarly argument to find that
	\[
		\E^{\Pb^{M, \eps}} \big[A_s h(A_s) \psi \big(B^t_{s \wedge \cdot} \big) \mathbf{1}_{s \le \taub}\big]
		=
		\E^{\Pb^{M, \eps}} \big[  \E^{\Pb^{M, \eps}} \big[A_{s} \big| B^t_{s \wedge \cdot} \big] h(A_s) \psi \big(B^t_{s \wedge \cdot} \big) \mathbf{1}_{s \le\taub}\big].
	\]
	This implies that
	\[
		\E^{\Pb^{M, \eps}} \Big[  \Big( A_s - \E^{\Pb^{M, \eps}} \big[A_{s} \big| B^t_{s \wedge \cdot} \big] \Big) h(A_s) \psi \big(B^t_{s \wedge \cdot} \big) \Big]
		=0,
		~\mbox{\rm and hence}~
		A_s 
		=
		\E^{\Pb^{M, \eps}} \big[A_s \big| B^t_{s \wedge \cdot} \big],
		~\Pb^{M, \eps}\mbox{--a.s.}
	\]
	In other words, $A$ is a continuous process, adapted to the $\Pb^{M, \eps}$--augmented filtration generated by $B^t$,
	then there exists a Borel  measurable function $\hat \phi: [t,T] \x \Cc^{\ell} \longrightarrow U$ such that $A_s = \hat \phi(s, B^t_{s \wedge \cdot})$, for all $s \in [t,T]$, $\Pb^{M, \eps}$--a.s.,
	and hence $\Pb^{M, \eps} \in \Pcb_S^{\B}(t,\nu)$,
	which concludes the proof.
\end{proof}

\paragraph{Proof of Theorem \ref{thm:B-StrongDPP}}     

	The proof is almost the same as that of Theorem \ref{thm:WeakDPP}.
	First, one has the measurability of $V_S^\B$ by \Cref{lemm:measurability_B-Strong}.
	Next, notice that a $\G^{t, \circ}$--stopping time  $\tau$ on $\Om^t$ can be considered as a special $\Gb^t$--stopping time $\taub$ on $\Omb$.
	then using the conditioning argument in \Cref{lemm:conditioning-concatenation_StrongFormulation}, it follows that 
	\begin{align*}
	    V^{\B}_S(t, \nu) 
		\le 
		\sup_{\Pb \in \Pcb_S^\B(t, \nu)} 
		\E \bigg[
			\int_t^{\tau} L(s, X_{s \wedge \cdot} , \mub_s, \alphab_s) \mathrm{d}s
			+ 
			V_S^\B\big(\tau, \mu_{\tau} \big) 
		\bigg].
	\end{align*}
    	Finally, it is enough to use the concatenation argument in \Cref{lemm:conditioning-concatenation_StrongFormulation} and sending $M \to \infty$ to obtain the reverse inequality
	\begin{align*}
		V_S^\B(t,\nu)
		\ge
		\sup_{\Pb \in \Pcb_S^\B(t,\nu)}\E^\Pb \bigg[
			\int_t^{\tau} L(s, X_{s \wedge \cdot} ,\mub_s,\alphab_s) \mathrm{d}s
			+ 
			V_S^\B\big(\tau, \mu_{\tau} \big) 
		\bigg].
	\end{align*}
	\qed

\begin{appendix}
\section{Some technical results on controlled McKean--Vlasov SDEs}

	Let us first recall a technical optional projection result.

\begin{lemma} \label{lemm:Cond_Law}
    
	Let $E$ be a Polish space, $(\Om, \Fc, \P)$ be a complete probability space, equipped with a complete filtration $\G:=(\Gc_t)_{t\ge 0}$.

\medskip
	\noindent $(i)$
	Given an $E$--valued measurable process $(X_t)_{t \in [0,T]}$,
	there exists a $\Pc(E)$--valued $\G$--optional process $\beta$ such that
	\[
		\beta_{\tau} =\Lc^\P \big(X_{\tau} \big|\Gc_{\tau}\big),\; \P\mbox{\rm --a.s.},\;\mbox{\rm for all}\; \G \mbox{\rm--stopping times}\; \tau.
	\]

	\noindent $(ii)$
	Assume in addition that $X$ is a continuous process,
	and that the $\G$--optional $\sigma$--field is identical to the $\G$--predictable $\sigma$--field.
	Then one can choose $\beta$ to be an a.s. continuous process.
\end{lemma}
\begin{proof}
	$(i)$ The existence of such process $\beta$ is ensured by, e.g. \citeauthor*{kurtz1998martingale} \cite[Theorem A.3]{kurtz1998martingale} or \citeauthor*{yor1977sur} \cite[Proposition 1]{yor1977sur}.

	\vspace{0.5em}

	$(ii)$ When $X$ is a continuous process, it follows again by  \cite[Theorem A.3]{kurtz1998martingale} (or \cite[Proposition 1]{yor1977sur}) that $\beta$ is  c\`adl\`ag $\P$--a.s.
	Further, let $\varphi \in C_b(E)$ and $(\tau_n)_{n \ge 1}$ be a increasing sequence of uniformly bounded $\G$--stopping times\footnote{which is $\G$--predictable time as soon as the $\G$--optional $\sigma$--field is identical to the $\G$--predictable $\sigma$--field.}.
	One has $\langle \varphi, {\beta}_{\tau_n} \rangle=\E^{\P}[\varphi(X_{\tau_n})| \Gc_{\tau_n}]$, $\P$--a.s.,
	and hence $\lim_{n\rightarrow\infty} \E^{\P}[\langle \varphi, {\beta}_{\tau_n} \rangle]=\E^{\P}[\langle \varphi, {\beta}_{\lim_n \tau_n} \rangle]$. 
	 Then it follows by \citeauthor*{dellacherie1972capacites} \cite[Theorem IV--T24]{dellacherie1972capacites} that $(\langle \varphi, {\beta}_{t} \rangle)_{t \in [0,T]}$ is left--continuous, $\P$--a.s. By considering a countable dense family of functions $\varphi$ in $C_b(E)$, one concludes that $\beta$ is also left--continuous a.s.
\end{proof}

	\vspace{0.5em}

	Let $(\Om, \Fc,\P)$ be a complete probability space,
	$\F = (\Fc_s)_{s \ge 0}$ a complete filtration, 
	supporting two independent $\F$--Brownian motions $B^{\star}$ and $W^{\star}$, which are respectively $\R^d$-- and $\R^\ell$--valued.
	Let us fix a $\R^n$--valued, $\F$--adapted continuous process $(\xi_s)_{s \ge 0}$,
	a $U$--valued $\F$--predictable process $(\alpha_s)_{s \ge 0}$, a complete sub--filtration $\G = (\Gc_s)_{s \ge 0}$ of $\F$,
	and denote $W^{*,t}_s := W^{\star}_{s \vee t} - W^{\star}_t$ and $B^{*,t}_s := B^{\star}_{s \vee t} - B^{\star}_t$.
	We will study the following SDE with data $(t,\xi, \alpha, \G)$:
	$X_s=\xi_s$ for all $s \in [0,t]$, and with $\mub_r := \Lc^\P(X_{r \wedge \cdot},\alpha_{r}|\Gc_r)$,
	\begin{equation} \label{eq:General_SDE-McKeanVlasov}
		X_s
		= 
		\xi_t 
		+
		\int_t^s b \big(r, X_{r \wedge \cdot}, \mub_r, \alpha_r \big) \mathrm{d}r
		+
		\int_t^s \sigma\big(r, X_{r \wedge \cdot}, \mub_r , \alpha_r\big) \mathrm{d} W^{\star}_r
		+ 
		\int_t^s \sigma_0 \big(r, X_{r \wedge \cdot}, \mub_r, \alpha_r \big) \mathrm{d}B^{\star}_r,\; 
		\mbox{for all}\; s \ge t,\; \P\mbox{--a.s.}
	\end{equation}

	\begin{definition} \label{def:strong_sol}
		A strong solution of {\rm SDE} {\rm\eqref{eq:General_SDE-McKeanVlasov}}, with  data $(t,\xi, \alpha, \G)$, on $[0,T]$,
		is an $\R^n$--valued $\F$--adapted continuous process $X = (X_t)_{t \ge 0}$ such that
		$\E \big[  \sup_{s \in [0,T]} |X_s|^2 \big] < \infty$ and \eqref{eq:General_SDE-McKeanVlasov} holds true. 
	\end{definition}

\begin{theorem} \label{theorem_Existence/uniqueness-SDE}
	Let $0 \le t \le T$, {\rm\Cref{assum:Lip}} hold true, 
	$\E \big[\sup_{s \in [0,t]} |\xi_s|^p \big] < \infty$ and
	$\E \big[\int_t^T \rho(u_0,\alpha_s)^p \mathrm{d} s \big]< \infty$ for some $p \ge 2$. Then

	\medskip	
	
	$(i)$ there exists a unique strong solution $X^{t,\xi,\alpha}$ of \eqref{eq:General_SDE-McKeanVlasov} on $[0,T]$ with data $(t,\xi,\alpha,\G)$.
	Moreover, it holds that $\E \big[ \sup_{s \in [0,T]} \big| X^{t,\xi,\alpha}_s \big|^p \big] < \infty;$

	\medskip

	$(ii)$
	assume in addition that
	$(\xi_{t \wedge \cdot},W^{\star},B^{\star}_{t \wedge \cdot})$ is independent of $\G$, and $B^t$ is $\G$--adapted, 
	and there exists a Borel measurable function $\phi: [0,T] \x \Cc^n \x \Cc^d \x \Cc^{\ell} \longrightarrow U$ such that
	\[
		\alpha_s
		=
		\phi\big(s,\xi_{t \wedge \cdot},W^{\star,t}_{s \wedge \cdot},B^{\star,t}_{s \wedge \cdot}\big),\;\P\mbox{\rm--a.s.},\;\mbox{\rm for all}\;s \in [0,T].
	\]
	Then, with $A_s:=\int_t^{s \vee t} \pi(\alpha_r) \mathrm{d}r$,
	there exists a continuous process $(\muh_t)_{t \in [0,T]}$ such that for all $s \in [0,T]$ 
	\[
	    \muh_s
	    =
	    \Lc^\P \big((X^{t,\xi,\alpha}_{s \wedge \cdot},A_{s \wedge \cdot}, W^{\star}, B^{\star}_{s \wedge \cdot}) \big|\Gc_T \big)
	    =
	    \Lc^\P \big((X^{t,\xi,\alpha}_{s \wedge \cdot},A_{s \wedge \cdot}, W^{\star}, B^{\star}_{s \wedge \cdot}) \big|\Gc_s \big)
	    =
	    \Lc^\P \big((X^{t,\xi,\alpha}_{s \wedge \cdot},A_{s \wedge \cdot}, W^{\star}, B^{\star}_{s \wedge \cdot}) \big|B^{*,t}_{s \wedge \cdot} \big)
	   ,\;\P\mbox{\rm --a.s.}
	\]
\end{theorem}
\begin{proof}
	$(i)$ We follow \cite[Theorem 5.1.1]{stroock2007multidimensional} to prove the existence and uniqueness of a strong solution to \eqref{eq:General_SDE-McKeanVlasov}.
	Let $\Sc^p$ be defined by
	\[
		\Sc^p
		:=
		\big\{
			Y := (Y_s)_{s \in [0,T]}: \R^n \mbox{--valued and}\; \F\mbox{--adapted and continuous process such that}\; 
		\E^\P \big[
		\|Y\|^p_T
		\big]
		<
		\infty 
		\big \},
	\]
	where $\| \xb \|_s:= \sup_{r \in [0,s]} |\xb_r|$ for $s \in [0,T]$ and $\xb \in \Cc^n$. 
	For all $Y \in \Sc^p$, we define, with $\mub_s := \Lc^\P \big(Y_{s \wedge \cdot},\alpha_{s} \big| \Gc_T\big)$, $\Psi(Y) := (\Psi(Y)_s)_{0 \le s \le T}$ by
	\[
		\Psi(Y)_s
		:=
		\xi_{t \wedge s}
		+
		\int_t^{t \vee s} b \big(r, Y_{r \wedge \cdot}, \mub_r, \alpha_r \big) \mathrm{d}r
		+
		\int_t^{t \vee s} \sigma\big(r, Y_{r \wedge \cdot}, \mub_r, \alpha_r \big) \mathrm{d} W^{\star}_r
		+ 
		\int_t^{t \vee s} \sigma_0 \big(r, Y_{r \wedge \cdot}, \mub_r, \alpha_r \big) \mathrm{d}B^{\star}_r.
	\]
	Then, for $(Y^1, Y^2) \in \Sc^p\times\Sc^p$, with $\mub^i_s := \Lc^\P \big(Y^i_{s \wedge \cdot},\alpha_{s} \big| \Gc_T\big)$, $i\in\{1,2\}$, one has
	\begin{align*}
		\E^\P \Big[
		\big\| \Psi(Y^1)-\Psi(Y^2) \big\|_s^p
		\Big]
		\le&\
		3^p \E^\P \bigg[
		\sup_{v \in [t,s]}
			\bigg|\int_t^v \big( \sigma\big(r, Y^\mathbf{1}_{r \wedge \cdot}, \mub^\mathbf{1}_r, \alpha_r \big)
			-
			\sigma\big(r, Y^2_{r \wedge \cdot}, \mub^2_r, \alpha_r \big) \big) \mathrm{d}W^{\star}_r \bigg|^p
		\bigg]
		\\
		&+
		3^p \E^\P \bigg[
		\sup_{v \in [t,s]}
		\bigg|\int_t^v \big(\sigma_0\big(r, Y^\mathbf{1}_{r \wedge \cdot}, \mub^\mathbf{1}_r, \alpha_r \big)
		-
		\sigma_0\big(r, Y^2_{r \wedge \cdot}, \mub^2_r, \alpha_r \big) \big)\mathrm{d}B^{\star}_r \bigg|^p
		\bigg]
		\\
		&+
		3^p \E^\P \bigg[ \bigg|
		\int_t^s \big| b\big(r, Y^\mathbf{1}_{r \wedge \cdot}, \mub^\mathbf{1}_r, \alpha_r \big)
		-
		b\big(r, Y^2_{r \wedge \cdot}, \mub^2_r, \alpha_r \big)\big|\mathrm{d}r \bigg|^p
		\bigg].
	\end{align*}
	Notice that, for all $r \in [t,T]$,
	\[
		\Wc_2 \big(\mub^\mathbf{1}_r,\mub^2_r \big)^p
		\le 
		\Wc_p \big(\mub^\mathbf{1}_r,\mub^2_r \big)^p
		=
		\Wc_p \Big(\Lc^\P((Y^\mathbf{1}_{r \wedge \cdot},\alpha_r)|\Gc_r),\Lc^\P((Y^2_{r \wedge \cdot},\alpha_r)|\Gc_r) \Big)^p
		\le
		\E^\P \Big[ \| Y^\mathbf{1}_{r \wedge \cdot}- Y^2_{r \wedge \cdot} \|^p \Big| \Gc_r\Big].
	\]
	Then by Burkholder--Davis--Gundy inequality, Jensen's inequality and \Cref{assum:Lip}, there is some constant $C_T>0$ such that
	\begin{equation} \label{eq:Psi_estim}
		\E^\P \Big[
		\|\Psi(Y^1)-\Psi(Y^2)\|_s^p
		\Big]
		\le
		C_T
		\int_t^s
		\E^\P \Big[
		\|Y^\mathbf{1}_{r \wedge \cdot}-Y^2_{r \wedge \cdot}\|_r^p
		\Big] \mathrm{d}r.
	\end{equation}
	Besides, by \Cref{assum:Lip}
	\[
		\E \big[ \|\Psi(0) \|^p \big]
		\le
		C \bigg( 1 + \E^\P \bigg[ \sup_{r \in [0,t]} |\xi_r|^p \bigg] + \E^\P \bigg[ \int_t^T \rho(u_0,\alpha_r)^p \mathrm{d}r \bigg]\bigg).
	\]
	Then by taking $Y^2=0$, \eqref{eq:Psi_estim} implies that $\Psi(Y) \in \Sc^p$ whenever $Y \in \Sc^p$.
	Moreover, for any positive integer $n$ 
	\begin{align*}
        \E^\P \Big[
        \|\Psi^n(Y^1)-\Psi^n(Y^2)\|_s^p
        \Big]
        &\le
        C_T
        \int_t^s
        \E^\P \Big[
        \|\Psi^{n-1}(Y^1)-\Psi^{n-1}(Y^2)\|_r^p
        \Big] \mathrm{d}r
        \\
        &\le
        (C_T)^2
        \int_t^s
        \int_t^r
        \E^\P \Big[
        \|\Psi^{n-2}(Y^1)-\Psi^{n-2}(Y^2)\|_v^p
        \Big] \mathrm{d}v \mathrm{d}r
        \\
        &\le
        (C_T)^n
        \int
        \mathbf{1}_{\{s \ge v_1 \ge v_2 \ge ... \ge v_n \ge t\}}
        \E^\P \Big[
        \|Y^1-Y^2\|^p_{v_n}
        \Big]
        \mathrm{d}v_1\dots \mathrm{d}v_n
        \\
        &\le
        (C_T)^n
        \E^\P \Big[
        \|Y^1-Y^2\|^p_{s}
        \Big]
        \frac{(s-t)^n}{n!}.
	\end{align*}
	Let $Y \in \Sc^p$, $X^0:=Y$, and $X^n:=\Psi^n(Y)$, for $n \ge 1$, 
	it follows that
	\[
		\E^\P \big[
		\|X^n-X^{n+1}\|_s^p
		\Big]
		\le
		(C_T)^n
		\E^\P \Big[
		\|Y-\Psi(Y)\|^p_{s}
		\Big]
		\frac{(s-t)^n}{n!},
		~\mbox{\rm and hence}~
		\E^\P \Big[
			\sum_{n\geq 1}
			\|X^n-X^{n+1}\|_T^p
		\Big] < \infty,
	\]
	which implies that the sequence $(X^n)_{n\geq 1}$ converges uniformly, $\P$--a.s., to some $X \in \Sc^p$. 
	Finally, it is straightforward to see that $X$ is the unique strong solution of \eqref{eq:General_SDE-McKeanVlasov} with data $(t,\xi,\alpha,\G)$. 
    


	\medskip

	$(ii)$ 
	Let $\nu := \P \circ (\xi_{t \wedge \cdot})^{-1}$.
	We recall that the canonical space $\Om^t:= \Cc^n_{0,t} \x \Cc^{d}_{t,T} \x \Cc^{\ell}_{t,T}$ was introduced in Section \ref{subsubsec:strong_form},
	with corresponding canonical processes $\zeta = (\zeta_s)_{0 \le s \le t}$, $W = (W_s)_{t \le s \le T}$, and $ B = (B_s)_{t \le s \le T}$.
	and filtration $\F^{t,\circ}$, $\G^{t,\circ}$.
	Moreover, under $\P^t_{\nu}$, $W^t_s := W_{s \vee t} - W_t$ and $B^t_s := B_{s \vee t} - B_t$ are standard Brownian motion on $[t,T]$ independent of $\zeta$,
	and  $\F^t$ and $\G^t$ are the $\P^t_{\nu}$--augmented filtration of $\F^{t,\circ}$ and $\G^{t,\circ} $.

	\vspace{0.5em}
	
	Define $\widetilde{\alpha}_s:=\phi(s,\zeta_{t \wedge \cdot}, W^t_{s \wedge \cdot}, B^t_{s \wedge \cdot}),$ for all $s \in [t,T]$. 
	There exists a unique solution ${Y}^\alpha$ of \Cref{eq:General_SDE-McKeanVlasov} on $\Om^t$, 
	associated with $(t, \zeta_{t \wedge \cdot},\widetilde{\alpha},\G^t)$. 
	As ${Y}^\alpha$ is an $\F^t$--adapted continuous process, 
	there exists an Borel measurable function $\Psi: [0,T] \x \Cc^n \x \Cc^d \x \Cc^{\ell} \to \R^n$ such that
	\[
	    {Y}^\alpha_s
	    =
	    \Psi_s (\zeta_{t \wedge \cdot}, W^t_{s \wedge \cdot}, B^t_{s \wedge \cdot}),~ s \in [0,T],~\P^t_{\nu}\mbox{--a.s.}
	\]
	Next, on the probability space $(\Om,\Fc,\P)$, let us define
	\[
	    X^\alpha_s:= \Psi_s (\xi_{t \wedge \cdot}, W^{*,t}_{s \wedge \cdot},B^{*,t}_{s \wedge \cdot}).
	\]
	Then it is clear that $X^\alpha$ is the unique solution, on $(\Om,\F,\Fc,\P)$, of \Cref{eq:General_SDE-McKeanVlasov}, associated with $(t,\xi_{t \wedge \cdot},\alpha,\F^{B^{\star,t}})$, where $\F^{B^{\star,t}}:=(\Fc^{B^{\star,t}}_s)_{s \in [0,T]}$ is the $\P$--augmented filtration generated by $B^{\star,t}$. 
	Moreover, as $(\xi_{t \wedge \cdot},W^{\star}, B^{\star}_{t \wedge \cdot})$ is independent of $\G$, and $B^{\star,t}$ is $\G$--adapted, one has, for all $s \in [t,T]$,
	\begin{align*}
		\mub_s
		=
		\Lc^\P \big( (X^\alpha_{s \wedge \cdot}, \alpha_s) \big| B^{\star,t}_{s \wedge \cdot} \big)
		&=
		\Lc^\P \big( (\Psi_{s \wedge \cdot} (\xi_{t \wedge \cdot}, W^{\star,t}_{s \wedge \cdot}, B^{\star,t}_{s \wedge \cdot}), \phi(s,\xi_{t \wedge \cdot},W^{\star,t}_{s \wedge \cdot},B^{*,t}_{s \wedge \cdot})) \big| B^{*,t}_{s \wedge \cdot} \big)
		\\
		&=
		\Lc^\P \big( (\Psi_{s \wedge \cdot} (\xi_{t \wedge \cdot}, W^{\star,t}_{s \wedge \cdot}, B^{\star,t}_{s \wedge \cdot}), \phi(s,\xi_{t \wedge \cdot},W^{\star,t}_{s \wedge \cdot}, B^{\star,t}_{s \wedge \cdot}) )\big| \Gc_s \big)
		\\
		&=
		\Lc^\P \big( (X^\alpha_{s \wedge \cdot}, \alpha_s )\big| \Gc_s \big), ~\P\mbox{\rm--a.s.}
	\end{align*}
	This implies that $X^\alpha$ is also a solution of \Cref{eq:General_SDE-McKeanVlasov} associated with $(t,\xi_{t \wedge \cdot},\alpha,\G)$,
	and hence $X^{t,\xi,\alpha} = X^\alpha$ by uniqueness of solution to \Cref{eq:General_SDE-McKeanVlasov}.
    
	\medskip
	Further, as $A_s$ is a Borel measurable function of $(\xi_{t \wedge \cdot},W^{\star, t}_{s \wedge \cdot},B^{\star,t}_{s \wedge \cdot})$, 
	it follows by the same argument that, for all $s \in [0,T]$,
	\[
		\muh_s 
		:=
		\Lc^\P \big(X^{t,\xi,\alpha}_{s \wedge \cdot},A_{s \wedge \cdot}, W^{\star}, B^{\star}_{s \wedge \cdot} \big|\Gc_T \big)
	    =
	    \Lc^\P \big(X^{t,\xi,\alpha}_{s \wedge \cdot},A_{s \wedge \cdot}, W^{\star}, B^{\star}_{s \wedge \cdot} \big|\Gc_s \big)
	    =
	    \Lc^\P \big(X^{t,\xi,\alpha}_{s \wedge \cdot},A_{s \wedge \cdot}, W^{\star}, B^{\star}_{s \wedge \cdot} \big|B^{\star,t}_{s \wedge \cdot} \big)
	   ,\;\P-\mbox{a.s.}
	\]
	Finally, using \Cref{lemm:Cond_Law}, one can choose the process $\muh$ to be continuous.
	\end{proof}

	Let us consider the following system of {\rm SDE}, 
	where a solution is a couple of $\F$--adapted continuous processes $(X, \muh)$ such that: for some $\nuh_0 \in \Pc(\Cc^n \x \Cc \x \Cc^d \x \Cc^\ell),$ $\E^{\P} \big[\|X\|^2 + \Wc_2(\muh,\nuh_0)^2 \big]< \infty,$
	$X_s=\xi_s$ for $s \in [0,t]$, and
	\begin{equation} \label{eq:double_SDE-McKeanVlasov}
		X_s
		= 
		\xi_t 
		+
		\int_t^s b \big(r, X_{r \wedge \cdot}, \mub_r, \alpha_r \big) \mathrm{d}r
		+
		\int_t^s \sigma\big(r, X_{r \wedge \cdot}, \mub_r, \alpha_r \big) \mathrm{d} W^{\star}_r
		+ 
		\int_t^s \sigma_0 \big(r, X_{r \wedge \cdot}, \mub_r, \alpha_r \big) \mathrm{d}B^{\star}_r,\;  s \in [t,T],\; \P\mbox{\rm --a.s.},
	\end{equation}
	and $\mub_r := \Lc^\P(X_{r \wedge \cdot},\alpha_{r}|B^{*,t}_{r \wedge \cdot}, \muh_{r \wedge \cdot})$ for all $r \ge t$,
	and with $A_s:=\int_t^{s \wedge t} \pi(\alpha_r) \mathrm{d}r$,
	$\muh_s=\Lc^\P \big( X_{s \wedge \cdot},A_{s \wedge \cdot}, W^{\star}, B^{\star}_{s \wedge \cdot} \big| B^{t,*}_{s \wedge \cdot},\muh_{s \wedge \cdot} \big)$ for all $s \ge 0$, and finally $(\xi_{t \wedge \cdot},W^{\star},B^{\star}_{t \wedge \cdot})$ is independent of $(B^{*,t},\muh)$.

	\begin{corollary} \label{corollary_doubleSDE}
		Let {\rm\Cref{assum:Lip}} hold true, 
		and assume that there exists a Borel measurable function $\phi: [0,T] \x \Cc^n \x \Cc^d \x \Cc^\ell \longrightarrow U$ such that 
		\[
			\alpha_s=\phi \big(s,\xi_{t \wedge \cdot},W^{\star,t}_{s \wedge \cdot},B^{\star,t}_{s \wedge \cdot} \big),\; \text{\rm for}\; \mathrm{d}\P \otimes \mathrm{d}t\; \mbox{\rm--a.e.} \; (s,\om) \in [t,T] \x \Om,
			~\mbox{and}~
			\E \bigg[\int_t^T \rho(u_0,\alpha_s)^2 \mathrm{d} s \bigg]< \infty.
		\]
		Then, {\rm\Cref{eq:double_SDE-McKeanVlasov}} has a unique solution  $(X,\muh)$,
		where $X$ is the strong solution of {\rm\Cref{eq:General_SDE-McKeanVlasov}} with data $(t,\xi, \alpha, \F^{B^{\star,t}})$,
		with $\F^{B^{\star,t}}$ being the $\P$--augmented filtration generated by $B^{\star,t}$ and
		\[
			\muh_s
			=
			\Lc^\P \big( X_{s \wedge \cdot}, A_{s \wedge \cdot}, W^{\star}, B^{\star}_{s \wedge \cdot} \big | B^{\star,t}_{s \wedge \cdot} \big),\; s\in[t,T],\; \P\mbox{\rm --a.s.}
		\]
	\end{corollary}
	\begin{proof}
		Given a solution $(X, \muh)$ to {\rm\Cref{eq:double_SDE-McKeanVlasov}},
		we notice that $X$ is a strong solution of \Cref{eq:General_SDE-McKeanVlasov} associated with data $(t,\alpha,\xi,\F^{B^{\star,t},\muh})$, where $\F^{B^{\star,t},\muh}:=(\Fc^{B^{\star,t},\muh}_s)_{s \in [0,T]}$ with $\Fc^{B^{\star,t},\muh}_s:= \sigma (B^{\star,t}_{s \wedge \cdot},\muh_{s \wedge \cdot})$. 
		As $(\xi_{t \wedge \cdot},W^{\star},B^{\star}_{t \wedge \cdot})$ is independent of $(B^{\star,t},\muh)$,
		it is then enough to apply \Cref{theorem_Existence/uniqueness-SDE} to conclude that $X$ is the strong solution of {\rm\Cref{eq:General_SDE-McKeanVlasov}} with data $(t,\xi, \alpha, \F^{B^{\star,t}})$.
	\end{proof}

\end{appendix}

\bibliography{bibliographyDylan}
\end{document}